\documentclass[a4paper,12pt,reqno]{amsart}
\usepackage[T1]{fontenc}
\usepackage[utf8]{inputenc}
\usepackage{bera}
\usepackage{amssymb,dsfont,mathrsfs}
\usepackage[margin=1in]{geometry}
\usepackage{enumitem}
\usepackage{graphicx,color,tikz}
\usepackage{xparse}
\usepackage[bookmarksdepth=2]{hyperref}

\usetikzlibrary{quotes,angles,decorations.markings}

\numberwithin{equation}{section}

\newtheorem{theorem}{Theorem}[section]
\newtheorem{lemma}[theorem]{Lemma}
\newtheorem{corollary}[theorem]{Corollary}
\newtheorem{proposition}[theorem]{Proposition}

\theoremstyle{definition}
\newtheorem{definition}[theorem]{Definition}

\newtheorem{remark}[theorem]{Remark}

\DeclareMathOperator{\re}{Re}
\DeclareMathOperator{\im}{Im}
\DeclareMathOperator{\Arg}{Arg}

\DeclareMathOperator{\sign}{sign}
\DeclareMathOperator{\dist}{dist}
\DeclareMathOperator{\Int}{Int}
\DeclareMathOperator{\Cl}{Cl}

\DeclareMathOperator{\esssupp}{ess\,supp}
\DeclareMathOperator{\ind}{\mathds{1}}

\newcommand{\hp}{\mathds{H}}
\newcommand{\pvint}{\operatorname{pv}\!\int}
\newcommand{\laplace}{\mathcal{L}}
\newcommand{\sub}{\subseteq}
\newcommand{\pr}{\mathds{P}}
\newcommand{\ex}{\mathds{E}}
\newcommand{\C}{\mathds{C}}
\newcommand{\R}{\mathds{R}}
\newcommand{\ph}{\varphi}
\newcommand{\eps}{\varepsilon}
\newcommand{\thet}{\vartheta}
\newcommand{\ro}{\varrho}
\newcommand{\exprv}{\mathbf{e}}

\renewcommand{\le}{\leqslant}
\renewcommand{\ge}{\geqslant}
\newcommand{\ul}{\underline}
\newcommand{\ol}{\overline}

\NewDocumentCommand{\formula}{ssom}{%
 \IfBooleanTF{#1}{%
  \IfBooleanTF{#2}{%
   \IfValueTF{#3}%
    {\begin{align}\label{#3}\begin{gathered}#4\end{gathered}\end{align}}%
    {\begin{gather}#4\end{gather}}%
  }{%
   \IfValueTF{#3}%
    {\begin{align}\label{#3}\begin{aligned}#4\end{aligned}\end{align}}%
    {\begin{gather*}#4\end{gather*}}%
  }%
 }{%
  \IfValueTF{#3}%
   {\begin{align}\label{#3}#4\end{align}}%
   {\begin{align*}#4\end{align*}}%
 }%
}

\newcommand{\ignore}[1]{}

\begin{document}

\title[Suprema of Lévy processes with completely monotone jumps]{Suprema of Lévy processes with completely monotone jumps: spectral-theoretic approach}
\author{Mateusz Kwaśnicki}
\thanks{Work supported by the Polish National Science Centre (NCN) grants no.\@ 2015/19/B/ST1/01457 and 2019/33/B/ST1/03098}
\address{Mateusz Kwaśnicki \\ Department of Pure Mathematics \\ Wrocław University of Science and Technology \\ ul. Wybrzeże Wyspiańskiego 27 \\ 50-370 Wrocław, Poland}
\email{mateusz.kwasnicki@pwr.edu.pl}
\date{\today}

\begin{abstract}
We study spectral-theoretic properties of non-self-adjoint operators arising in the study of one-dimensional Lévy processes with completely monotone jumps with a one-sided barrier. With no further assumptions, we provide an integral expression for the bivariate Laplace transform of the transition density $p_t^+(x, y)$ of the killed process in $(0, \infty)$, and under a minor regularity condition, a generalised eigenfunction expansion is given for the corresponding transition operator $P_t^+$. Assuming additionally appropriate growth of the characteristic exponent, we prove a generalised eigenfunction expansion of the transition density $p_t^+(x, y)$. Under similar conditions, we additionally show integral formulae for the cumulative distribution functions of the infimum and supremum functionals $\underline{X}_t$ and $\overline{X}_t$. The class of processes covered by our results include many stable and stable-like Lévy processes, as well as many processes with Brownian components. Our results recover known expressions for the classical risk process, and provide similar integral formulae for some other simple examples of Lévy processes.
\end{abstract}

\maketitle

%
%

\section{Introduction}
\label{sec:intro}

\subsection{Motivation}

Spectral theory proved to be an important tool in the study of Markov processes. In the simplest discrete setting, it is one of the standard tools in the theory of Markov chains. For a brief introduction to the case of discrete time processes in continuous state space, and a summary of related literature, we refer to Chapter~22 of~\cite{dmps}. When both time and the state space are continuous, spectral theory of Markov processes is strongly linked with functional inequalities and the analysis of transition probabilities, or heat kernels; see, for example, \cite{davies,dv,wang}. On the analysis side, this connection provides a very intuitive and powerful way to work with operators arising naturally in harmonic analysis, theory of PDEs and quantum physics; we refer to~\cite{bh,dv,doob,osekowski} for further discussion and references.

A vast number of works apply spectral theory to the study of symmetric (or self-dual) Markov processes. In this case the underlying operators are self-adjoint. When considering non-symmetric processes, however, one typically has to deal with operators that are no longer normal, let alone self-adjoint. The lack of general spectral theory of such operators makes spectral-theoretic approach to non-symmetric Markov processes extremely difficult. The only exceptions known to the author include Markov chains with discrete state space, and only a few articles dealing with Markov processes in continuous space; see~\cite{cpsv,kk:stable,mp,mucha,ogura:cb,ogura:cbi,ps:cauchy,ps:laguerre,ps:spectral,psz,pv}.

The main purpose of the present article is to develop spectral theory for a relatively wide class of non-symmetric Markov processes: Lévy processes with completely monotone jumps (defined below) killed upon leaving a half-line. Our primary motivation is to study eigenfunction expansions of the transition density (the heat kernel) and related objects, and to derive explicit, or at least semi-explicit expressions for them. This work is meant to be a `proof of concept' rather than a complete theory, so we are not aiming at the greatest generality possible; nevertheless, the class of Lévy processes covered by our results includes many examples that appear frequently in literature. We show that even though the transition operators are not normal, they allow a spectral decomposition similar to the one provided by the spectral theorem for normal operators. We also apply the same technique to study the distribution of the supremum and infimum functionals of Lévy processes with completely monotone jumps, the fundamental objects in the fluctuation theory of Lévy processes.

Our results extend the previous work~\cite{kk:stable} by Alexey Kuznetsov and the author for stable processes. Although the arguments used here and in~\cite{kk:stable} share a number of similar ideas (Wiener--Hopf factorisation, heavy use of complex-analytic tools), there are essential differences that we would like to highlight here. Scale invariance played an essential role in~\cite{kk:stable}: it allowed to work with one generalised eigenfunction and its dilations rather than with a continuous family of generalised eigenfunctions, and it also eliminated the need for complicated contours of integration. Additionally, the key objects in~\cite{kk:stable} were defined in terms of known special functions. Here we need to work with rather intricate integral expressions and bivariate functions, and a detailed analysis of the appropriate contour of integration is an essential step in our development.

The present work is closely related to the article~\cite{ps:orbit} by Pierre Patie and Rohan Sarkar, where self-similar Markov processes in $(0, \infty)$ are studied using Mellin transform methods. In this case generalised eigenfunctions are again dilations of each other; in contrast to~\cite{kk:stable}, however, they are no longer given in terms of known special functions. Instead, the authors of~\cite{ps:orbit} describe the Mellin transform of generalised eigenfunctions in terms of Bernstein-gamma functions, which are holomorphic solutions of a certain function equation. The theory of Bernstein-gamma functions was developed by Pierre Patie and Mladen Savov in~\cite{ps:bernstein}, and the functional equation was first considered by the same authors and Juan Carlos Pardo in~\cite{pps}.

Our results partially resemble the statements in~\cite{ps:orbit}, but in a different context of Lévy processes: the only processes that are covered by both works are strictly stable Lévy processes, studied already in~\cite{kk:stable}. Additionally, we use Laplace transform methods and the theory of \emph{Rogers functions} (closely related to Nevanlinna--Pick functions), while the main tools for self-similar Markov processes are the Mellin transform and Bernstein-gamma functions.

Another partially related work is the article~\cite{ps:laguerre} by Pierre Patie and Mladen Savov, which builds upon another paper~\cite{ps:cauchy} by the same authors. In these two documents the authors provide generalised eigenfunction expansion for \emph{generalised Laguerre semigroups}, which correspond to certain ergodic space-time transformations of positive self-similar Markov processes. Unlike in the present paper and in~\cite{ps:orbit}, however, in~\cite{ps:cauchy,ps:laguerre} the spectrum is discrete, and there are countably many eigenfunctions and co-eigenfunctions.

There is one more line of research that should be mentioned here, focused on continuous state branching processes with or without immigration. Spectral theory for the corresponding operators was developed by Yukio Ogura in~\cite{ogura:cb,ogura:cbi}, and significantly refined recently by Marie Chazal, Ronnie L.~Loeffen and Pierre Patie in~\cite{clp}. These works again focus on cases when the eigenfunction expansion involves countably many eigenvalues and eigenfunctions, but~\cite{ogura:cb,ogura:cbi} also include some results about generalised eigenfunction expansions, to some extent similar to those discussed here.

The present work builds upon the results obtained in~\cite{kwasnicki:sym,kmr} in the symmetric case, and the fluctuation theory developed in~\cite{kwasnicki:rogers}. The development for the half-line is, to some extent, parallel to the theory in $\R \setminus \{0\}$, studied in~\cite{jk,kwasnicki:point,mucha}.

The author hopes that the present article will stimulate the development of spectral theory for non-normal operators arising in the field of Lévy and Markov processes, and possibly also beyond this scope. Additionally, the results can be of interest for probabilists: the explicit expressions that we find below may lead to certain monotonicity properties of the densities of first-passage times (just as it was the case for symmetric processes, see~\cite{kmr}), and they might be applicable for numerical methods.

We state our results using probabilistic language, in terms of transition densities and passage times of appropriate stochastic processes. We remark, however, that the transition density (or heat kernel) $p_t^+(x, y)$ can be defined without using probability; see Remark~\ref{rem:analysis}. Furthermore, the proofs of our main results are purely analytic. Thus, statements and proofs of Theorems~\ref{thm:heat}, \ref{thm:heat:laplace} and~\ref{thm:heat:spectral} require essentially no knowledge of probability.

\subsection{Main results}

In order to formulate the main theorems, we need a number of definitions. Let $X_t$ be a \emph{Lévy process}: a stochastic process with independent and stationary increments, and càdlàg paths. Throughout the article we assume that $X_t$ has \emph{completely monotone jumps}: the Lévy measure of $X_t$ has a density function $\nu(x)$ such that
\formula{
 \text{both $\nu(x)$ and $\nu(-x)$ are completely monotone functions of $x > 0$.}
}
This class of processes was introduced by L.C.G.~Rogers in 1983, see~\cite{rogers}, in the context of Wiener--Hopf factorisation of Lévy processes, and recently revisited by the author in~\cite{kwasnicki:rogers}. In the present article we extend the methods and results of the latter paper to find spectral-type integral expressions for the distribution functions of the supremum and infimum functionals of $X_t$, as well as for the transition density $p_t^+(x, y)$ of $X_t$ in the half-line $(0, \infty)$. The supremum and infimum functionals are defined by
\formula{
 \overline{X}_t & = \sup\{X_s : s \in [0, t]\} , & \underline{X}_t & = \inf\{X_s : s \in [0, t]\} ,
}
while if $\pr^x$ is the probability corresponding to the process $X_t$ started at $X_0 = x$ and $\tau_{(0, \infty)} = \inf\{t \ge 0 : X_t \le 0\}$ is the first exit time from $(0, \infty)$, then $p_t^+(x, y)$ satisfies
\formula{
 \pr^x(X_t \in E, \, t < \tau_{(0, \infty)}) & = \int_E p_t^+(x, y) dy
}
for every Borel set $E$ and $x > 0$. Detailed definitions and further properties are discussed in Sections~\ref{sec:pre} and~\ref{sec:heat}, and existing literature on the subject is described later in this section.

\medskip

Throughout the paper, by $f(\xi)$ we denote the characteristic exponent of $X_t$, defined by the formula
\formula{
 \ex^0 e^{i \xi X_t} & = e^{-t f(\xi)}
}
for $t \ge 0$ and $\xi \in \R$. By a result of Rogers, $f(\xi)$ extends to a holomorphic function in the right complex half-plane $\{\xi \in \C : \re \xi > 0\}$ --- coined \emph{Rogers function} in~\cite{kwasnicki:rogers} --- and it is shown in~\cite{kwasnicki:rogers} that the equation $\im f(\xi) = 0$ defines a curve $\Gamma$ in the right complex half-plane --- the \emph{spine} of $f$ --- which can be parameterised as
\formula{
 \Gamma & = \{\zeta(r) : r \in Z\} ,
}
with $|\zeta(r)| = r$ and $Z$ an open subset of $(0, \infty)$. Furthermore, the real-valued function
\formula{
 \lambda(r) & = f(\zeta(r))
}
is continuous and increasing with $r$. This description is slightly simplified, and we refer to Section~\ref{sec:rogers} for a detailed exposition. Here we mention that if $X_t$ is symmetric, then $\Gamma$ is the horizontal ray $(0, \infty)$ and $\lambda(r) = f(r)$: we have $Z = (0, \infty)$ and $\zeta(r) = r$. We also make the following remark: in our main results we impose certain conditions on the shape of $\Gamma$, which are satisfied when $\Gamma$ is contained in the sector $\{\xi \in \C : \lvert\Arg \xi\rvert \le \tfrac{\pi}{2} - \eps\}$ for some $\eps > 0$; that is, when $Z = (0, \infty)$ and $\lvert\Arg \zeta(r)\rvert \le \tfrac{\pi}{2} - \eps$ for $r > 0$. However, apparently this condition cannot be easily expressed directly in terms of the characteristics of the process $X_t$; see Section~\ref{sec:ex:stable-like} for a closely related discussion.

The expressions in our main theorems involve certain functions $F_+(r; y)$ and $F_-(r; x)$, where $r \in Z$ and $x, y > 0$. These functions play the role of generalised eigenfunctions of the generator of the process $X_t$ killed at the first exit time from $(0, \infty)$, with corresponding generalised eigenvalues $\lambda(r)$. They are rigorously introduced in Definition~\ref{def:eig}; here we are satisfied with the following description: we have
\formula{
 F_+(r; y) & = e^{b(r) y} \sin(a(r) y + c_+(r)) - G_+(r; y) , \\
 F_-(r; x) & = e^{-b(r) x} \sin(a(r) x + c_-(r)) - G_-(r; x) ,
}
where $a(r) + i b(r) = \zeta(r)$, $c_+(r), c_-(r) \in [0, \pi)$, and $G_+(r; y)$ and $G_-(r; x)$ are certain completely monotone functions of $x, y > 0$. We stress that while $G_+(r; y)$ and $G_-(r; x)$ are smooth, positive, bounded and integrable functions of $x, y > 0$, the other terms $e^{b(r) y} \sin(a(r) y + c_+(r))$ and $e^{-b(r) y} \sin(a(r) y + c_-(r))$ are oscillatory, and unless $\zeta(r)$ is real, one of them has exponential growth at infinity.

The Laplace transforms of $F_+(r; y)$ and $F_-(r; x)$, denoted by $\laplace F_+(r; \eta)$ and $\laplace F_-(r; \xi)$, are well-defined when $\re \eta > \max\{b(r), 0\}$ and $\re \xi > \max\{-b(r), 0\}$, respectively. Both Laplace transforms extend analytically to $\xi, \eta \in (0, \infty)$. We denote these extensions by the same symbols, and furthermore we write $\laplace F_+(r; 0^+)$ and $\laplace F_-(r; 0^+)$ for the corresponding right limits at $0$.

The following are the main results of this article. In all of them we assume that $X_t$ is a non-deterministic Lévy process with completely monotone jumps.

\begin{theorem}
\label{thm:extrema}
Let $\eps \in (0, \tfrac{\pi}{2})$, $\beta \in (1, 1 + \eps / (\tfrac{\pi}{2} - \eps))$, $\delta \in [0, \tfrac{\pi}{2})$ and $\ro > 0$.
\begin{enumerate}[label=\textrm{(\alph*)}]
\item\label{thm:extrema:a} Suppose that
\formula**{
\label{eq:extrema:sup:a2}
 \limsup_{r \to \infty} \Arg \zeta(r) < \frac{\pi}{2} - \eps , \\
\label{eq:extrema:sup:a3}
 \sup \{ \Arg f(-i e^{i \delta} r) : r \in (0, \ro) \} < 0
}
(see Figure~\ref{fig:extrema}(a)), and that for every $t > 0$ there is a constant $C$ such that for $s > 0$ we have
\formula[eq:extrema:sup:a4]{
 \int_Z e^{s \max\{ \im \zeta(r), 0 \} - t \lambda(r)} \lvert\zeta'(r)\rvert dr & \le e^{C (1 + s)^\beta} .
}
Then
\formula[eq:extrema:sup]{
 \pr^0(\overline{X}_t < y) & = \frac{2}{\pi} \int_Z e^{-t \lambda(r)} F_+(r; y) \laplace F_-(r; 0^+) \lvert\zeta'(r)\rvert dr
}
for $t \ge 0$ and $y > 0$.
\item\label{thm:extrema:b} Similarly, suppose that
\formula**{
\label{eq:extrema:inf:a2}
 \liminf_{r \to \infty} \Arg \zeta(r) > -\frac{\pi}{2} + \eps , \\
\label{eq:extrema:inf:a3}
 \inf \{ \Arg f(i e^{-i \delta} r) : r \in (0, \ro) \} > 0 ,
}
and that for every $t > 0$ there is a constant $C$ such that for $s > 0$ we have
\formula[eq:extrema:inf:a4]{
 \int_Z e^{s \max\{ -\im \zeta(r), 0 \} - t \lambda(r)} \lvert\zeta'(r)\rvert dr & \le e^{C (1 + s)^\beta} .
}
Then
\formula[eq:extrema:inf]{
 \pr^0(\underline{X}_t > -x) & = \frac{2}{\pi} \int_Z e^{-t \lambda(r)} \laplace F_+(r; 0^+) F_-(r; x) \lvert\zeta'(r)\rvert dr
}
for $t \ge 0$ and $x > 0$.
\end{enumerate}
\end{theorem}

\begin{theorem}
\label{thm:heat}
Let $\eps \in (0, \tfrac{\pi}{2})$, $\beta \in (1, 1 + \eps / (\tfrac{\pi}{2} - \eps))$. Suppose that
\formula[eq:heat:a1]{
 \limsup_{r \to \infty} \, \lvert\Arg \zeta(r)\rvert & < \frac{\pi}{2} - \eps
}
(see Figure~\ref{fig:extrema}(b)), and that for every $t > 0$ there is a constant $C$ such that for $s > 0$ we have
\formula[eq:heat:a2]{
 \int_Z e^{s \lvert\im \zeta(r)\rvert - t \lambda(r)} \lvert\zeta'(r)\rvert dr & \le e^{C (1 + s)^\beta} .
}
Then
\formula[eq:heat]{
 p_t^+(x, y) & = \frac{2}{\pi} \int_Z e^{-t \lambda(r)} F_+(r; y) F_-(r; x) \lvert\zeta'(r)\rvert dr
}
for $t \ge 0$ and $x, y > 0$.
\end{theorem}

\begin{figure}
\centering
\begin{tabular}{cc}
\begin{tikzpicture}
\footnotesize
\coordinate (X) at (3,0);
\coordinate (Y) at (0,3);
\coordinate (Xn) at (-1,0);
\coordinate (Yn) at (0,-3);
\coordinate (O) at (0,0);
\coordinate (E) at (0.6,3);
\coordinate (R) at (0.3,1.5);
\coordinate (T) at (0.4,-1.2);
\coordinate (P) at (0.2,-3);
\phantom{\node[below, olive] at (P) {$\Arg f < 0$};}
\fill[fill=teal!50!white, draw=teal] (E) -- (R) arc (78.69:90:1.5297) -- (Y);
\draw[dotted, teal] (O) -- (R);
\node[above, teal] at (E) {$\Arg f > 0$};
\pic["$\eps$",draw,angle eccentricity=1.1,angle radius=2cm] {angle=E--O--Y};
\draw[very thick, violet] (O) -- (T) node[right] {$\sup \Arg f < 0$};
\pic["$\delta$",draw,angle eccentricity=1.2,angle radius=1cm] {angle=Yn--O--T};
\draw[->] (Xn) -- (X) node[above] {$\re$};
\draw[->] (Yn) -- (Y) node[left] {$\im$};
\end{tikzpicture}
&
\begin{tikzpicture}
\footnotesize
\coordinate (X) at (3,0);
\coordinate (Y) at (0,3);
\coordinate (Xn) at (-1,0);
\coordinate (Yn) at (0,-3);
\coordinate (O) at (0,0);
\coordinate (E) at (0.6,3);
\coordinate (R) at (0.3,1.5);
\coordinate (En) at (0.6,-3);
\coordinate (Rn) at (0.3,-1.5);
\fill[fill=teal!50!white, draw=teal] (En) -- (Rn) arc (-78.69:-90:1.5297) -- (Yn);
\draw[dotted, teal] (O) -- (Rn);
\node[below, teal] at (En) {$\Arg f < 0$};
\fill[fill=teal!50!white, draw=teal] (E) -- (R) arc (78.69:90:1.5297) -- (Y);
\draw[dotted, teal] (O) -- (R);
\node[above, teal] at (E) {$\Arg f > 0$};
\pic["$\eps$",draw,angle eccentricity=1.1,angle radius=2cm] {angle=E--O--Y};
\pic["$\eps$",draw,angle eccentricity=1.1,angle radius=2cm] {angle=Yn--O--En};
\draw[->] (Xn) -- (X) node[above] {$\re$};
\draw[->] (Yn) -- (Y) node[left] {$\im$};
\end{tikzpicture}
\\
(a)&(b)
\end{tabular}
\caption{(a) Setting for Theorem~\ref{thm:extrema}\ref{thm:extrema:a}. (b) Setting for Theorems~\ref{thm:heat} and~\ref{thm:heat:spectral}.}
\label{fig:extrema}
\end{figure}
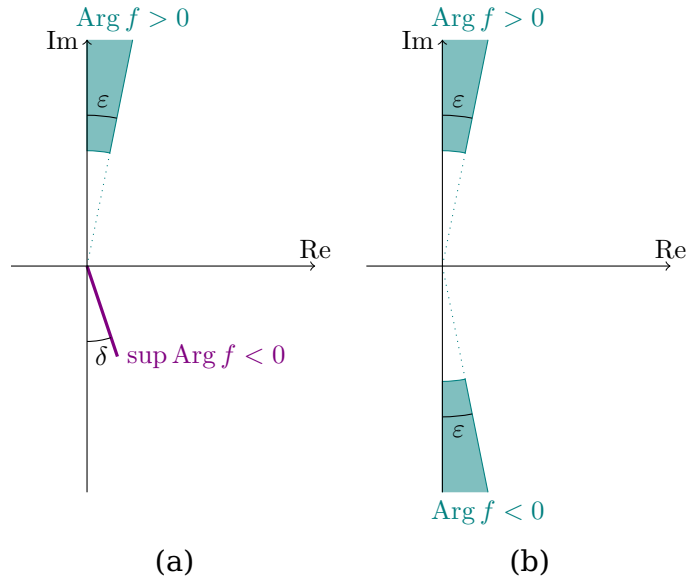

\begin{remark}
\label{rem:analysis}
The transition density $p_t^+(x, y)$ is the Dirichlet heat kernel in $(0, \infty)$, with zero exterior condition in $(-\infty, 0]$, for the \emph{Lévy operator} $L$ defined for smooth, compactly supported function $u$ by the formula
\formula{
 L u(x) & = \tfrac{1}{2} \sigma^2 u''(x) + \mu u'(x) + \int_{-\infty}^\infty \bigl(u(x + y) - u(x) - y u'(x) \ind_{(-1, 1)}(y)\bigr) \nu(y) dy ,
}
where $\sigma \ge 0$, $\mu \in \R$ and both $\nu(z)$ and $\nu(-z)$ are completely monotone functions on $(0, \infty)$. The Fourier symbol of $L$ is equal to the characteristic exponent $f(\xi)$ of the Lévy process, and it is given by the Lévy--Khintchine formula
\formula{
 f(\xi) & = \tfrac{1}{2} \sigma^2 \xi^2 - i \mu \xi + \int_{-\infty}^\infty \bigl(1 - e^{i \xi x} + i \xi x \ind_{(-1, 1)}(x)\bigr) \nu(x) dx .
}
The operator $L$ is the generator of the Lévy process $X_t$. Noteworthy, estimates of Dirichlet heat kernels of Lévy processes attracted significant attention over the past decade, and explicit expressions for the Dirichlet heat kernel are rarely available.
\end{remark}

\begin{remark}
\label{rem:probability}
Our formulae for the distribution of the supremum and infimum functionals in Theorem~\ref{thm:extrema} are, in general, relatively complicated: they include several integrals, and they involve implicitly defined quantities, such as $\zeta(r)$. Nevertheless, for some relatively simple Lévy processes with completely monotone jumps one can simplify these expressions significantly, and obtain an explicit integral formula for $\pr^0(\overline{X}_t < y)$ or $\pr^0(\underline{X}_t > -x)$. Some examples are given in Section~\ref{sec:ex}. Thus, we extend the (surprisingly short) list of Lévy processes for which numerically tractable expressions for distribution functions of $\overline{X}_t$ or $\underline{X}_t$ are available. Other contributions to this list known to the author include: Brownian motion~\cite{levy}, Brownian motion with drift~\cite{schrodinger,smoluchowski,cm}, symmetric Poisson process~\cite{bd}, Poisson process with drift~\cite{gp,pyke}, symmetric 1-stable Lévy process (or Cauchy process)~\cite{darling}, classical risk process~\cite{asmussen,aa}, one-sided stable Lévy processes~\cite{skorokhod,heyde,bingham,bdp,doney-1,gj}, and virtually all stable Lévy processes~\cite{hk}. Additional examples are provided by specialising Takács's formula~\cite{cramer,prabhu,takacs} to particular one-sided Lévy processes.
\end{remark}

Assumptions~\eqref{eq:extrema:sup:a2}, \eqref{eq:extrema:inf:a2} and~\eqref{eq:heat:a1} should be regarded as (rather mild) geometric constraints on the curve $\Gamma$. While assumptions~\eqref{eq:extrema:sup:a3} and~\eqref{eq:extrema:inf:a3} also impose geometric conditions on $\Gamma$, they additionally require that the decay of $|f(\xi)|$ near~$0$ is faster than some power of $|\xi|$; see Lemma~\ref{lem:r:growth}. Assumptions~\eqref{eq:extrema:sup:a4}, \eqref{eq:extrema:inf:a4} and~\eqref{eq:heat:a2} are another growth conditions on $|f(\xi)|$, but this time for large $\xi$. They are automatically satisfied if $|f(\xi)|$ grows at least as fast as $|\xi|$ raised to a power greater than $1$; we refer to Section~\ref{sec:ex} for further discussion and examples.

The assumptions in the above theorems may appear inexplicit. In practice, however, they turn out to be easily applicable and relatively general. In particular, Theorems~\ref{thm:extrema} and~\ref{thm:heat} cover at least some typical examples of Lévy processes with completely monotone jumps: a wide class of processes with non-zero Brownian component; symmetric Lévy processes with completely monotone jumps under mild growth condition on $|f(\xi)|$ at infinity; all strictly stable Lévy processes with index $\alpha > \tfrac{3}{2}$, and some strictly stable Lévy processes with index $\alpha \in (1, \tfrac{3}{2}]$; and many stable-like Lévy processes. These examples are discussed in Section~\ref{sec:ex}.

\medskip

Although apparently our results hold true in a slightly wider context (for example, the main results of the present paper are proved in~\cite{kk:stable} for all stable processes with index $\alpha > 1$), it is clear that Theorems~\ref{thm:extrema} and~\ref{thm:heat} do not extend directly to all Lévy processes with completely monotone jumps. Indeed, if, for example, $X_t$ is a non-symmetric strictly stable Lévy process with index $\alpha \le 1$, then the integral in~\eqref{eq:heat}, as well as one of the integrals in~\eqref{eq:extrema:sup} and~\eqref{eq:extrema:inf}, diverges. Once again we refer to Section~\ref{sec:ex} for more examples and a detailed discussion.

It is easy to see that
\formula{
 \pr^0(\underline{X}_t > -x) & = \pr^x(\underline{X}_t > 0) = \int_0^\infty p_t^+(x, y) dy ,
}
and, by Hunt's switching identity (Theorem~II.5 in~\cite{bertoin}), similarly
\formula{
 \pr^0(\overline{X}_t < y) & = \int_0^\infty p_t^+(x, y) dx .
}
This suggests that Theorem~\ref{thm:extrema} is a simple consequence of Theorem~\ref{thm:heat}. Due to exponential growth of $F_+(r; y)$ or $F_-(r; x)$, however, this does not seem to be the case, and our proofs follow a different path. On the other hand, these two results clearly have a common root. In fact, both are derived from the following general result, in which, for notational convenience, we extend the definition of $\zeta(r)$ to all $r > 0$ in such a way that it is continuous on $(0, \infty)$ and piecewise linear on $(0, \infty) \setminus Z$. We stress that here we impose no restrictions on the Lévy process $X_t$ other than it is non-deterministic and it has completely monotone jumps.

\begin{theorem}
\label{thm:heat:laplace}
Let $t \ge 0$. For every $\xi, \eta > 0$ such that $\Arg \zeta(\xi) < \tfrac{\pi}{2}$ and $\Arg \zeta(\eta) > -\tfrac{\pi}{2}$, we have
\formula[eq:heat:laplace]{
 \int_0^\infty \int_0^\infty e^{-\eta x - \xi y} p_t^+(x, y) dy dx & = \frac{2}{\pi} \int_Z e^{-t \lambda(r)} \laplace F_+(r; \xi) \laplace F_-(r; \eta) \lvert\zeta'(r)\rvert dr .
}
\end{theorem}

We remark that the above result also covers the case when the distribution of $X_t$ contains an atom, and in this case $p_t^+(x, y) dy$ needs to be understood appropriately. We also point out that the sets of admissible $\xi$ and admissible $\eta$ in Theorem~\ref{thm:heat:laplace} always contain the (non-empty) open set $Z$, and that the union of these sets is $(0, \infty)$.

\medskip

Theorem~\ref{thm:heat:laplace} can be viewed as a spectral-type decomposition of the transition operators of $X_t$ in $(0, \infty)$. More precisely, let
\formula{
 P_t^+ u(x) & = \int_0^\infty p_t^+(x, y) u(y) dy
}
whenever the integral converges. Furthermore, write
\formula{
 \langle u, v \rangle & = \int_0^\infty u(x) v(x) dx
}
whenever the integral is finite. If we define $u(y) = e^{-\xi y}$ and $v(x) = e^{-\eta x}$, then the assertion of Theorem~\ref{thm:heat:laplace} can be formally written as
\formula{
 \langle v, P_t^+ u \rangle & = \frac{2}{\pi} \int_Z e^{-t \lambda(r)} \langle v, F_-(r; \cdot) \rangle \langle F_+(r; \cdot), u \rangle \lvert\zeta'(r)\rvert dr ,
}
except that the integrals on the right-hand side:
\formula{
 \langle v, F_-(r; \cdot) \rangle & = \int_0^\infty F_-(r; x) e^{-\eta x} dx = \laplace F_-(r; \eta) , \\
 \langle F_+(r; \cdot), u \rangle & = \int_0^\infty F_+(r; y) e^{-\xi y} dy = \laplace F_+(r; \xi)
}
are not necessarily well-defined. A rigorous statement of that kind requires a more careful choice of test functions $u(y)$ and $v(x)$, and an additional assumption.

\begin{theorem}
\label{thm:heat:spectral}
Let $\eps > 0$, and suppose that
\formula[eq:heat:spectral:a1]{
 \limsup_{r \to \infty} \, \lvert\Arg \zeta(r)\rvert & < \frac{\pi}{2} - \eps
}
(see Figure~\ref{fig:extrema}(b)). Let $t \ge 0$, and let $u(y)$ and $v(x)$ be holomorphic functions in the region $\{z \in \C : \lvert\Arg z\rvert < \tfrac{\pi}{2} - \eps\}$, which are real-valued on $(0, \infty)$, and which satisfy
\formula[eq:heat:spectral:a2]{
 |u(y)| & \le C_1 e^{-C_2 |y| \log(1 + |y|)} , & |v(x)| & \le C_1 e^{-C_2 |x| \log(1 + |x|)}
}
for some constants $C_1$ and $C_2$ and all $x$ and $y$ in the region $\{z \in \C : \lvert\Arg z\rvert < \tfrac{\pi}{2} - \eps\}$. Then
\formula[eq:heat:spectral]{
 \langle v, P_t^+ u \rangle & = \frac{2}{\pi} \int_Z e^{-t \lambda(r)} \langle v, F_-(r; \cdot) \rangle \langle F_+(r; \cdot), u \rangle \lvert\zeta'(r)\rvert dr .
}
\end{theorem}

Theorem~\ref{thm:heat:laplace} is proved in Section~\ref{sec:heat} as Proposition~\ref{prop:pt:lap}. The above result follows then by an involved argument, which requires various contour deformations; see Proposition~\ref{prop:pt:uv}. On the other hand, Theorem~\ref{thm:heat} is deduced from Theorem~\ref{thm:heat:spectral} in a rather straightforward way: we use Fubini's theorem to change the order of the three integrals on the right-hand side of~\eqref{eq:heat:spectral} and then apply a density argument. This is done in Section~\ref{sec:inf:pt}, where Theorem~\ref{thm:heat} is restated as Corollary~\ref{cor:pt:log}. Theorem~\ref{thm:extrema} is proved in Section~\ref{sec:inf:inf}, and it requires a minor modification of Theorem~\ref{thm:heat:spectral} given in Proposition~\ref{prop:lpt:uv}.

\begin{remark}
\label{rem:nonnormal}
The operators $P_t^+$ are closely related to the Wiener--Hopf factorisation of convolution operators. More precisely, the generator $L^+$ of the semigroup of operators $P_t^+$ (or its inverse $(L^+)^{-1}$, known as the potential operator) is the \emph{Wiener--Hopf operator} for the Lévy operator $L$ introduced in Remark~\ref{rem:analysis} (or its inverse $L^{-1}$).

At least in the following two cases: (a)~when $X_t$ is a compound Poisson process (and then $L$ is a bounded operator on $L^2(\R)$); or (b)~when $X_t$ is killed at a positive rate (this case is not covered in the introduction, but it is allowed in Section~\ref{sec:pre}; then $L^{-1}$ is a bounded operator on $L^2(\R)$), spectral properties of $L^+$ or $(L^+)^{-1}$ are relatively well-understood. In particular, it is known that if $X_t$ is not symmetric, then $L^+$ (in case~(a)) or $(L^+)^{-1}$ (in case~(b)) is not a normal operator on $L^2((0, \infty))$, as either this operator or its adjoint has an uncountable family of eigenfunctions. This means that for no $t > 0$ is $P_t^+$ a normal operator on $L^2((0, \infty))$. We refer the reader to~\cite{as} for these properties, further discussion and references.

While it is natural to expect that the operators $P_t^+$ are not normal for general non-symmetric Lévy processes with completely monotone jumps, such a result does not seem to be available in literature; see also Remark~\ref{rem:eigenfunctions}.
\end{remark}

\subsection{Spectral-theoretic motivations}

As mentioned in the first part of this section, our main results are one of the first examples of spectral decomposition of operators associated to non-symmetric Markov processes. Here we extend this discussion by providing additional context.

Suppose first that the operators $P_t$, $t \ge 0$, form a strongly continuous semigroup of finite-dimensional self-adjoint contractions. Then $P_t$ take a diagonal form in the basis of eigenvectors of $P_t$ (which do not depend on $t$). More precisely, there is a complete orthonormal set of eigenvectors $F_n$, $n = 1, \ldots, N$, of the operators $P_t$, with corresponding eigenvalues of the form $e^{-t \lambda_n}$, such that
\formula[eq:spectral:series:1]{
 \langle v, P_t u \rangle & = \sum_{n = 1}^N e^{-t \lambda_n} \langle v, F_n \rangle \langle F_n, u \rangle
}
for all vectors $u$ and $v$; here $\langle \cdot, \cdot \rangle$ is the usual (complex) inner product. The same result is true if $P_t$ is a strongly continuous semigroup of compact self-adjoint contractions on a Hilbert space of infinite dimension; in this case $N = \infty$ in~\eqref{eq:spectral:series:1}, and $\lim_{n \to \infty} \lambda_n = \infty$.

In the context of general strongly continuous semigroups of bounded self-adjoint contractions on a Hilbert space $L^2(X, m)$, a similar spectral resolution of $P_t$ is provided by the classical spectral theorem. This is a much more abstract result, compared to the case of compact operators. The spectral theorem takes the following more explicit form, essentially due to Gårding~\cite{garding}, when $P_t$ are \emph{Carleman operators}, that is, when the kernel function $p_t(x, y)$ is square-integrable with respect to $y$ for almost every $x$ (or, equivalently, due to the Chapman--Kolmogorov equation, $p_{2 t}(x, x) < \infty$ for almost every $x$; see~\cite{hs} for a general account on Carleman operators). There is a set $Z$, a $\sigma$-finite measure $\mu(dr)$ on $Z$, a measurable real-valued function $\lambda(r)$ on $Z$, and a family of \emph{generalised eigenfunctions} $F_r(x)$, with $r \in Z$ and $x \in X$, with the following property: for every $t > 0$ we have
\formula{
 p_t(x, y) & = \int_Z e^{-t \lambda(r)} F_r(x) F_r(y) \mu(dr)
}
for almost all $x, y \in X$. If we denote $\langle u, v \rangle = \int_X u(x) v(x) m(dx)$ whenever the integral is finite, then also
\formula[eq:spectral:int:1]{
 \langle v, P_t u \rangle & = \int_Z e^{-t \lambda(r)} \langle v, F_r \rangle \langle F_r, u \rangle \mu(dr)
}
for all bounded functions $u$, $v$ which vanish outside of a set of finite measure $m$. Here $F_r$ are real-valued, they need not belong to $L^2(X, m)$, but nevertheless we have $P_t F_r(x) = e^{-t \lambda(r)} F_r(x)$ for almost all $x \in X$ with respect to $m$, for all $t > 0$ and all $r \in Z$; this explains the name \emph{generalised eigenfunction}. The above description is a reformulation of the results of Getoor, see Section~7 in~\cite{getoor}.

The picture is much less clear when we drop the assumption that $P_t$ are self-adjoint, and we allow $P_t$ to be arbitrary (not even normal) contractions. In the finite-dimensional case, if $P_t$ are diagonalisable, one can still write
\formula[eq:spectral:series:2]{
 \langle v, P_t u \rangle & = \sum_{n = 1}^N e^{-t \overline{\lambda_n}} \langle v, F_n^- \rangle \langle F_n^+, u \rangle ,
}
where $F_n^-$ are eigenvectors of $P_t$, and $F_n^+$ are co-eigenvectors of $P_t$ (that is, eigenvectors of the adjoint operator); more generally, Jordan normal form can be used. However, even in the case of compact operators on an infinite-dimensional Hilbert space, the author is aware of no general result similar to the ones given above. In particular, even when the operators $P_t$ have linearly dense families of eigenvectors $F_n^-$ and co-eigenvectors $F_n^+$, convergence of the series~\eqref{eq:spectral:series:2} is problematic. Virtually nothing seems to be known in general, when $P_t$ are not assumed to be compact. The results of~\cite{kk:stable} for strictly stable Lévy processes, of~\cite{ps:orbit,ps:cauchy,ps:laguerre} for positive self-similar Markov processes, and of~\cite{ogura:cb,ogura:cbi} for continuous state branching processes were already discussed in the introduction, while discrete counterparts of positive self-similar Markov processes are discussed in this context in~\cite{mps}. Other loosely related works include: non-normal perturbations of self-adjoint operators~\cite{ramm}, operators with `small' commutators~\cite{campbell,delvalle}, and certain differential operators~\cite{browder,dtv,ps:spectral}; see also~\cite{sjostrand:lecture,sjostrand:non-self-adjoint,te} for a more general discussion.

When $X_t$ is non-symmetric, the operators $P_t^+$ considered in Theorem~\ref{thm:heat:spectral} are, in general, not normal (see Remark~\ref{rem:nonnormal}). Therefore, Theorem~\ref{thm:heat:spectral} is a rare example of spectral resolution of a non-normal operator. We emphasise that it fits into the above picture: by Theorem~\ref{thm:heat:spectral}, for an appropriate class of functions $u$ and $v$, we have
\formula[eq:spectral:int:2]{
 \langle v, P_t^+ u \rangle & = \int_Z e^{-t \lambda(r)} \langle v, F_r^- \rangle \langle F_r^+, u \rangle \mu(dr) ,
}
with $Z$ and $\lambda(r)$ as in Theorem~\ref{thm:heat:spectral}, with $\mu(dr) = \tfrac{2}{\pi} \lvert\zeta'(r)\rvert dr$, and with generalised eigenfunctions $F^-_r(x) = F_-(r; x)$ and generalised co-eigenfunctions $F^+_r(y) = F_+(r; y)$. 

\medskip

It should be emphasised that the spectral resolution described in Theorem~\ref{thm:heat:spectral} is likely not unique. Here it is instructive to consider the transition semigroup of $X_t$ in $\R$ rather than in $(0, \infty)$. Let $p_t(x)$ be the density function of the distribution of $X_t$ with respect to $\pr^0$, and let
\formula{
 P_t u(x) & = \int_{-\infty}^\infty p_t(y - x) u(y) dy = \ex^x u(X_t)
}
whenever the integral converges. If $e^{-t f(\xi)}$ is an integrable function of $\xi \in \R$ for every $t > 0$, then, by the Fourier inversion formula, transition densities $p_t(y - x)$ indeed exist, and we have
\formula{
 p_t(y - x) & = \frac{1}{2 \pi} \int_{-\infty}^\infty e^{-t f(r)} e^{i r x} e^{-i r y} dr .
}
If we write $F_r(x) = e^{i r x}$, and if we denote by $\langle u, v \rangle = \int_{-\infty}^\infty u(x) \overline{v(x)} dx$ the usual inner product in $L^2(\R)$ (with a complex conjugate: now all functions are complex-valued), then it follows that
\formula[eq:spectral:heat:1]{
 \langle v, P_t u \rangle & = \frac{1}{2 \pi} \int_{-\infty}^\infty e^{-t f(r)} \langle v, F_r \rangle \langle F_r, u \rangle dr
}
if, say, $u$ and $v$ are smooth and compactly supported. This follows the pattern~\eqref{eq:spectral:int:1} discussed above (note that $P_t$ are normal operators, so $\lambda(r) = f(r)$ and $F_r(x) = e^{-i r x}$ are now complex-valued). However, under mild assumptions, we can deform the contour of integration on the right-hand side to the contour $\Gamma$ discussed above, and get a similar expression that only involve real-valued functions. Indeed, after a short calculation that we omit here, we arrive at
\formula*[eq:spectral:heat:2]{
 \langle v, P_t u \rangle & = \frac{1}{\pi} \int_Z e^{-t \lambda(r)} \langle v, F_{1,r}^- \rangle \langle F_{1,r}^+, u \rangle \lvert\zeta'(r)\rvert dr \\
 & \qquad + \frac{1}{\pi} \int_Z e^{-t \lambda(r)} \langle v, F_{2,r}^- \rangle \langle F_{2,r}^+, u \rangle \lvert\zeta'(r)\rvert dr ,
}
where we denoted
\formula{
 F_{1,r}^\pm(x) & = e^{\pm y \im \zeta(r)} \cos (y \re \zeta(r) \mp \thet(r)), \\
 F_{2,r}^\pm(x) & = e^{\pm y \im \zeta(r)} \sin (y \re \zeta(r) \mp \thet(r)) ,
}
with $\thet(r) = \tfrac{1}{2} \Arg \zeta'(r)$. This is again of the same form as in~\eqref{eq:spectral:int:2}, with $Z$ replaced by $\{1, 2\} \times Z$, with $\lambda(j, r) = \lambda(r)$ and $\mu(dj, dr) = \tfrac{1}{\pi} (\delta_1(dj) + \delta_2(dj)) \lvert\zeta'(r)\rvert dr$. This example shows that the spectral resolution of the form~\eqref{eq:spectral:int:2} is not unique, even if the operators $P_t$ are normal.

It is not clear in general which choice of $F_r^+(y)$ and $F_r^-(x)$ is more appropriate. In the above example it seems more natural to consider the decomposition given in~\eqref{eq:spectral:heat:1}: in this expression the eigenfunctions and co-eigenfunctions coincide, and additionally they are bounded functions. On the other hand, expressions in~\eqref{eq:spectral:heat:2} only involves real-valued quantities. Our Theorem~\ref{thm:heat:spectral} provides an analogue of the latter formula for the transition semigroup $P_t^+$ of the process $X_t$ in $(0, \infty)$. In this case, that seems to be an optimal choice: a perfect analogue of~\eqref{eq:spectral:heat:1} is not possible because the operators $P_t^+$ are not normal, and apparently it is not possible to choose a more regular (for example: bounded) family of generalised eigenfunctions.

\medskip

In some cases, described in Theorem~\ref{thm:heat}, the class of admissible functions $u$ and $v$ in~\eqref{eq:spectral:int:2} (or in Theorem~\ref{thm:heat:spectral}) is sufficiently rich in order to get an expression for the kernel $p_t^+(x, y)$ of $P_t^+$. Note, however, that this class does not include any compactly supported functions, and it is apparently a difficult question under what assumptions~\eqref{eq:spectral:int:2} can be extended to (sufficiently regular) compactly supported functions $u$ and $v$.

\subsection{Potential spectral-theoretic implications}

At a purely formal level, Theorem~\ref{thm:heat:spectral} can be regarded as a similarity (or intertwining) relation between the operator $P_t^+$ with kernel $p_t^+(x, dy)$, defined on the class of admissible functions $u$ described in the statement of the proposition, and the multiplication operator with symbol $e^{-t \lambda(r)}$. This is to be understood as follows.

Let $U$ denote the class of functions $u(y)$ or $v(x)$ satisfying the assumptions of Theorem~\ref{thm:heat:spectral}, and define
\formula{
 \Pi_+ u(r) & = \int_0^\infty F_+(r; y) u(y) dy , & \Pi_- v(r) & = \int_0^\infty F_-(r; x) v(x) dx
}
whenever $u(y)$ and $v(x)$ are in $U$. Then Proposition~\ref{prop:pt:uv} can be stated as
\formula{
 \int_0^\infty u(x) P_t^+ v(x) dx & = \frac{2}{\pi} \int_Z e^{-t \lambda(r)} \Pi_+ u(r) \Pi_- v(r) \lvert\zeta'(r)\rvert dr .
}
In particular, for $t = 0$, we find that
\formula{
 \int_0^\infty u(x) v(x) dx & = \frac{2}{\pi} \int_Z \Pi_+ u(r) \Pi_- v(r) \lvert\zeta'(r)\rvert dr .
}

Let us consider $\Pi_-$ and $\Pi_+$ as densely defined unbounded operators from the Hilbert space $V = L^2((0, \infty), dx)$ to the Hilbert space $W = L^2(Z, \lvert\zeta'(r)\rvert dr)$, with domain $U \subseteq V$ (see Remark~\ref{rem:pi}). Denote by $Q_t$ the multiplication operator $Q_t h(r) = e^{-t \lambda(r)} h(r)$ on $W$. We thus have
\formula{
 \langle u,  P_t^+ v \rangle_V & = \frac{2}{\pi} \langle \Pi_+ u, Q_t \Pi_- v \rangle_W
}
and
\formula{
 \langle u,  v \rangle_V & = \frac{2}{\pi} \langle \Pi_+ u, \Pi_- v \rangle_W .
}
The latter equality implies that for every $v \in U$, the function $\Pi_- v$ is in the domain of the adjoint of $\Pi_+$, and $\Pi_+^* \Pi_- v = v$. Thus, by the former equality,
\formula{
 \langle u, P_t^+ \Pi_+^* \Pi_- v \rangle_V & = \langle u, P_t^+ v \rangle_V \\
 & = \frac{2}{\pi} \langle \Pi_+ u, Q_t \Pi_- v \rangle_W = \frac{2}{\pi} \langle u, \Pi_+^* Q_t \Pi_- v \rangle_V .
}
This proves the following partial similarity relation:
\formula[eq:af:sim:1]{
 P_t^+ \Pi_+^* & = \Pi_+^* Q_t && \text{on $\Pi_- U$,}
}
which becomes more meaningful if $\Pi_- U$ is dense in $W$. Exactly the same reasoning shows that $\Pi_-^* \Pi_+ u = u$ for every $u \in U$, and so
\formula{
 \langle \Pi_-^* \Pi_+ u, P_t^+ v \rangle_W & = \langle \Pi_+ u, Q_t \Pi_- v \rangle_W && \text{for $u, v \in U$.}
}
If we were able to prove that $\Pi_+ U$ is dense in $W$, $\Pi_-$ is closable, and $P_t^+ v$ lies in the domain of the closure $\overline{\Pi}_-$ of $\Pi_-$, then we would obtain another partial similarity relation
\formula[eq:af:sim:2]{
 \Pi_- P_t^+ & = Q_t \Pi_- && \text{on $U$.}
}
However, we end this discussion here, and we leave a more detailed study of the above similarity relations for a future work. We also refer to~\cite{patie} for a thorough discussion of related ideas for normal operators on a more abstract level.

We conclude this part with the following observation.

\begin{remark}
\label{rem:eigenfunctions}
Consider $r \in Z$. If $\im \zeta(r) < 0$, then $F_+(r; y)$ is in $L^2((0, \infty))$ as a function of $y$, while if $\im \zeta(r) > 0$, then $F_-(r; x)$ is in $L^2((0, \infty))$ as a function of $x$. It is natural to expect that in the latter case, $F_-(r; x)$ is a true (not generalised) eigenfunction of the operators $P_t^+$, and similarly in the former case $F_+(r; y)$ is a true co-eigenfunction of $P_t^+$. In particular, this would imply that $P_t^+$ is never a normal operator unless the process is symmetric; see Remark~\ref{rem:nonnormal}. However, the results obtained in the present article do not seem to immediately imply that $F_+(r; y)$ or $F_-(r; x)$ is necessarily a (co-)eigenfunction whenever it is square integrable, and we postpone a detailed analysis of spectral properties of $P_t^+$ to a future work.

Let us also note that if $\im \zeta(r) \le 0$, then $F_+(r; y)$ is a bounded function of $y$, while if $\im \zeta(r) \ge 0$, then $F_-(r; x)$ is a bounded function of $x$, and it is natural to expect that these functions are then (co-)eigenfunctions of the operators $P_t^+$ on $L^\infty((0, \infty))$. This is known to be the case for symmetric processes; see~\cite{kwasnicki:sym,kmr}.
\end{remark}

\subsection{Idea of the proof}

The proof of our main theorems is rather lengthy, so for the convenience of the reader, below we outline the main ideas. The argument is given in Section~\ref{sec:heat}, where Theorem~\ref{thm:heat:laplace} is proved, and Section~\ref{sec:inf}, containing the proof of the other main results.

\medskip

In Section~\ref{sec:heat}, we begin with the following well-known identity essentially due to Pecherskii and Rogozin:
\formula{
 \int_0^\infty \int_0^\infty \int_0^\infty e^{-\sigma t - \eta x - \xi y} p_t^+(x, y) dx dy dt & = \frac{1}{\xi + \eta} \, \frac{1}{\kappa^+(\sigma, \xi) \kappa^-(\sigma, \eta)} \, ;
}
see Proposition~\ref{prop:pt:pr}. Here $\kappa^+(\sigma, \xi)$ and $\kappa^-(\sigma, \eta)$ are the Wiener--Hopf factors corresponding to $X_t$, and by the Baxter--Donsker formula we have
\formula{
 \frac{\kappa^+(\sigma, \xi)}{\kappa^+(\sigma, 0)} & = \exp\biggl(\frac{1}{2 \pi} \int_{-\infty}^\infty \biggl(\frac{1}{\xi + i z} - \frac{1}{i z}\biggr) \log(\sigma + f(z)) dz\biggr) \\
 \frac{\kappa^-(\sigma, \xi)}{\kappa^-(\sigma, 0)} & = \exp\biggl(\frac{1}{2 \pi} \int_{-\infty}^\infty \biggl(\frac{1}{\xi - i z} - \frac{1}{-i z}\biggr) \log(\sigma + f(z)) dz\biggr) .
}
(the above formula requires that $f(\xi) / \xi$ is integrable near $\xi = 0$, but let us ignore this technicality here). In~\cite{kwasnicki:rogers}, the contour of integration with respect to $z$ in the above Baxter--Donsker formulae was deformed from $\R$ to $\Gamma$, the spine of $f(\xi)$. This led to a more convenient expression for the Wiener--Hopf factors, and in our case it allows us to write the (tri-variate) Laplace transform of $p_t^+(x, y)$ as
\formula{
 \int_0^\infty \int_0^\infty \int_0^\infty e^{-\sigma t - \eta x - \xi y} p_t^+(x, y) dx dy dt & = \frac{1}{\sigma (\xi + \eta)} \, \exp\biggl(-\frac{1}{\pi} \int_0^{\infty} \psi(r) \, \frac{ds}{\sigma + \lambda(r)} \biggr) .
}
for an appropriate function $\psi$; see~\eqref{eq:pt:bd}. We recall that the spine $\Gamma$ is the line along which the holomorphic extension of $f(\xi)$ takes real values, $\zeta(r)$ is the parameterisation of $\Gamma_r$ such that $|\zeta(r)| = r$, and $\lambda(r) = f(\zeta(r))$ are the values taken by $f$ along the spine $\Gamma$. We emphasise that the derivation of the above variant of Baxter--Donsker formula crucially depends on the properties of the holomorphic extension of $f(\xi)$, which is a consequence of our assumptions that the Lévy process $X_t$ has completely monotone jumps (or, equivalently, that $f(\xi)$ is a Rogers function).

It is now rather straightforward to inverse the Laplace transform with respect to the temporal variable $t$. Indeed: it suffices to use inverse Fourier--Laplace transform on the right-hand side, and deform the contour of integration to Hankel's contour (the one that goes from $-\infty$ to $0$ on the bottom side of the real axis, turns around $0$, and goes back to $-\infty$ along the top side of the real axis). The corresponding result is given in Proposition~\ref{prop:pt:laplace}.

A much more challenging task is to identify the expressions obtained by the above procedure in terms of the Laplace transforms of the generalised eigenfunctions $F_+(r; y)$ and $F_-(r; x)$. In order to do so in Proposition~\ref{prop:pt:quot}, we need to show appropriate regularity of $\zeta(r)$ (Proposition~\ref{prop:zeta:holder}) and $\lambda(r)$ (Proposition~\ref{prop:lambda:holder}), as well as prove an auxiliary result about the Wiener--Hopf factors of the Rogers function $f(\xi)$ (Proposition~\ref{prop:r:quot:bd}). The last property mentioned above has a very natural and seemingly simple statement, and a surprisingly technical proof.

Strictly speaking, Proposition~\ref{prop:pt:quot} is stated in terms of the Wiener--Hopf factors $f^+(r; \xi)$ and $f^-(r; \eta)$ of the Rogers function
\formula{
 f(r; \xi) & = \frac{(\xi - \zeta(r)) (\xi + \overline{\zeta(r)})}{f(\xi) - \lambda(r)} \, .
}
In Definition~\ref{def:eig} we define the generalised eigenfunctions $F_+(r; y)$ and $F_-(r; x)$ by providing a formula for their Laplace transforms in terms of $f^+(r; \xi)$ and $f^-(r; \eta)$. Next, a minor extension of Proposition~\ref{prop:pt:quot} to complex (rather than real) $\xi, \eta$ is stated in terms of $\laplace F_+(r; \xi)$ and $\laplace F_-(r; \eta)$ in Proposition~\ref{prop:pt:lap}, and this is tantamount to Theorem~\ref{thm:heat:laplace}. In the final part of Section~\ref{sec:heat}, we give two auxiliary estimates of the eigenfunctions (Lemma~\ref{lem:eig:est}) and their Laplace transforms (Lemma~\ref{lem:eig:lbound}).

\medskip

In Section~\ref{sec:inf:pt}, we prove Theorems~\ref{thm:heat:spectral} and~\ref{thm:heat}. As this is quite technical, we begin with a simplified variant of Theorem~\ref{thm:heat}, given as Proposition~\ref{prop:pt:bounded}. For the proof of the general result, we need an auxiliary estimate of Laplace transforms of test functions $u(y)$ and $v(x)$ (Lemma~\ref{lem:uv:est}). Theorem~\ref{thm:heat:spectral} is restated as Proposition~\ref{prop:pt:uv}. The rather lengthy and technical proof essentially boils down to contour deformation arguments and various estimates.

As explained above, Theorem~\ref{thm:heat}, restated as Corollary~\ref{cor:pt:log}, follows now easily by Fubini's theorem and a density argument; the details are somewhat complicated, though. In particular, we need to show continuity of both sides of~\eqref{eq:heat} in order to have equality everywhere rather than almost everywhere.

The proof of Theorem~\ref{thm:extrema} is given in Section~\ref{sec:inf:inf}. Since it is similar to the proof of Theorem~\ref{thm:heat}, we only sketch some parts of the argument. There are, however, important differences. In fact, we study first the integral
\formula[eq:inf:lpt]{
 & \int_0^\infty e^{-\xi y} p_t^+(x, y) dy
}
for $\xi > 0$, and only then we consider the limit as $\xi \to 0^+$ in order to recover the expression for
\formula{
 \pr(\underline{X}_t > -x) & = \int_0^\infty p_t^+(x, y) dy .
}
The first part requires a result halfway between Theorems~\ref{thm:heat:laplace} and~\ref{thm:heat:spectral}, given in Proposition~\ref{prop:lpt:uv}. The expression for the Laplace transform~\eqref{eq:inf:lpt} is given in Corollary~\ref{cor:lpt:log}, and the limit as $\xi \to 0^+$ is discussed in Corollary~\ref{cor:inf:log}, which is a reformulation of Theorem~\ref{thm:extrema}\ref{thm:extrema:b}.

\medskip

Some our results involve more than one Rogers function. For this reason, with some exceptions, we generally indicate the dependence on $f$ by writing it explicitly in the subscript. Thus, in the remaining part of the text we tend to write, for example, $\Gamma_f, Z_f, \zeta_f(r), \lambda_f(r), F_{f+}(r; y), F_{f-}(r; x)$, instead of $\Gamma, Z, \zeta(r), \lambda(r), F_+(r; y), F_-(r; x)$ used in the introduction.

\medskip

We conclude the introduction with a brief description of the structure of the paper. The remaining part of the article consists of five sections. In Preliminaries, we fix the notation and discuss the notions of Stieltjes and complete Bernstein functions. Section~\ref{sec:rogers} is devoted to Rogers function, that is, holomorphic extensions of characteristic exponents of Lévy processes with completely monotone jumps. We recall the relevant results from~\cite{kwasnicki:rogers} and prove additional auxiliary lemmas. The contents of Sections~\ref{sec:heat} and~\ref{sec:inf}, where proofs of our main results are given, is discussed above. We conclude the article with a number of examples in Section~\ref{sec:ex}.

\subsection*{Acknowledgment}

The author expresses his gratitude to Pierre Patie and Mladen Savov for their comments to a preliminary version of this article, and to Zbigniew Palmowski for pointing out known formulae for the classical risk process. The author also thanks two anonymous referees for exceptionally helpful reports and many valuable suggestions.

%
%

\section{Preliminaries}
\label{sec:pre}

\subsection{Basic notation}

We use the standard notation $\R$ and $\C$ for the sets of real and complex numbers, and $\re \xi$ and $\im \xi$ for the real and imaginary part of $\xi \in \C$. We use $\Arg \xi \in (-\pi, \pi)$ for the principal argument of $\xi \in \C \setminus (-\infty, 0]$, and $\log \xi = \log |\xi| + i \Arg \xi$ for the principal branch of the complex logarithm. We denote by
\formula{
 \hp & = \{ \xi \in \C : \re \xi > 0 \}
}
the right complex half-plane. We write $i \R$ for the imaginary axis, we let $[\xi, \eta]$ be the interval in the complex plane with endpoints $\xi, \eta \in \C$, and if $\xi \in \C$ and $\thet \in \R$, then $[\xi, e^{i \thet} \infty)$ denotes the ray $\xi + r e^{i \thet}$, where $r \in [0, \infty)$.

To avoid confusion with complex conjugation, we denote the closure of a set $A$ by $\Cl A$. The interior and boundary of $A$ are denoted by $\Int A$ and $\partial A$, respectively.

Following~\cite{kwasnicki:rogers}, we generally use $x, y$ for spatial variables, $\eta, \xi$ for the corresponding Fourier or Laplace variables, $t \ge 0$ for a temporal variable, and $\sigma$ for the corresponding Laplace variable. Whenever this causes no confusion, we write explicitly the arguments of a function or a measure, for example, we usually write `function $f(\xi)$' rather than `function $f$', or `process $X_t$' rather than `process $X$'.

\subsection{Completely monotone, Stieltjes and complete Bernstein functions}

We briefly recall the definitions and basic properties of three classes of functions commonly used throughout the article.

A function $f(x)$ on $(0, \infty)$ is said to be \emph{completely monotone} if $f$ is smooth and $(-1)^n f^{(n)}(x) \ge 0$ for all $x \in (0, \infty)$ and $n = 0, 1, \ldots$, where $f^{(n)}$ is the $n$th derivative of $f$; see Chapter~1 in~\cite{ssv}. By Bernstein's theorem, $f$ is completely monotone if and only if $f(x)$ is the Laplace transform of a Borel measure on $[0, \infty)$:
\formula[eq:cm:int]{
 f(x) & = \int_{[0, \infty)} e^{-s x} \mu(ds)
}
for a Borel measure $\mu$ on $[0, \infty)$ such that the above integral is finite for every $x > 0$. Completely monotone functions thus extend to holomorphic functions in $\hp$.

A function $f(\xi)$, defined initially on $(0, \infty)$, is a \emph{Stieltjes function} if there are constants $b, c \ge 0$ and a Borel measure $\mu$ on $(0, \infty)$ such that $\int_{(0, \infty)} \min(1, s^{-1}) \mu(ds) < \infty$ and
\formula[eq:s:int]{
 f(\xi) & = \frac{b}{\xi} + c + \frac{1}{\pi} \int_{(0, \infty)} \frac{1}{\xi + s} \, \mu(ds)
}
for all $\xi \in \C \setminus (-\infty, 0]$; see Chapter~2 in~\cite{ssv}. Note that if $c = 0$ and $\mu(\{0\}) = b$, then $f(\xi)$ is the Laplace transform of the completely monotone function given by the right-hand side of~\eqref{eq:cm:int}. Thus, up to addition by a non-negative constant, Stielties functions are Laplace transforms of Laplace transforms of Borel measures. Note that Stieltjes functions automatically extend to holomorphic functions on $\C \setminus (-\infty, 0]$. If $f(\xi)$ is given by~\eqref{eq:s:int}, then
\formula{
 b & = \lim_{\xi \to 0^+} \xi f(\xi) , & c & = \lim_{\xi \to \infty} f(\xi) ,
}
and, in the sense of vague convergence of measures on $(0, \infty)$,
\formula{
 \mu(ds) & = \lim_{t \to 0^+} (-\im f(-s + i t) ds) .
}

A function $f(\xi)$, defined initially on $(0, \infty)$, is a \emph{complete Bernstein function}, if there are constants $b, c \ge 0$ and a Borel measure $\mu$ on $(0, \infty)$ such that $\int_{(0, \infty)} \min(s^{-1}, s^{-2}) \mu(ds) < \infty$ and
\formula[eq:cbf:int]{
 f(\xi) & = b \xi + c + \frac{1}{\pi} \int_{(0, \infty)} \frac{\xi}{\xi + s} \, \frac{\mu(ds)}{s}
}
for all $\xi \in \C \setminus (-\infty, 0]$; see Chapter~6 in~\cite{ssv}. As before, every complete Bernstein function extends to a holomorphic function in $\C \setminus (-\infty, 0]$. For a function $f(\xi)$ given by~\eqref{eq:cbf:int}, we have
\formula{
 b & = \lim_{\xi \to 0^+} f(\xi) , & c & = \lim_{\xi \to \infty} \frac{f(\xi)}{\xi} \, ,
}
and, in the sense of vague convergence of measures on $(0, \infty)$,
\formula{
 \mu(ds) & = \lim_{t \to 0^+} (\im f(-s + i t) ds) .
}

We remark that $f(\xi)$ is a complete Bernstein function if and only if $f(\xi) / \xi$ is a Stieltjes function, and if $f$ is not identically zero, then $f(\xi)$ is a complete Bernstein function if and only if $1 / f(\xi)$ is a Stieltjes function. Additionally, a holomorphic function $f(\xi)$ in $\C \setminus (-\infty, 0]$ is a complete Bernstein function if and only if $f(\xi) \ge 0$ for $\xi \in (0, \infty)$ and $\im f(\xi) \ge 0$ for all $\xi \in \C$ such that $\im \xi > 0$.

In particular, if $f(\xi)$ is a complete Bernstein function, then $\xi / f(\xi)$ is a complete Bernstein function, and hence $\im (\xi / f(\xi)) \ge 0$ when $\im \xi > 0$, and $\im (\xi / f(\xi)) \le 0$ when $\im \xi < 0$. In other words,
\formula*[eq:cbf:arg]{
 0 & \le \Arg f(\xi) \le \Arg \xi && \text{when $\im \xi > 0$,} \\
 0 & \ge \Arg f(\xi) \ge \Arg \xi && \text{when $\im \xi < 0$.}
}

For a detailed discussion of the classes of functions introduced above, we refer to~\cite{ssv}. A summary of properties related to the present context is given in~\cite{kwasnicki:rogers}.

\subsection{Lévy processes}

A \emph{Lévy process} is a stochastic process $X_t$ with independent and stationary increments, and càdlàg paths. More formally, we assume that $X_t$, where $t \in [0, \infty)$, is a collection of random variables such that the distribution of the increment $X_{s + t} - X_s$ does not depend on $s \ge 0$, increments over disjoint intervals are independent random variables, and the paths $t \mapsto X_t$ are right-continuous and have left limits.

It is customary to assume that $X_0 = 0$ with probability one. However, we will work with an arbitrary starting point $X_0 = x \in \R$, and denote the corresponding probability and expectation by $\pr^x$ and $\ex^x$. We additionally allow $X_t$ to be \emph{killed} at a uniform rate $c \ge 0$. By this we mean that we augment the state space $\R$ by an additional \emph{cemetery point} $\partial$, and, given $s > 0$, with probability $1 - e^{c s}$ we have $X_{s + t} = \partial$ for all $t \in [0, \infty)$.

A Lévy process is completely determined by its \emph{characteristic exponent}: a function $f(\xi)$, defined initially on $\R$, such that
\formula{
 \ex^0 e^{i \xi X_t} & = e^{-t f(\xi)} .
}
The characteristic exponent is given by the \emph{Lévy--Khintchine formula}: we have
\formula{
 f(\xi) & = a \xi^2 - i b \xi + c + \int_{\R \setminus \{0\}} \bigl(1 - e^{i \xi z} + i \xi z \ind_{(-1, 1)}(z)\bigr) \nu(dx) ,
}
where $a \ge 0$ is the \emph{Gaussian coefficient}, $b \in \R$ corresponds to the \emph{drift} of the process, $c \ge 0$ is the \emph{killing rate}, and the \emph{Lévy measure} $\nu(dx)$ satisfies the integrability condition $\int_{\R \setminus \{0\}} \min\{1, x^2\} \nu(dx) < \infty$ and describes the intensity of jumps of $X_t$.

We remark that most authors denote the Gaussian coefficient $a$ by $\tfrac{1}{2} \sigma^2$, and some use the symbols $\mu$ and $q$ for the drift term $b$ and the killing rate $c$. There is no common notation for Lévy measure, various references use, for example, $\nu$, $\pi$, $\Pi$ or $J$. In most applications one has $c = 0$, and in many other situations assuming that $c = 0$ leads to no loss of generality. This will be the case here, and in Sections~\ref{sec:heat} and~\ref{sec:inf} we will restrict our attention to non-killed Lévy processes with $c = 0$.

If the Lévy measure $\nu(dx)$ has a density function $\nu(x)$ such that $\nu(x)$ and $\nu(-x)$ are completely monotone functions of $x \in (0, \infty)$, then we say that $X_t$ has \emph{completely monotone jumps}. Modifying the drift coefficient $b$ appropriately, we can rewrite the Lévy--Khintchine formula as
\formula{
 f(\xi) & = a \xi^2 - i b \xi + c + \int_{-\infty}^\infty \bigl(1 - e^{i \xi x} + i \xi (1 - e^{-|x|}) \sign x \bigr) \nu(x) dx ,
}
and by using Bernstein's theorem for $\nu(x)$ and $\nu(-x)$, we find that
\formula{
 f(\xi) & = a \xi^2 - i b \xi + c + \frac{1}{\pi} \int_{\R \setminus \{0\}} \biggl(\frac{\xi}{\xi + i s} + \frac{i \xi \sign s}{1 + |s|}\biggr) \frac{\mu(ds)}{|s|} \, ,
}
where $\mu$ is a Borel measure on $\R \setminus \{0\}$ such that $\int_{\R \setminus \{0\}} |s|^{-3} \min\{1, s^2\} \mu(ds) < \infty$; see Theorem~3.3 in~\cite{kwasnicki:rogers}.

In particular, it follows that characteristic exponents of Lévy processes with completely monotone jumps extend to holomorphic functions in $\C \setminus i \R$, which we call \emph{Rogers functions}, following~\cite{kwasnicki:rogers}. We refer to that article for a more detailed discussion of Lévy processes with completely monotone jumps, and to Section~\ref{sec:rogers} for a summary of properties of Rogers functions.

Recall that the \emph{transition probabilities} of $X_t$ are defined by $p_t(x, A) = \pr^x(X_t \in A)$. Since $X_t$ has stationary increments, the transition probabilities are translation invariant, in the sense that $p_t(x, A) = p_t(0, A - x)$. We denote by $\tau_{(0, \infty)}$ the first exit time from $(0, \infty)$, and by $p_t(x, A)$ the transition probabilities of the \emph{killed process} $X_t$ on $(0, \infty)$:
\formula{
 p_t^+(x, A) & = \pr^x(X_t \in A, \, t < \tau_{(0, \infty)}) .
}
If $p_t(x, A)$ or $p_t^+(x, A)$ has a density function, we denote it by the same symbol $p_t(x, y) = p_t(0, y - x)$ and $p_t^+(x, y)$. In analysis, $p_t(x, y)$ is said to be the \emph{heat kernel} for the generator $L$ of $X_t$, and $p_t^+(x, y)$ is the heat kernel in $(0, \infty)$ with zero (or Dirichlet) condition in $(-\infty, 0]$; see Remark~\ref{rem:analysis}.

By $\overline{X}_t = \sup\{X_s : s \in [0, t]\}$ and $\underline{X}_t = \inf\{X_s : s \in [0, t]\}$ we denote the supremum and infimum functionals. The events $\tau_{(0, \infty)} > t$ and $\underline{X}_t > 0$ differ by a set of probability zero, and thus
\formula{
 \pr^0(\underline{X}_t > -x) & = \pr^x(\underline{X}_t > 0) = \pr^x(\tau_{(0, \infty)} > t) = p_t^+(x, (0, \infty)) .
}

As it is customary, we denote by $\kappa^+(\sigma, \xi)$, $\kappa^-(\sigma, \eta)$ the bivariate Laplace exponents of ladder processes corresponding to $X_t$, defined by 
\formula*[eq:kappa]{
 \kappa^+(\sigma, \xi) & = \exp\biggl(\int_0^\infty \int_{(0, \infty)} \frac{e^{-t} - e^{-\sigma t - \xi x}}{t} \, \pr(X_t \in dx) dt\biggr) , \\
 \kappa^-(\sigma, \xi) & = \exp\biggl(\int_0^\infty \int_{(-\infty, 0)} \frac{e^{-t} - e^{-\sigma t + \xi x}}{t} \, \pr(X_t \in dx) dt\biggr) ;
}
see~\cite{fristedt,pr,rogozin}. These functions are closely related to Wiener--Hopf factorisation of the characteristic exponent $f(\xi)$, and thus they are often called the \emph{Wiener--Hopf factors} of the process $X_t$. The above definition, with inner integrals over $(-\infty, 0)$ and $(0, \infty)$, follows the convention used in~\cite{kwasnicki:rogers}, where the missing integral over $\{0\}$ is denoted by
\formula{
 \kappa^\circ(\sigma) & = \exp\biggl(\int_0^\infty \frac{e^{-t} - e^{-\sigma t}}{t} \, \pr(X_t = 0) dt\biggr) .
}
Of course, $\kappa^\circ(\sigma) = 1$, unless $X_t$ is a compound Poisson process (or, equivalently, $f$ is bounded on $\R$).

We only use the functions $\kappa^+(\sigma, \xi)$ and $\kappa^-(\sigma, \eta)$ in Proposition~\ref{prop:pt:pr}, where we apply the \emph{Pecherski--Rogozin identities} (see~\cite{fristedt,pr,rogozin})
\formula[eq:pr]{
 \ex^0 \exp(-\xi \ol{X}_S) & = \frac{\kappa^+(\sigma, 0)}{\kappa^+(\sigma, \xi)} \, , & \ex^0 \exp(\eta \ul{X}_S) & = \frac{\kappa^-(\sigma, 0)}{\kappa^-(\sigma, \eta)}
}
where $\re \xi \ge 0$, $\re \eta \ge 0$, and the factorisation identity (see formula~(2.7) in~\cite{kwasnicki:rogers})
\formula[eq:pr:fact]{
 \kappa^\circ(\sigma) \kappa^+(\sigma, -i \xi) \kappa^-(\sigma, i \xi) & = \frac{\sigma + f(\xi)}{1 + f(0)} \, ,
}
where $\sigma > 0$ and $\xi \in \R$. Then we immediately replace the Wiener--Hopf factors $\kappa^+(\sigma, \xi)$ and $\kappa^-(\sigma, \eta)$ of $X_t$ with the analytical Wiener--Hopf factors $f_\sigma^+(\xi)$ and $f_\sigma^-(\eta)$ of the function $f_\sigma(\xi) = \sigma + f(\xi)$. In particular, we will never need the definition~\eqref{eq:kappa}. For this reason, we stop our discussion of $\kappa^+(\sigma, \xi)$ and $\kappa^-(\sigma, \eta)$ here, and we refer to~\cite{kwasnicki:rogers} and the references given there for further details.

\subsection{Auxiliary estimates}

The following simple inequality is used a few times below. If $\thet = \lvert\Arg (\xi / \eta)\rvert$, then
\formula{
 |\xi - \eta|^2 & = |\xi|^2 + |\eta|^2 - 2 \re(\xi \overline{\eta}) \\
 & = |\xi|^2 + |\eta|^2 - 2 |\xi \eta| \cos \thet \\
 & = (\sin \tfrac{\thet}{2})^2 (|\xi| + |\eta|)^2 + (\cos \tfrac{\thet}{2})^2 (|\xi| - |\eta|)^2 \\
 & \ge (\sin \tfrac{\thet}{2})^2 (|\xi| + |\eta|)^2 .
}
Thus,
\formula[eq:triangle]{
 |\xi - \eta| & \ge (|\xi| + |\eta|) \sin \tfrac{\thet}{2}
}
We also note that if $\xi \in \C \setminus (-\infty, 0]$, then
\formula{
 \frac{\re \xi}{|\xi|} & = \cos \Arg \xi ,
}
and that
\formula[eq:coscos]{
 \cos(\alpha + \tfrac{\pi}{4}) \cos(\alpha - \tfrac{\pi}{4}) & = \tfrac{1}{2} \cos(2 \alpha) .
}

%
%

\section{Rogers functions}
\label{sec:rogers}

The term \emph{Rogers function} was introduced in the unpublished paper~\cite{kwasnicki:pre} as a name for a class of functions introduced by L.C.G.~Rogers in~\cite{rogers}. The core of~\cite{kwasnicki:pre} appeared recently in~\cite{kwasnicki:rogers}, which contains a detailed analysis of the Wiener--Hopf factorisation of Rogers functions. Below we recall the definition and some properties of Rogers functions, and we prove several auxiliary lemmas.

\subsection{Definition and basic properties}

We begin with the definition of a Rogers function.

\begin{definition}[Definition~3.2 in~\cite{kwasnicki:rogers}]
A function $f(\xi)$ is a \emph{Rogers function} if it is a holomorphic function in the right complex half-plane $\hp = \{\xi \in \C : \re \xi > 0\}$ such that $\re (f(\xi) / \xi) \ge 0$ whenever $\re \xi > 0$.
\end{definition}

A Rogers function $f(\xi)$ is said to be \emph{non-zero} if $f(\xi)$ is not identically equal to zero in $\hp$; $f(\xi)$ is \emph{non-constant} if $f(\xi)$ is not a constant function in $\hp$; finally, $f(\xi)$ is \emph{non-degenerate} if $f(\xi)$ is not of the form $i b \xi$ for some $b \in \R$.

Every characteristic function of a Lévy process with completely monotone jumps (possibly killed at a uniform rate) extends to a Rogers function, and every Rogers function corresponds in this way to some Lévy process with completely monotone jumps. More precisely, we have the following equivalent characterisations of the class of Rogers functions.

\begin{theorem}[Theorem~3.3 in~\cite{kwasnicki:rogers}]
\label{thm:rogers}
Suppose that $f(\xi)$ is a continuous function on~$\R$, satisfying $f(-\xi) = \overline{f(\xi)}$ for all $\xi \in \R$. The following conditions are equivalent:
\begin{enumerate}[label=\rm (\alph*)]
\item\label{it:r:a}
$f(\xi)$ extends to a Rogers function;
\item\label{it:r:b}
$f(\xi)$ is the characteristic exponent of a Lévy process with completely monotone jumps, possibly killed at a uniform rate;
\item\label{it:r:c}
we have
\formula[eq:r:int]{
 f(\xi) & = a \xi^2 - i b \xi + c + \frac{1}{\pi} \int_{\R \setminus \{0\}} \biggl(\frac{\xi}{\xi + i s} + \frac{i \xi \sign s}{1 + |s|}\biggr) \frac{\mu(ds)}{|s|}
}
for all $\xi \in \R$, where $a \ge 0$, $b \in \R$, $c \ge 0$ and $\mu(ds)$ is a Borel measure on $\R \setminus \{0\}$ such that $\int_{\R \setminus \{0\}} |s|^{-3} \min\{1, s^2\} \mu(ds) < \infty$;
\item\label{it:r:d}
either $f(\xi) = 0$ for all $\xi \in \R$ or
\formula[eq:r:exp]{
 f(\xi) & = c \exp\biggl(\frac{1}{\pi} \int_{-\infty}^\infty \biggl(\frac{\xi}{\xi + i s} - \frac{1}{1 + |s|}\biggr) \frac{\ph(s)}{|s|} \, ds\biggr)
}
for all $\xi \in \R$, where $c > 0$ and $\ph(s)$ is a Borel function on $\R$ with values in $[0, \pi]$.
\end{enumerate}
\end{theorem}

We say that~\eqref{eq:r:int} is a \emph{Stieltjes representation} of a Rogers function $f(\xi)$, while~\eqref{eq:r:exp} gives an \emph{exponential representation} of a (non-zero) Rogers function $f(\xi)$. If $f(\xi)$ is a non-zero Rogers function with exponential representation~\eqref{eq:r:exp}, then we define the \emph{domain} of $f(\xi)$ by the formula
\formula{
 D_f & = \C \setminus (-i \esssupp \ph) ,
}
where $\esssupp \ph$ denotes the essential support of $\ph$, that is, the smallest closed set $A$ such that $\ph = 0$ almost everywhere on $\R \setminus A$; see formula~(3.3) in~\cite{kwasnicki:rogers}. By Remark~3.4(a) in~\cite{kwasnicki:rogers}, $f(\xi)$ extends to a holomorphic function in $D_f$, denoted by the same symbol $f(\xi)$, and this extension satisfies $f(-\overline{\xi}) = \overline{f(\xi)}$.

We will use the following results from~\cite{kwasnicki:rogers}. The first one was given there without proof, so we include the details below.

\begin{proposition}[Propostion~3.12 in~\cite{kwasnicki:rogers}]
\label{prop:r:prop}
For all Rogers functions $f(\xi)$ and $g(\xi)$:
\begin{enumerate}[label=\rm (\alph*)]
\item\label{it:r:prop:a} $\xi^2 f(1 / \xi)$ is a Rogers function;
\item\label{it:r:prop:b} $\xi^2 / f(\xi)$ and $1 / f(1 / \xi)$ are Rogers functions if $f(\xi)$ is non-zero;
\item\label{it:r:prop:d} $g(\xi) f(\xi / g(\xi))$ is a Rogers function if $g(\xi)$ is non-zero;
\item\label{it:r:prop:g} $a f(b \xi) + c$ is a Rogers function if $a, b, c \ge 0$.
\end{enumerate}
\end{proposition}

\begin{proof}
Suppose that $\re (f(\xi) / \xi) \ge 0$ and $\re (g(\xi) / \xi) \ge 0$ when $\re \xi > 0$, and define $f_1(\xi) = \xi^2 f(1 / \xi)$, $f_2(\xi) = \xi^2 / f(\xi)$, $f_3(\xi) = 1 / f(1 / \xi)$, $f_4(\xi) = g(\xi) f(\xi / g(\xi)$ and $f_5(\xi) = a f(b \xi) + c$. If $\re \xi > 0$, then $\re (1 / \xi) > 0$, and hence
\formula{
 \re \frac{f_1(\xi)}{\xi} & = \re \frac{f(1 / \xi)}{1 / \xi} \ge 0 .
}
If $f(\xi)$ is non-zero, then additionally
\formula{
 \re \frac{f_2(\xi)}{\xi} & = \re \frac{\xi}{f(\xi)} \ge 0 , & \re \frac{f_3(\xi)}{\xi} & = \re \frac{1 / \xi}{f(1 / \xi)} \ge 0 .
}
Similarly, if $g(\xi)$ is non-zero, then $\re(\xi / g(\xi)) > 0$, and so
\formula{
 \re \frac{f_4(\xi)}{\xi} & = \re \frac{f(\xi / g(\xi))}{\xi / g(\xi)} \ge 0 .
}
Finally, $\re(f_5(\xi) / \xi) = a b \re(f(b \xi) / (b \xi)) + c \re(1 / \xi) \ge 0$.
\end{proof}

\begin{proposition}[Proposition~3.14 in~\cite{kwasnicki:rogers}]
\label{prop:r:limit}
If $f(\xi)$ is a Rogers function, then the limit $f(0^+) = \lim_{\xi \to 0^+} f(\xi)$ exists. More precisely, if $f(\xi)$ has Stieltjes representation~\eqref{eq:r:int}, then $f(0^+) = c$, and if $f(\xi)$ has the exponential representation~\eqref{eq:r:exp}, then
\formula[eq:r:limit:zero]{
 f(0^+) & = c \exp \biggl( -\frac{1}{\pi} \int_0^\infty \frac{1}{1 + |s|} \, \frac{\ph(s)}{|s|} \, ds \biggr) ,
}
where we understand that $\exp(-\infty) = 0$. Similarly, $f(\infty^-) = \lim_{\xi \to \infty} f(\xi)$ exists, and if $f(\xi)$ has the exponential representation~\eqref{eq:r:exp}, then
\formula[eq:r:limit:inf]{
 f(\infty^-) & = c \exp \biggl( \frac{1}{\pi} \int_0^\infty \frac{|s|}{1 + |s|} \, \frac{\ph(s)}{|s|} \, ds \biggr) ,
}
where we understand that $\exp(\infty) = \infty$
\end{proposition}

\begin{proposition}[Proposition~3.17 in~\cite{kwasnicki:rogers}]
\label{prop:r:bound}
If $f(\xi)$ is a Rogers function and $r > 0$, then
\formula[eq:r:bound]{
 \frac{1}{\sqrt{2}} \, \frac{|\xi|^2}{r^2 + |\xi|^2} \biggl(\frac{\re \xi}{|\xi|}\biggr) |f(r)| \le |f(\xi)| & \le \sqrt{2} \, \frac{r^2 + |\xi|^2}{r^2} \biggl(\frac{|\xi|}{\re \xi}\biggr) |f(r)|
}
when $\re \xi > 0$.
\end{proposition}

If $h(\xi)$ is a complete Bernstein function, then $f(\xi) = h(\xi^2)$ is a Rogers function. In this case the above estimate applied to $f(\xi)$ reads as follows.

\begin{corollary}
\label{cor:cbf:bound}
If $h(\xi)$ is a complete Bernstein function, $r > 0$ and $\xi \in \C \setminus (-\infty, 0]$, then
\formula[eq:cbf:bound]{
 \frac{1}{\sqrt{2}} \, \frac{|\xi|}{r + |\xi|} \biggl(\frac{\re \sqrt{\xi}}{\lvert\sqrt{\xi}\rvert}\biggr) h(r) \le |h(\xi)| & \le \sqrt{2} \, \frac{r + |\xi|}{r} \biggl(\frac{\lvert\sqrt{\xi}\rvert}{\re \sqrt{\xi}}\biggr) h(r) .
}
\end{corollary}

The following result is an extension of Proposition~3.18 in~\cite{kwasnicki:rogers}, which corresponds to the case $f(\xi) \in (0, \infty)$ and $p = 0$. Since the proof is exactly the same as that of Proposition~3.18 in~\cite{kwasnicki:rogers}, we omit it.

\begin{proposition}
\label{prop:r:spine:bound}
Suppose that $f(\xi)$ is a non-zero Rogers function and $p \ge 0$. Let $\xi \in D_f$, and suppose that $\Arg f(\xi) \le p \re \xi / |\xi|$ if $\im \xi \ge 0$, or $\Arg f(\xi) \ge -p \re \xi / |\xi|$ if $\im \xi \le 0$. Then
\formula[eq:r:spine:prime]{
 \biggl| \frac{f'(\xi)}{f(\xi)} \biggr| & \le \frac{\pi + 2 p}{|\xi|} \, ,
}
and, for some $c > 0$,
\formula[eq:r:spine:log]{
 \lvert\log f(\xi)\rvert & \le \lvert\log c\rvert + \sqrt{2 (\pi + 2 p)} \, \frac{1 + |\xi|}{\sqrt{|\xi|}} \, .
}
More precisely, $c$ is the constant in the exponential representation~\eqref{eq:r:exp} of $f(\xi)$, and with the notation of~\eqref{eq:r:exp}, we have
\formula[eq:r:spine:bound]{
 \frac{1}{\pi} \int_{-\infty}^\infty \biggl| \frac{1}{\xi + i s} - \frac{1}{1 + |s|} \biggr| \frac{\ph(s)}{|s|} \, ds \le \sqrt{2 (\pi + 2 p)} \, \frac{1 + |\xi|}{\sqrt{|\xi|}} \, .
}
\end{proposition}

The \emph{spine} of $f(\xi)$ is the system of curves
\formula{
 \Gamma_f & = \{\xi \in \hp : f(\xi) \in (0, \infty)\} ;
}
see Definition~4.1 in~\cite{kwasnicki:rogers}. Regularity of the spine is described by the following two results. The author recently learned that part~\ref{it:r:real:d} of Proposition~\ref{prop:r:real} was in fact first observed in~\cite{dh}, with a better constant, and an optimal constant was found in~\cite{steinerberger}.

\begin{proposition}[Theorem~4.2 in~\cite{kwasnicki:rogers}]
\label{prop:r:real}
Let $f(\xi)$ be a non-constant Rogers function. There is a unique continuous complex-valued function $\zeta_f(r)$ on $(0, \infty)$ such that the following assertions hold:
\begin{enumerate}[label=\rm (\alph*)]
\item\label{it:r:real:a} We have $|\zeta_f(r)| = r$ and $\Arg \zeta_f(r) \in [-\tfrac{\pi}{2}, \tfrac{\pi}{2}]$ for all $r > 0$.
\item\label{it:r:real:b} If $\re \xi > 0$ and $r = |\xi|$, then
\formula{
 \sign \im f(\xi) & = \sign (\Arg \xi - \Arg \zeta_f(r)) .
}
\item\label{it:r:real:c}
The spine $\Gamma_f$ is the union of pairwise disjoint simple real-analytic curves, which begin and end at the imaginary axis or at infinity. Furthermore, $\Gamma_f$ has parameterisation
\formula[eq:r:spine]{
 \Gamma_f & = \{\zeta_f(r) : r \in Z_f\} , 
}
where
\formula[eq:r:z]{
 Z_f & = \{r \in (0, \infty) : \Arg \zeta_f(r) \in (-\tfrac{\pi}{2}, \tfrac{\pi}{2})\} .
}
\item\label{it:r:real:d}
For every $r > 0$, the spine $\Gamma_f$ restricted to the annular region $r \le |\xi| \le 2 r$ is a system of rectifiable curves of total length at most $C r$, where one can take $C = 300$. Furthermore, if $\zeta_f(r) = r e^{i \thet(r)}$ for $r \in Z_f$, then
\formula[eq:r:real:curv]{
 |r (r \thet'(r))'| & \le \frac{9 ((r \thet'(r))^2 + 1)}{\cos \thet(r)}
}
for $r \in Z_f$.
\end{enumerate}
\end{proposition}

\begin{proposition}[Theorem~4.3 in~\cite{kwasnicki:rogers}]
\label{prop:r:lambda}
Suppose that $f(\xi)$ is a non-constant Rogers function.
\begin{enumerate}[label=\rm (\alph*)]
\item\label{it:r:lambda:a} For every $r \in (0, \infty) \setminus \partial Z_f$ we have $\zeta_f(r) \in D_f$.
\item\label{it:r:lambda:b} The function $\lambda_f(r)$, defined for $r \in (0, \infty) \setminus \partial Z_f$ by
\formula[eq:r:lambda]{
 \lambda_f(r) & = f(\zeta_f(r)) ,
}
extends in a unique way to a continuous, strictly increasing function of $r \in (0, \infty)$, and $\lambda_f'(r) > 0$ for every $r \in (0, \infty) \setminus \partial Z_f$.
\item\label{it:r:lambda:c} We have $\lambda(0^+) = f(0^+)$, and $\lambda(\infty^-) = f(\infty^-)$.
\end{enumerate}
\end{proposition}

We define the \emph{symmetrised spine} $\Gamma_f^\star$ to be the union of the spine $\Gamma_f$, its mirror image $-\overline{\Gamma}_f$, and the endpoints of $\Gamma_f$ (which necessarily lie on the imaginary axis); see Section~4.2 in~\cite{kwasnicki:rogers}. The symmetrised spine (or, more precisely, its closure) divides the complex plane into two regions (see Figure~\ref{fig:dom}):
\formula{
 D_f^+ & = \Int \{r e^{i \alpha} \in \C : r \ge 0 , \alpha \in [\thet(r), \pi - \thet(r)]\} , \\
 D_f^- & = \Int \{r e^{i \alpha} \in \C : r \ge 0 , \alpha \in [-\pi - \thet(r), \thet(r)]\} ;
}
again, see Section~4.2 in~\cite{kwasnicki:rogers}. On the complement of the imaginary axis, these sets are simply the regions where $\im f(\xi)$ is positive or negative, respectively:
\formula{
 D_f^+ \setminus i \R & = \{\xi \in \C \setminus i \R : \im f(\xi) > 0\} , \\
 D_f^- \setminus i \R & = \{\xi \in \C \setminus i \R : \im f(\xi) < 0\} .
}

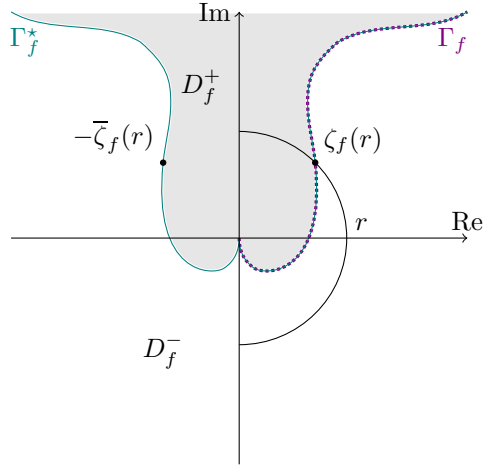
\begin{figure}
\centering
\begin{tikzpicture}
\footnotesize
\coordinate (X) at (3,0);
\coordinate (Y) at (0,3);
\coordinate (Xn) at (-3,0);
\coordinate (Yn) at (0,-3);
\coordinate (O) at (0,0);
\coordinate (E) at (1,2);
\coordinate (zeta) at (1,1);
\coordinate (nzeta) at (-1,1);
\coordinate (modzetax) at (1.4142,0);
\coordinate (modzetay) at (0,1.4142);
\draw[white, very thick, fill=gray!20!white] (O) .. controls (-0.00,-0.15) and (-0.05,-0.3) .. (-0.2,-0.4) .. controls (-0.35,-0.5) and (-1.15,-0.5) .. (-1,1) .. controls (-0.97,1.3) and (-0.9,1.5) .. (-0.9,1.8) .. controls (-0.9,2.1) and (-0.95,2.3) .. (-1.3,2.6) .. controls (-1.8,2.9) and (-2.5,2.7) .. (-3,3) -- (Y);
\draw[white, very thick, fill=gray!20!white] (O) .. controls (0.00,-0.15) and (0.05,-0.3) .. (0.2,-0.4) .. controls (0.35,-0.5) and (1.15,-0.5) .. (1,1) .. controls (0.97,1.3) and (0.9,1.5) .. (0.9,1.8) .. controls (0.9,2.1) and (0.95,2.3) .. (1.3,2.6) .. controls (1.8,2.9) and (2.5,2.7) .. (3,3) -- (Y);
\draw[very thick, violet, densely dotted] (O) .. controls (0.00,-0.15) and (0.05,-0.3) .. (0.2,-0.4) .. controls (0.35,-0.5) and (1.15,-0.5) .. (1,1) .. controls (0.97,1.3) and (0.9,1.5) .. (0.9,1.8) .. controls (0.9,2.1) and (0.95,2.3) .. (1.3,2.6) .. controls (1.8,2.9) and (2.5,2.7) .. (3,3);
\draw[teal] (O) .. controls (0.00,-0.15) and (0.05,-0.3) .. (0.2,-0.4) .. controls (0.35,-0.5) and (1.15,-0.5) .. (1,1) .. controls (0.97,1.3) and (0.9,1.5) .. (0.9,1.8) .. controls (0.9,2.1) and (0.95,2.3) .. (1.3,2.6) .. controls (1.8,2.9) and (2.5,2.7) .. (3,3);
\draw[teal] (O) .. controls (-0.00,-0.15) and (-0.05,-0.3) .. (-0.2,-0.4) .. controls (-0.35,-0.5) and (-1.15,-0.5) .. (-1,1) .. controls (-0.97,1.3) and (-0.9,1.5) .. (-0.9,1.8) .. controls (-0.9,2.1) and (-0.95,2.3) .. (-1.3,2.6) .. controls (-1.8,2.9) and (-2.5,2.7) .. (-3,3);
\draw (modzetay) arc (90:-90:1.4142);
\filldraw (zeta) circle[radius=1pt] node[above right] {$\zeta_f(r)$};
\filldraw (nzeta) circle[radius=1pt] node[above left] {$-\overline{\zeta}_f(r)$};
\node[above right] at (modzetax) {$r$};
\node[violet, below right] at (2.5,2.9) {$\Gamma_f$};
\node[teal, below left] at (-2.5,2.9) {$\Gamma_f^\star$};
\node at (-0.5,2) {$D_f^+$};
\node at (-1,-1.5) {$D_f^-$};
\draw[->] (Xn) -- (X) node[above] {$\re$};
\draw[->] (Yn) -- (Y) node[left] {$\im$};
\end{tikzpicture}
\caption{The spine $\Gamma_f$ (purple dotted line), the symmetrised spine $\Gamma_f^\star$ (teal line), and the regions $D_f^+$ (gray) and $D_f^-$ (white).}
\label{fig:dom}
\end{figure}

\begin{proposition}[Theorem~2 in~\cite{rogers}; Theorem~5.1 in~\cite{kwasnicki:rogers}]
\label{prop:r:wh}
A function $f(\xi)$ holomorphic on $\hp$ is a non-zero Rogers function if and only if it admits a Wiener--Hopf factorisation
\formula[eq:r:wh:fact]{
 f(\xi) & = f^+(-i \xi) f^-(i \xi)
}
for some non-zero complete Bernstein functions $f^+(\xi)$, $f^-(\xi)$ and for all $\xi \in \hp$, or, equivalently, $\xi \in D_f$. The factors $f^+(\xi)$ and $f^-(\xi)$ are defined uniquely, up to multiplication by a constant: for a given pair of Wiener--Hopf factors $f^+(\xi)$ and $f^-(\xi)$, all other pairs are of the form $C f^+(\xi)$ and $C^{-1} f^-(\xi)$ for $C > 0$.
\end{proposition}

By formula~(5.2) in~\cite{kwasnicki:rogers}, if $f(\xi)$ has the exponential representation~\eqref{eq:r:exp}, then the Wiener--Hopf factors are given by
\formula*[eq:r:wh]{
 f^+(\xi) & = c_+ \, \exp \biggl(\frac{1}{\pi} \int_0^\infty \biggl(\frac{\xi}{\xi + s} - \frac{1}{1 + s}\biggr) \frac{\ph(s)}{s} \, ds\biggr) , \\
 f^-(\xi) & = c_- \, \exp \biggl(\frac{1}{\pi} \int_0^\infty \biggl(\frac{\xi}{\xi + s} - \frac{1}{1 + s}\biggr) \frac{\ph(-s)}{s} \, ds\biggr) ,
}
where $c_+, c_- > 0$ satisfy $c_+ c_- = c$. We stress that expressions of the form $f^+(\xi) / f^+(\eta)$, $f^-(\xi) / f^-(\eta)$ and $f^+(\xi) f^-(\eta)$ do not depend on the choice of the pair of Wiener--Hopf factors in Proposition~\ref{prop:r:wh}.

The next result is Theorem~5.7 in~\cite{kwasnicki:rogers} applied with $R = 0$ to the Rogers function $f_\sigma(\xi) = \sigma + f(\xi)$, combined with Lemma~6.1 in~\cite{kwasnicki:rogers}.

\begin{proposition}[Theorem~5.7 and Lemma~6.1 in~\cite{kwasnicki:rogers}]
\label{prop:r:wh:bd}
If $f(\xi)$ is a non-constant Rogers function such that $f(0^+) = 0$, $\sigma > 0$ and $f_\sigma(\xi) = \sigma + f(\xi)$, then for $\xi, \eta > 0$ we have
\formula{
 f_\sigma^+(\xi) f_\sigma^-(\eta) & = \sigma \exp\biggl(\frac{1}{\pi} \int_0^\infty \bigl( \Arg(\zeta_f(r) + i \eta) - \Arg(\zeta_f(r) - i \xi) \bigr) \, \frac{d\lambda_f(r)}{\sigma + \lambda_f(r)} \biggr) ,
}
and the above expression defines a complete Bernstein function of $\sigma$. Similarly, if $0 < \xi < \eta$, then
\formula{
 \frac{f_\sigma^+(\xi)}{f_\sigma^+(\eta)} & = \exp\biggl(\frac{1}{\pi} \int_0^\infty \bigl(\Arg(\zeta_f(r) - i \eta) - \Arg(\zeta_f(r) - i \xi)\bigr) \, \frac{d\lambda_f(r)}{\sigma + \lambda_f(r)} \biggr) , \\
 \frac{f_\sigma^-(\xi)}{f_\sigma^-(\eta)} & = \exp\biggl(\frac{1}{\pi} \int_0^\infty \bigl(\Arg(\zeta_f(r) + i \xi) - \Arg(\zeta_f(r) + i \eta)\bigr) \, \frac{d\lambda_f(r)}{\sigma + \lambda_f(r)} \biggr) ,
}
and both expressions define a complete Bernstein function of $\sigma$.
\end{proposition}

\subsection{Extensions and auxiliary results}

We will need a few more properties of Rogers functions. The first one is a simple observation.

\begin{proposition}
\label{prop:r:inc}
If $f(\xi)$ is a Rogers function, then $|f(\xi)|$ is a non-decreasing function of $\xi > 0$.
\end{proposition}

\begin{proof}
Suppose that $f(\xi)$ is a non-zero Rogers function with exponential representation~\eqref{eq:r:exp}. Since $\re(\xi / (\xi + i s)) = \xi^2 / (\xi^2 + s^2)$ for $\xi > 0$ and $s \in \R$, we have
\formula{
 \log |f(\xi)| & = \re \Arg f(\xi) = \log c + \frac{1}{\pi} \int_{-\infty}^\infty \biggl(\frac{\xi^2}{\xi^2 + s^2} - \frac{1}{1 + |s|}\biggr) \frac{\ph(s)}{|s|} \, ds
}
for $\xi > 0$, and the right-hand side is clearly a non-decreasing function of $\xi$.
\end{proof}

Our next result extends Theorem~5.5 in~\cite{kwasnicki:rogers}. Its proof is exactly the same as that of Theorem~5.5 in~\cite{kwasnicki:rogers}, with one modification: where originally Proposition~3.18 in~\cite{kwasnicki:rogers} is used to justify the use of Fubini's theorem, one should apply its extension given in Proposition~\ref{prop:r:spine:bound} above. We omit the details.

\begin{proposition}
\label{prop:r:bd2}
If $f(\xi)$ is a non-constant and non-degenerate Rogers function, $h(\xi)$ is a non-zero Rogers function, $\xi_1, \xi_2 \in D_f^+ \cup D_f^-$, and for some constant $C$ we have $\lvert\Arg h(\xi)\rvert \le C \re \xi / |\xi|$ for all $\xi \in \Gamma_f$, then
\formula*[eq:r:bd2]{
 \hspace*{7em} & \hspace*{-7em} \frac{1}{2 \pi i} \int_{\Gamma_f^\star} \biggl( \frac{1}{z - \xi_1} - \frac{1}{z - \xi_2} \biggr) \log h(z) dz \\
 & = \begin{cases}
  \log h^+(-i \xi_1) - \log h^+(-i \xi_2) & \text{if $\xi_1, \xi_2 \in D_f^+$,} \\ 
  \log h^-(i \xi_2) - \log h^-(i \xi_1) & \text{if $\xi_1, \xi_2 \in D_f^-$,} \\ 
  \log h^+(-i \xi_1) + \log h^-(i \xi_2) & \text{if $\xi_1 \in D_f^+$, $\xi_2 \in D_f^-$;}
 \end{cases}
}
in particular, the integral is absolutely convergent.
\end{proposition}

If $Z_f = (0, \infty)$, then $\zeta_f(r)$ and $\lambda_f^{-1}(s)$ are smooth functions. In the general case, we will need the following two technical results.

\begin{proposition}
\label{prop:zeta:holder}
If $f(\xi)$ is a non-constant Rogers function, then $\zeta_f(r)$ is a locally Hölder continuous function of $r \in (0, \infty)$, with exponent $\tfrac{1}{30}$.
\end{proposition}

\begin{proof}
We use the notation and results of Section~7 in~\cite{kwasnicki:rogers}. We write $r = e^R$, and we let $\Theta(R) = \Arg \zeta_f(e^R)$. By Lemma~7.1 in~\cite{kwasnicki:rogers}, we have
\formula[eq:zeta:holder:71]{
 |\Theta''(R)| & \le \frac{9 ((\Theta'(R))^2 + 1)}{\cos \Theta(R)}
}
whenever $e^R \in Z_f$, or, equivalently, when $\Theta(R) \in (-\tfrac{\pi}{2}, \tfrac{\pi}{2})$. Furthermore, by Lemma~7.3 in~\cite{kwasnicki:rogers}, if $\Theta(R_1) = \Theta(R_2) = 0$ and $|\Theta'(R_1)| \ge 1$, then $|R_2 - R_1| > \tfrac{1}{2}$.

Local Hölder continuity of $\zeta_f(r) = e^{i \Theta(\log r)}$ on $(0, \infty)$ is equivalent to local Hölder continuity of $\Theta(R)$ on $\R$. Below we prove the latter property. We begin with the following observation: for every interval $[A, B]$, the set of points $R \in [A, B]$ such that $\Theta(R) = 0$ and $|\Theta'(R)| \ge 1$ is finite (because every pair of such points is more than $\tfrac{1}{2}$ apart), and hence
\formula{
 M(A, B) & = \sup\{|\Theta'(R)| : R \in [A, B], \, \Theta(R) = 0\}
}
is finite for every $A$ and $B$ such that $\Theta(R) = 0$ for some $R \in [A, B]$. We set $M(A, B) = 0$ if $\Theta(R) \ne 0$ for $R \in [A, B]$. Let $p = 9 \pi + 1$. Our goal is to prove that $\Theta(R)$ is Hölder continuous on $[A, B]$ with exponent $\tfrac{1}{p}$, with constant in the Hölder condition determined by $M(A, B)$.

We consider an auxiliary function
\formula{
 \Phi(R) & = (1 - \tfrac{2}{\pi} \Theta(R))^p .
}
We claim that $\Phi(R)$ is locally Lipschitz continuous on the set $\{R \in \R : \Theta(R) \ge 0\}$. Let
\formula{
 \tilde{Z}_f^+ & = \{R \in \R : \Theta(R) \in (0, \tfrac{\pi}{2})\} .
}
Observe that $\cos \Theta(R) > (1 - \tfrac{2}{\pi} \Theta(R))$ for $R \in \tilde{Z}_f^+$, while if $|\Theta'(R)| \ge 1$, then $(\Theta'(R))^2 + 1 \le 2 (\Theta'(R))^2$. Thus, formula~\eqref{eq:zeta:holder:71} implies that if $R \in \tilde{Z}_f^+$ and $|\Theta'(R)| \ge 1$, then
\formula{
 |\Theta''(R)| & < \frac{18 (\Theta'(R))^2}{1 - \tfrac{2}{\pi} \Theta(R)} \, .
}
In this case,
\formula{
 \Phi''(R) & = \frac{4 p (p - 1)}{\pi^2} \, (\Theta'(R))^2 (1 - \tfrac{2}{\pi} \Theta(R))^{p - 2} - \frac{2 p}{\pi} \, \Theta''(R) (1 - \tfrac{2}{\pi} \Theta(R))^{p - 1} \\
 & > \frac{4 p (p - 1)}{\pi^2} \, (\Theta'(R))^2 (1 - \tfrac{2}{\pi} \Theta(R))^{p - 2} - \frac{36 p}{\pi} \, (\Theta'(R))^2 (1 - \tfrac{2}{\pi} \Theta(R))^{p - 2} = 0 .
}
On the other hand,
\formula{
 \Phi'(R) & = -\frac{2 p}{\pi} \, \Theta'(R) (1 - \tfrac{2}{\pi} \Theta(R))^{p - 1} ,
}
so that if $|\Phi'(R)| \ge \tfrac{2 p}{\pi}$, then $|\Theta'(R)| \ge 1$. The above observations show that
\formula{
 \text{if $R \in \tilde{Z}_f^+$ and $|\Phi'(R)| \ge \tfrac{2 p}{\pi}$, then $\Phi''(R) > 0$.}
}
By the above property, if $\Phi'(R_1) \ge \tfrac{2 p}{\pi}$ for some $R_1 \in \tilde{Z}_f^+$, then $\Phi'(R) \ge \Phi'(R_1)$ in some right neighbourhood of $R_1$, and it follows that there is $R_2$ such that $[R_1, R_2) \sub \tilde{Z}_f^+$, $\Theta(R_2) = 0$ (or, equivalently, $\Phi(R_2) = 1$), and $\Phi''(R) > 0$ for $R \in [R_1, R_2)$. Consequently, $|\Phi'(R_1)| \le |\Phi'(R_2)|$. Similarly, if $\Phi'(R_1) \le -\tfrac{2 p}{\pi}$ for some $R_1 \in \tilde{Z}_f^+$, then $\Phi''(R) > 0$ for $R \in (R_2, R_1]$ for some $R_2$ such that $\Theta(R_2) = 0$, and again $|\Phi'(R_1)| \le |\Phi'(R_2)|$. Furthermore, in both cases $|R_2 - R_1| \le \tfrac{\pi}{2 p}$. It follows that for every interval $[A, B]$ we have
\formula{
 \sup \{|\Phi'(R)| : R \in \tilde{Z}_f^+ \cap [A, B]\} & \le \max\{\tfrac{2 p}{\pi}, M(A - \tfrac{\pi}{2 p}, B + \tfrac{\pi}{2 p})\} .
}
Since $\Phi$ is continuous on $\mathbb R$, we find that $\Phi(R)$ is locally Lipschitz continuous on the closure of $\tilde{Z}_f^+$, as claimed. Consequently, $\Theta(R)$ is locally Hölder continuous with exponent $\tfrac{1}{p}$ on the closure of $\tilde{Z}_f^+$.

A very similar argument involving another auxiliary function $(1 + \tfrac{2}{\pi} \Theta(R))^p$ shows that $\Theta(R)$ is locally Hölder continuous with exponent $\tfrac{1}{p}$ also on the closure of $\tilde{Z}_f^- = \{R \in \R : \Theta(R) \in (-\tfrac{\pi}{2}, 0)\}$, and, consequently, $\Theta(R)$ is locally Hölder continuous with exponent $\tfrac{1}{p}$ on the closure of $\tilde{Z}_f = \{R \in \R : \Theta(R) \in (-\tfrac{\pi}{2}, \tfrac{\pi}{2}) \setminus \{0\}\}$. Since $\Theta(R)$ is continuous on $\R$ and piecewise constant on the complement of $\tilde{Z}_f$, we conclude that $\Theta(R)$ is locally Hölder continuous on all of $\R$, and the proof is complete.
\end{proof}

We conjecture that in fact $\zeta_f(r)$ is locally Hölder continuous with exponent $\tfrac{1}{2}$.

\begin{proposition}
\label{prop:lambda:holder}
If $f(\xi)$ is a non-constant Rogers function, then $\lambda_f^{-1}(s)$ is a locally Hölder continuous function of $s \in (\lambda_f(0^+), \lambda_f(\infty))$, with exponent $\tfrac{1}{3}$.
\end{proposition}

\begin{proof}
Denote $\Phi(r) = \log \lambda_f(r)$. We first prove that $\Phi(r)$ increases fast enough.

For $r$ in $(0, \infty) \setminus \partial Z_f$, which is a dense subset of $(0, \infty)$, we have $\zeta_f(r) \in D_f$ (by Proposition~\ref{prop:r:lambda}\ref{it:r:lambda:a}), and hence $\Phi(r) = \log f(\zeta_f(r))$. By the exponential representation~\eqref{eq:r:exp}, for $\xi \in D_f$ we have
\formula{
 \log f(\xi) & = \log c + \frac{1}{\pi} \int_{-\infty}^\infty \biggl(\frac{\xi}{\xi + i s} - \frac{1}{1 + |s|}\biggr) \frac{\ph(s)}{|s|} \, ds .
}
If $\xi = x + i y$ with $x \ge 0$ and $y \in \R$, then an elementary calculation shows that
\formula{
 \re \log f(\xi) & = \log c + \frac{1}{\pi} \int_{-\infty}^\infty \biggl(\frac{-(y + s) \sign s}{x^2 + (y + s)^2} + \frac{1}{1 + |s|}\biggr) \ph(s) ds , \\
 \im \log f(\xi) & = -\frac{1}{\pi} \int_{-\infty}^\infty \frac{x \sign s}{x^2 + (y + s)^2} \, \ph(s) ds .
}
We set $\xi = \zeta_f(r)$ for some $r \in (0, \infty) \setminus \partial Z_f$. If $r \in Z_f$, then $x > 0$, and consequently $x^{-1} \im \log f(\xi) = 0$, so that
\formula[eq:lambda:holder:aux1]{
 \frac{1}{\pi} \int_{-\infty}^\infty \frac{\sign s}{x^2 + (y + s)^2} \, \ph(s) ds & = 0 \qquad \text{if $x > 0$.}
}
In the other case, if $r \in (0, \infty) \setminus \Cl Z_f$, we have $x = 0$, and the integral in~\eqref{eq:lambda:holder:aux1} need not be equal to zero. However, we claim that
\formula[eq:lambda:holder:aux2]{
 \frac{y}{\pi} \int_{-\infty}^\infty \frac{\sign s}{x^2 + (y + s)^2} \, \ph(s) ds & \ge 0 \qquad \text{if $x = 0$}.
}
Indeed, observe that either $\xi = i r$ or $\xi = -i r$. In the former case, we have $\im \log f(i e^{-i \eps} r) = \Arg f(i e^{-i \eps} r) \le 0$ for $\eps \in (0, \pi)$ (by Proposition~\ref{prop:r:real}\ref{it:r:real:b}), and hence
\formula{
 0 & \le -\lim_{\eps \to 0^+} \frac{\im \log f(i e^{-i \eps} r)}{\re(i e^{-i \eps} r)} \\
 & = \lim_{\eps \to 0^+} \frac{1}{\pi} \int_{-\infty}^\infty \frac{\sign s}{(r \sin \eps)^2 + (r \cos \eps + s)^2} \, \ph(s) ds = \frac{1}{\pi} \int_{-\infty}^\infty \frac{\sign s}{(r + s)^2} \, \ph(s) ds
}
(we used the dominated convergence theorem in the final step). Formula~\eqref{eq:lambda:holder:aux2} follows. A similar argument can be given in the other case, when $\xi = -i r$, and our claim is proved.

We now fix $r_1, r_2 \in (0, \infty) \setminus \partial Z_f$ such that $r_1 < r_2$, and we apply the above expression for $\re \log f(\xi)$ to $\xi = \zeta_f(r_1) = x_1 + i y_1$ and to $\xi = \zeta_f(r_2) = x_2 + i y_2$. This leads to
\formula{
 \Phi(r_2) - \Phi(r_1) & = \log f(\zeta_f(r_2)) - \log f(\zeta_f(r_1)) \\
 & = \frac{1}{\pi} \int_{-\infty}^\infty \biggl(\frac{y_1 + s}{x_1^2 + (y_1 + s)^2} - \frac{y_2 + s}{x_2^2 + (y_2 + s)^2}\biggr) \ph(s) \sign s \, ds \\
 & = \frac{1}{\pi} \int_{-\infty}^\infty \frac{x_2^2 (y_1 + s) - x_1^2 (y_2 + s) + (y_2 - y_1) (y_1 + s) (y_2 + s) }{(x_1^2 + (y_1 + s)^2) (x_2^2 + (y_2 + s)^2)} \, \ph(s) \sign s \, ds .
}
We define
\formula{
 \alpha_1 & = \frac{y_1 (x_1^2 + y_1^2 + x_2^2 + y_2^2) - 2 y_2 (x_1^2 + y_1^2)}{x_2^2 + y_2^2 - x_1^2 - y_1^2} , \\
 \alpha_2 & = \frac{y_2 (x_1^2 + y_1^2 + x_2^2 + y_2^2) - 2 y_1 (x_2^2 + y_2^2)}{x_2^2 + y_2^2 - x_1^2 - y_1^2} .
}
Note that if $x_1 = 0$, then
\formula{
 \frac{\alpha_1}{y_1} & = \frac{y_1^2 + x_2^2 + y_2^2 - 2 y_1 y_2}{x_2^2 + y_2^2 - y_1^2} \ge 0 ,
}
so that $\alpha_1$ and $y_1$ have equal sign. Thus, by~\eqref{eq:lambda:holder:aux2},
\formula{
 \frac{\alpha_1}{\pi} \int_{-\infty}^\infty \frac{\sign s}{x_1^2 + (y_1 + s)^2} \, \ph(s) ds & \ge 0 .
}
If $x_1 > 0$, then the integral in the above expression is zero by~\eqref{eq:lambda:holder:aux1}, and so the above inequality holds trivially. In a similar way,
\formula{
 \frac{\alpha_2}{\pi} \int_{-\infty}^\infty \frac{\sign s}{x_2^2 + (y_2 + s)^2} \, \ph(s) ds & \ge 0 .
}
It follows that
\formula{
 \Phi(r_2) - \Phi(r_1) & \ge \Phi(r_2) - \Phi(r_1) - \frac{\alpha_1}{\pi} \int_{-\infty}^\infty \frac{\sign s}{x_1^2 + (y_1 + s)^2} \, \ph(s) ds \\
 & \hspace*{11em} - \frac{\alpha_2}{\pi} \int_{-\infty}^\infty \frac{\sign s}{x_2^2 + (y_2 + s)^2} \, \ph(s) ds .
}
By an explicit (but tedious) calculation, we find that
\formula{
 & \Phi(r_2) - \Phi(r_1) - \frac{\alpha_1}{\pi} \int_{-\infty}^\infty \frac{\sign s}{x_1^2 + (y_1 + s)^2} \, \ph(s) ds - \frac{\alpha_2}{\pi} \int_{-\infty}^\infty \frac{\sign s}{x_2^2 + (y_2 + s)^2} \, \ph(s) ds \\
 & \qquad = \frac{1}{\pi} \int_{-\infty}^\infty \frac{(x_2^2 - x_1^2)^2 + 2 (x_1^2 + x_2^2) (y_2 - y_1)^2 + (y_2 - y_1)^4}{(x_2^2 + y_2^2 - x_1^2 - y_1^2) (x_1^2 + (y_1 + s)^2) (x_2^2 + (y_2 + s)^2)} \, |s| \ph(s) ds .
}
Finally, we have
\formula{
 (x_2^2 - x_1^2)^2 + 2 (x_1^2 + x_2^2) (y_2 - y_1)^2 + (y_2 - y_1)^4 & \ge ((x_2 - x_1)^2 + (y_2 - y_1)^2)^2 \\
 & = |\zeta_f(r_2) - \zeta_f(r_1)|^4 \ge (r_2 - r_1)^4
}
and $x_1^2 + y_1^2 = r_1^2$, $x_2^2 + y_2^2 = r_2^2$. By combining the above observations, we arrive at
\formula*[eq:lambda:holder:aux3]{
 \Phi(r_2) - \Phi(r_1) & \ge \frac{(r_2 - r_1)^4}{r_2^2 - r_1^2} \, \frac{1}{\pi} \int_{-\infty}^\infty \frac{1}{(x_1^2 + (y_1 + s)^2) (x_2^2 + (y_2 + s)^2)} \, |s| \ph(s) ds \\
 & \ge \frac{(r_2 - r_1)^3}{r_1 + r_2} \, \frac{1}{\pi} \int_{-\infty}^\infty \frac{1}{4 (r_1^2 + s^2) (r_2^2 + s^2)} \, |s| \ph(s) ds .
}
This is the desired lower bound for the growth of $\Phi(r)$. Recall that here $r_1, r_2 \in (0, \infty) \setminus \partial Z_f$, and $r_1 < r_2$. However, $\Phi(r) = \log \lambda_f(r)$ is continuous, and therefore~\eqref{eq:lambda:holder:aux3} extends to arbitrary $r_1, r_2 \in (0, \infty)$ such that $r_1 < r_2$.

Fix an interval $[R_1, R_2] \sub (0, \infty)$. If $R_1 \le r_1 < r_2 \le R_2$, then, by~\eqref{eq:lambda:holder:aux3},
\formula{
 \lambda_f(r_2) - \lambda_f(r_1) & = e^{\Phi(r_1)} (e^{\Phi(r_2) - \Phi(r_1)} - 1) \\
 & \ge e^{\Phi(R_1)} (\Phi(r_2) - \Phi(r_1)) \ge C(f, R_1, R_2) (r_2 - r_1)^3
}
for some constant $C(f, R_1, R_2) > 0$ (depending on $f$ and the interval $[R_1, R_2]$). This estimate is equivalent to Hölder continuity of $\lambda_f^{-1}(s)$ on $[\lambda_f(R_1), \lambda_f(R_2)]$ with exponent $\tfrac{1}{3}$.
\end{proof}

The exponent $\tfrac{1}{3}$ is sharp, as it is easily seen by inspecting the example $f(\xi) = \xi^2 (i \xi + 14) (i \xi + \tfrac{1}{2})^{-1} (-i \xi + 2)^{-1}$ near $\xi = 2 i$; see Figure~1(f) in~\cite{kwasnicki:rogers}.

\subsection{Difference quotients}

If $f(\xi)$ is a complete Bernstein function and $\zeta \in (0, \infty)$, then the difference quotient $(\xi - \zeta) / (f(\xi) - f(\zeta))$ (extended continuously at $\xi = \zeta$) is a complete Bernstein function of $\xi$. This property played an important role in~\cite{kwasnicki:sym}. The following result provides an analogous statement for Rogers functions.

\begin{proposition}
\label{prop:r:quot}
If $f(\xi)$ is a Rogers function and $\re \zeta > 0$, then
\formula[eq:r:quot:asym]{
 g(\xi) & = \frac{\xi^2}{(\xi - \zeta) (\xi + \overline{\zeta})} \biggl(f(\xi) - \re f(\zeta) - \frac{i \xi + \im \zeta}{\re \zeta} \, \im f(\zeta)\biggr) ,
}
defined for $\xi \in \hp \setminus \{\zeta\}$ and extended continuously at $\xi = \zeta$, is a Rogers function. In particular, if $f(\xi)$ is a non-constant Rogers function and $\zeta \in \Gamma_f$, then
\formula[eq:r:quot]{
 h(\xi) & = \frac{(\xi - \zeta) (\xi + \overline{\zeta})}{f(\xi) - f(\zeta)} \, ,
}
where $\xi \in \hp \setminus \{\zeta\}$, extended continuously at $\xi = \zeta$ so that $h(\zeta) = (2 \re \zeta) / f'(\zeta)$, is a Rogers function. Similarly, if $\zeta \in i \R$ is an accumulation point of $\Gamma_f$ and $|\zeta| = r \in [0, \infty)$, then
\formula[eq:r:quot:ir]{
 h(\xi) & = \frac{(\xi - \zeta)^2}{f(\xi) - \lambda_f(r)}
}
defines a Rogers function of $\xi \in \hp$.
\end{proposition}

\begin{proof}
Suppose that $f(\xi)$ has the Stieltjes representation~\eqref{eq:r:int}. By a simple calculation,
\formula{
 f(\xi) & = a (\xi^2 - 2 i \xi \im \zeta) - i \tilde{b} \xi + c + \frac{1}{\pi} \int_{\R \setminus \{0\}} \biggl(\frac{\xi}{\xi + i s} + \frac{i \xi s}{|\zeta + i s|^2}\biggr) \frac{\mu(ds)}{|s|}
}
for some $\tilde{b} \in \R$; namely, $\tilde{b} = b - 2 a \im \zeta + \tfrac{1}{\pi} \int_{\R \setminus \{0\}} (s / |\zeta + i s|^2 - \sign s / (1 + |s|)) |s|^{-1} \mu(ds)$. If we set $\xi = \zeta$, we obtain
\formula{
 f(\zeta) & = a (\zeta^2 - 2 i \zeta \im \zeta) - i \tilde{b} \zeta + c + \frac{1}{\pi} \int_{\R \setminus \{0\}} \frac{\zeta (\overline{\zeta} - i s) + i \zeta s}{|\zeta + i s|^2} \, \frac{\mu(ds)}{|s|} \\
 & = a |\zeta|^2 - i \tilde{b} \zeta + c + \frac{1}{\pi} \int_{\R \setminus \{0\}} \frac{|\zeta|^2}{|\zeta + i s|^2} \frac{\mu(ds)}{|s|} .
}
In particular, $\im f(\zeta) = -\tilde{b} \re \zeta$. After a short calculation (we omit the details), the above two equalities lead to
\formula{
 f(\xi) - \re f(\zeta) - \frac{i \xi + \im \zeta}{\re \zeta} \, \im f(\zeta) & = a (\xi^2 - 2 i \xi \im \zeta - |\zeta|^2) \\
 & \hspace*{-3em} + \frac{1}{\pi} \int_{\R \setminus \{0\}} \biggl(\frac{\xi}{\xi + i s} + \frac{i \xi s}{|\zeta + i s|^2} - \frac{|\zeta|^2}{|\zeta + i s|^2}\biggr) \frac{\mu(ds)}{|s|} \, .
}
By another simple calculation,
\formula{
 \frac{\xi}{\xi + i s} + \frac{i \xi s}{|\zeta + i s|^2} - \frac{|\zeta|^2}{|\zeta + i s|^2} & = \frac{\xi |\zeta + i s|^2 + (i \xi s - |\zeta|^2) (\xi + i s)}{(\xi + i s) |\zeta + i s|^2} \\
 & = \frac{i s (\xi^2 - 2 i \xi \im \zeta - |\zeta|^2)}{(\xi + i s) |\zeta + i s|^2} \, ,
}
and $\xi^2 - 2 i \xi \im \zeta - |\zeta|^2 = (\xi - \zeta) (\xi + \overline{\zeta})$. It follows that
\formula*[eq:r:quot:aux1]{
 \hspace*{3em} & \hspace*{-3em} f(\xi) - \re f(\zeta) - \frac{i \xi + \im \zeta}{\re \zeta} \, \im f(\zeta) \\
 & = (\xi - \zeta) (\xi + \overline{\zeta}) \biggl(a + \frac{1}{\pi} \int_{\R \setminus \{0\}} \frac{i s}{\xi + i s} \, \frac{\mu(ds)}{|s| |\zeta + i s|^2}\biggr) .
}
Therefore,
\formula*[eq:r:quot:aux2]{
 g(\xi) & = a \xi^2 + \frac{1}{\pi} \int_{\R \setminus \{0\}} \frac{i s \xi^2}{\xi + i s} \, \frac{\mu(ds)}{|s| |\zeta + i s|^2} \\
 & = a \xi^2 + \frac{1}{\pi} \int_{\R \setminus \{0\}} \biggl(\frac{\xi}{\xi + i s} + \frac{i \xi}{s}\biggr) \frac{|s| \mu(ds)}{|\zeta + i s|^2} \\
 & = a \xi^2 - i B \xi + \frac{1}{\pi} \int_{\R \setminus \{0\}} \biggl(\frac{\xi}{\xi + i s} + \frac{i \xi \sign s}{1 + |s|}\biggr) \frac{|s| \mu(ds)}{|\zeta + i s|^2}
}
for an appropriate $B \in \R$; namely, $B = \tfrac{1}{\pi} \int_{\R \setminus \{0\}} (\sign s / (1 + |s|) - 1 / s) |s| |\zeta + i s|^{-2} \mu(ds)$. The right-hand side is the Stieltjes representation~\eqref{eq:r:int} of a Rogers function, and therefore $g$ is a Rogers function. The first assertion of the lemma is proved.

If $\zeta \in \Gamma_f$, then $\im f(\zeta) = 0$, and hence $h(\xi) = \xi^2 / g(\xi)$ is a Rogers function. This proves the second statement of the lemma. The last one follows by the following observation: if $\zeta$ is an accumulation point of $\Gamma_f$, then there is a sequence $\zeta_n \in \Gamma_f$ which converges to $\zeta$, and the corresponding Rogers functions $h_n(\xi)$ converge locally uniformly in $\hp$ to the function $h(\xi)$ given by~\eqref{eq:r:quot:ir}. It remains to note that a locally uniform limit of Rogers functions is a Rogers function.
\end{proof}

\begin{remark}
\label{rem:r:quot}
If $f(\xi)$ is a non-constant Rogers function and $f(\zeta) \in (0, \infty)$ for some $\zeta \in \hp$, then $f(\xi)$ is non-degenerate, and from~\eqref{eq:r:quot:aux2} it follows that also $g(\xi)$ and $h(\xi)$ are non-degenerate.
\end{remark}

When $r \in \Cl Z_f$ and $\zeta = \zeta_f(r)$ in Proposition~\ref{prop:r:quot}, then we denote the function $h(\xi)$ defined there (in~\eqref{eq:r:quot}) by $f(r; \xi)$. In other words, we let
\formula[eq:r:fr]{
 f(r; \xi) & = \frac{(\xi - \zeta_f(r)) (\xi + \overline{\zeta_f(r)})}{f(\xi) - \lambda_f(r)}
}
for $\xi \in \hp \setminus \{\zeta\}$, where $\lambda_f(r) = f(\zeta_f(r))$, and if $r \in Z_f$, then we additionally set
\formula{
 f(r; \zeta_f(r)) & = \frac{2 \re \zeta_f(r)}{f'(\zeta_f(r))} \, .
}

The next result specializes Proposition~\ref{prop:r:quot} for Rogers functions of the form $h(-i \xi)$ or $h(i \xi)$, where $h(\xi)$ is a complete Bernstein function. A special case corresponding to $\zeta \in (0, i \infty)$ was proved in~\cite{kwasnicki:sym}.

\begin{proposition}[see Lemma~2.20 in~\cite{kwasnicki:sym}]
\label{prop:cbf:quot}
If $h(\xi)$ is a complete Bernstein function and $\im \zeta > 0$, then, for some Stieltjes function $g(\xi)$,
\formula{
 \frac{h(\xi)}{(\xi - \zeta) (\xi - \overline{\zeta})} & = \frac{1}{2 i \im \zeta} \biggl(\frac{h(\zeta)}{\xi - \zeta} - \frac{h(\overline{\zeta})}{\xi - \overline{\zeta}}\biggr) - g(\xi)
}
for $\xi \in D_h$. Furthermore, the constants $b$ and $c$ in the Stieltjes representation~\eqref{eq:s:int} of $g(\xi)$ are equal to $0$.
\end{proposition}

\begin{proof}
Suppose that $h(\xi)$ has the Stieltjes representation~\eqref{eq:cbf:int} and $\im \zeta > 0$. Then $f(\xi) = h(-i \xi)$ is a Rogers function with Stieltjes representation
\formula{
 f(\xi) & = -i b \xi + c + \frac{1}{\pi} \int_{(0, \infty)} \frac{\xi}{\xi + i s} \, \frac{\mu(ds)}{s} \\
 & = -i \tilde{b} \xi + c + \frac{1}{\pi} \int_{(0, \infty)} \biggl(\frac{\xi}{\xi + i s} + \frac{i \xi s}{|\zeta + i s|^2}\biggr) \frac{\mu(ds)}{s}
}
for some $\tilde{b} \in \R$. Identity~\eqref{eq:r:quot:aux1} from the proof of Proposition~\ref{prop:r:quot}, with $\xi$ and $\zeta$ replaced by $i \xi$ and $i \overline{\zeta}$, leads to
\formula{
 h(\xi) - \re f(\overline{\zeta}) - \frac{-\xi + \re \zeta}{\im \zeta} \, \im f(\overline{\zeta}) & = (i \xi - i \overline{\zeta}) (i \xi - i \zeta) \, \frac{1}{\pi} \int_{(0, \infty)} \frac{i s}{i \xi + i s} \, \frac{\mu(ds)}{|s| |i \overline{\zeta} + i s|^2}
}
for $\xi \in \C \setminus (-\infty, 0]$. After simplification, this gives
\formula{
 h(\xi) - \frac{(\xi - \overline{\zeta}) h(\zeta) - (\xi - \zeta) h(\overline{\zeta})}{2 i \im \zeta} & = -(\xi - \overline{\zeta}) (\xi - \zeta) \, \frac{1}{\pi} \int_{(0, \infty)} \frac{1}{\xi + s} \, \frac{\mu(ds)}{|\zeta + s|^2} \, .
}
Hence,
\formula{
 \frac{1}{2 i \im \zeta} \biggl(\frac{h(\zeta)}{\xi - \zeta} - \frac{h(\overline{\zeta})}{\xi - \overline{\zeta}}\biggr) - \frac{h(\xi)}{(\xi - \zeta) (\xi - \overline{\zeta})} & = \frac{1}{\pi} \int_{(0, \infty)} \frac{1}{\xi + s} \, \frac{s}{|\zeta + s|^2} \, \frac{\mu(ds)}{s} \, ,
}
and because $s / |\zeta + s|^2$ is positive and bounded, the right-hand side defines a Stieltjes function of $\xi$.
\end{proof}

We conclude this part with an application of Proposition~\ref{prop:r:bd2} which will play a crucial role in our development. The proof of this result is unexpectedly complicated, and requires the following auxiliary lemma.

\begin{lemma}
\label{lem:aux:int}
If $f(\xi)$ is a non-constant and non-degenerate Rogers function, $\zeta \in \Gamma_f$,
\formula{
 h(\xi) & = \frac{(\xi - \zeta) (\xi + \overline{\zeta})}{(\xi + i |\zeta|)^2} \, ,
}
and $\xi_1, \xi_2 \in D_f^+ \cup D_f^-$, then
\formula[eq:aux:int:1]{
 \frac{1}{2 \pi i} \int_{\Gamma_f^\star} \biggl( \frac{1}{z - \xi_1} - \frac{1}{z - \xi_2} \biggr) \log h(z) dz & = \ind_{D_f^+}(\xi_1) \log h(\xi_1) - \ind_{D_f^+}(\xi_2) \log h(\xi_2) .
}
Similarly, if 
\formula{
 h(\xi) & = \frac{(\xi - \zeta) (\xi + \overline{\zeta})}{(\xi - i |\zeta|)^2} \, ,
}
and $\xi_1, \xi_2 \in D_f^+ \cup D_f^-$, then
\formula[eq:aux:int:2]{
 \frac{1}{2 \pi i} \int_{\Gamma_f^\star} \biggl( \frac{1}{z - \xi_1} - \frac{1}{z - \xi_2} \biggr) \log h(z) dz & = \ind_{D_f^-}(\xi_2) \log h(\xi_2) - \ind_{D_f^-}(\xi_1) \log h(\xi_1) .
}
\end{lemma}

\begin{figure}
\centering
\begin{tikzpicture}
\footnotesize
\coordinate (X) at (3,0);
\coordinate (Y) at (0,3);
\coordinate (Xn) at (-3,0);
\coordinate (Yn) at (0,-3);
\coordinate (O) at (0,0);
\coordinate (E) at (1,2);
\coordinate (zeta) at (1,1);
\coordinate (nzeta) at (-1,1);
\draw[densely dotted, olive] (1.5,-3) -- (O) -- (1.5,3);
\draw[densely dotted, olive] (-1.5,-3) -- (O) -- (-1.5,3);
\draw[very thick, densely dotted, olive] (O) -- (0.2,-0.4) .. controls (0.35,-0.5) and (1.15,-0.5) .. (1,1) .. controls (0.97,1.3) and (0.9,1.5) .. (0.9,1.8) -- (1.3,2.6) .. controls (1.8,2.9) and (2.5,2.7) .. (3,3);
\draw[very thick, densely dotted, olive] (O) -- (-0.2,-0.4) .. controls (-0.35,-0.5) and (-1.15,-0.5) .. (-1,1) .. controls (-0.97,1.3) and (-0.9,1.5) .. (-0.9,1.8) -- (-1.3,2.6) .. controls (-1.8,2.9) and (-2.5,2.7) .. (-3,3);
\draw[teal] (O) .. controls (0.00,-0.15) and (0.05,-0.3) .. (0.2,-0.4) .. controls (0.35,-0.5) and (1.15,-0.5) .. (1,1) .. controls (0.97,1.3) and (0.9,1.5) .. (0.9,1.8) .. controls (0.9,2.1) and (0.95,2.3) .. (1.3,2.6) .. controls (1.8,2.9) and (2.5,2.7) .. (3,3);
\draw[teal] (O) .. controls (-0.00,-0.15) and (-0.05,-0.3) .. (-0.2,-0.4) .. controls (-0.35,-0.5) and (-1.15,-0.5) .. (-1,1) .. controls (-0.97,1.3) and (-0.9,1.5) .. (-0.9,1.8) .. controls (-0.9,2.1) and (-0.95,2.3) .. (-1.3,2.6) .. controls (-1.8,2.9) and (-2.5,2.7) .. (-3,3);
\draw[thick, violet] (zeta) arc (45:-225:1.4142);
\filldraw (zeta) circle[radius=1pt] node[above right] {$\zeta$};
\filldraw (nzeta) circle[radius=1pt] node[above left] {$-\overline{\zeta}$};
\node[teal, left] at (1,3) {$\Gamma_f^\star$};
\node[olive, right] at (1.5,3) {$\Gamma_\thet^\star$};
\pic["$\thet$",draw,angle eccentricity=1.3,angle radius=0.5cm] {angle=X--O--E};
\draw[->] (Xn) -- (X) node[above] {$\re$};
\draw[->] (Yn) -- (Y) node[left] {$\im$};
\end{tikzpicture}
\caption{Setting for the proof of Lemma~\ref{lem:aux:int}. Function $h$ has a branch cut at the purple arc. Contour $\Gamma_f^\star$ (teal line) is approximated by the contour $\Gamma_\thet^\star$ (olive dotted line).}
\label{fig:aux:int}
\end{figure}

\begin{proof}
Both identities are proved essentially in the same way, so we only discuss the case of~\eqref{eq:aux:int:1}. The argument is in fact very similar to the proof of Lemma~5.4 in~\cite{kwasnicki:rogers}, so we omit some details. As in that result, by continuity, with no loss of generality we may assume that $\xi_1$ and $\xi_2$ do not lie on the imaginary axis.

Observe that $\log h(\xi) = \log((\xi - \zeta) (\xi + \overline{\zeta}) / (\xi + i |\zeta|)^2)$ is a holomorphic function of $\xi$ with a branch cut along the arc of the circle $|\xi| = r$ with endpoints $\zeta$ and $-\overline{\zeta}$ which also contains $-i |\zeta|$ (see Figure~\ref{fig:aux:int}). We note that for the other function $h(\xi)$ in the statement of the lemma, the one appearing in~\eqref{eq:aux:int:2}, the function $\log h(\xi)$ is holomorphic with a branch cut along the complementary arc, that is, the other arc of the circle $|\xi| = r$ with endpoints $\zeta$ and $-\overline{\zeta}$.

For a given $\thet \in (0, \tfrac{\pi}{2})$ we consider the curve $\Gamma_\thet^\star$ which can be informally defined to be the curve $\Gamma_f^\star$ truncated at lines $\Arg(\pm \xi) = \pm \thet$. More formally, $\Gamma_\thet^\star$ is a curve parameterised by the function $\zeta_\thet(r)$ of $r \in \R$, which for $r > 0$ is given by
\formula{
 \zeta_\thet(r) & =
 \begin{cases}
  \zeta_f(r) & \text{if $\lvert\Arg \zeta_f(r)\rvert \le \thet$,} \\
  r e^{i \thet} & \text{if $\Arg \zeta_f(r) > \thet$,} \\
  r e^{-i \thet} & \text{if $\Arg \zeta_f(r) < -\thet$,}
 \end{cases}
}
and which satisfies $\zeta_\thet(-r) = -\overline{\zeta_\thet(r)}$ and $\zeta_\thet(0) = 0$. (This is a minor modification of the definition of the curve $\Gamma_p^\star$ in the proof of Lemma~5.4 in~\cite{kwasnicki:rogers}.) We also define $D_\thet^+$ to be the region above $\Gamma_\thet^\star$:
\formula{
 D_\thet^+ & = \{ r e^{i \alpha} : r > 0 , \, \alpha \in (\Arg \zeta_\thet(r), \pi - \Arg \zeta_\thet(r)) \} .
}
Observe that $\Gamma_\thet^\star$ is a simple curve, and it is the boundary of the region $D_\thet^+$. Furthermore, if $\thet > \Arg \zeta$, then $\log h(\xi)$ is a holomorphic function of $\xi \in D_\thet^+$, continuous on $(\Cl D_\thet^+) \setminus \{\zeta, -\overline{\zeta}\}$, bounded at infinity, and with logarithmic growth near $\xi = \zeta$ and $\xi = -\overline{\zeta}$. Therefore, a standard application of the residue theorem in $\{\xi \in D_\thet^+ : |\xi| < R, \, |\xi - \zeta| > \eps, \, |\xi + \overline{\zeta}| > \eps\}$ and a limiting procedure as $\eps \to 0^+$ and $R \to \infty$ lead us to
\formula{
 \frac{1}{2 \pi i} \int_{\Gamma_\thet^\star} \biggl( \frac{1}{z - \xi_1} - \frac{1}{z - \xi_2} \biggr) \log h(z) dz & = \ind_{D_\thet^+}(\xi_1) \log h(\xi_1) - \ind_{D_\thet^+}(\xi_2) \log h(\xi_2)
}
whenever $\xi_1, \xi_2$ do not lie on $\Gamma_\thet^\star$. Finally, for $\thet \in (0, \tfrac{\pi}{2})$ large enough, we have $\ind_{D_\thet^+}(\xi_j) = \ind_{D_f^+}(\xi_j)$ for $j = 1, 2$. Let $\zeta_f(0) = 0$ and $\zeta_f(-r) = -\overline{\zeta_f(r)}$. Then $|\zeta_\thet'(r)| \le \lvert\zeta_f'(r)\rvert$ for almost all $r \in \R$, and hence, by the dominated convergence theorem, the integrals
\formula{
 \hspace*{7em} & \hspace*{-7em} \frac{1}{2 \pi i} \int_{\Gamma_\thet^\star} \biggl( \frac{1}{z - \xi_1} - \frac{1}{z - \xi_2} \biggr) \log h(z) dz \\
 & = \frac{1}{2 \pi i} \int_{-\infty}^\infty \biggl( \frac{1}{\zeta_\thet(r) - \xi_1} - \frac{1}{\zeta_\thet(r) - \xi_2} \biggr) \log h(\zeta_\thet(r)) \zeta_\thet'(r) dr
}
converge to the integral
\formula[eq:aux:int:aux]{
 & \frac{1}{2 \pi i} \int_{-\infty}^\infty \biggl( \frac{1}{\zeta_f(r) - \xi_1} - \frac{1}{\zeta_f(r) - \xi_2} \biggr) \log h(\zeta_f(r)) \zeta_f'(r) dr .
}
Finally, note that $\zeta_f(r)$ for $r \in (-Z_f) \cup Z_f$ is a parameterisation of $\Gamma_f^\star$, and the integral restricted to $r \in (-\infty, 0) \setminus (-Z_f)$ cancels with the integral restricted to $(0, \infty) \setminus Z_f$. We conclude that the integral in~\eqref{eq:aux:int:aux} is equal to
\formula{
 & \frac{1}{2 \pi i} \int_{\Gamma_f^\star} \biggl( \frac{1}{z - \xi_1} - \frac{1}{z - \xi_2} \biggr) \log h(z) dz ,
}
and formula~\eqref{eq:aux:int:1} follows.
\end{proof}

\begin{proposition}
\label{prop:r:quot:bd}
Let $f(\xi)$ be a non-constant and non-degenerate Rogers function and $r \in Z_f$. Denote by $f(r; \xi)$ the Rogers function $h(\xi)$ defined in Proposition~\ref{prop:r:quot}, as in~\eqref{eq:r:fr}, and let $f^+(r; \xi)$ and $f^-(r; \xi)$ denote the Wiener--Hopf factors of $f(r; \xi)$. Let $g(\xi)$ denote the branch of $\sqrt{-(\xi - \zeta_f(r)) (\xi + \overline{\zeta_f(r)})}$ which is a holomorphic function in $D_f^+ \cup D_f^-$, equal to $|\xi - \zeta_f(r)|$ on $D_f^+ \cap (0, i \infty)$ and on $D_f^- \cap (-i \infty, 0)$. If $\xi_1, \xi_2 \in D_f^+ \cup D_f^-$, then
\formula*[eq:r:quot:bd]{
 \hspace*{5em} & \hspace*{-5em} \exp \biggl( \frac{1}{2 \pi i} \int_{\Gamma_f^\star} \biggl( \frac{1}{z - \xi_1} - \frac{1}{z - \xi_2} \biggr) \log |f(z) - \lambda_f(r)| dz \biggr) \\
 & = \begin{cases}
  \displaystyle \frac{g(\xi_1) f^+(r; -i \xi_2)}{g(\xi_2) f^+(r; -i \xi_1)}
   & \text{if $\xi_1, \xi_2 \in D_f^+$,} \\[1em] 
  \displaystyle \frac{g(\xi_2) f^-(r; i \xi_1)}{g(\xi_1) f^-(r; i \xi_2)}
   & \text{if $\xi_1, \xi_2 \in D_f^-$,} \\[1em]
  \displaystyle \frac{g(\xi_1) g(\xi_2)}{f^+(r; -i \xi_1) f^-(r; i \xi_2)}
   & \text{if $\xi_1 \in D_f^+$, $\xi_2 \in D_f^-$.}
 \end{cases}
}
\end{proposition}

\begin{proof}
Let us write $\zeta = \zeta_f(r)$, so that $\lambda_f(r) = f(\zeta)$ and (see~\eqref{eq:r:fr})
\formula{
 f(\xi) - \lambda_f(r) & = \frac{(\xi - \zeta) (\xi + \bar{\zeta})}{f(r; \xi)} \, .
}
Roughly speaking, our goal is to write $|f(\xi) - \lambda_f(r)|$ for $\xi \in \Gamma_f^\star$ as a product of four terms, apply Proposition~\ref{prop:r:bd2} to two of them, and perform direct integration for the other two. Throughout the proof we denote by $C$ a generic positive constant, depending on parameters listed in brackets. We stress that the value of $C$ may change even within a single equation.

\emph{Step 1.}
Suppose that
\formula{
 f(r; \xi) & = c_r \exp\biggl(\frac{1}{\pi} \int_{-\infty}^\infty \biggl(\frac{\xi}{\xi + i s} - \frac{1}{1 + |s|}\biggr) \frac{\ph_r(s)}{|s|} \, ds\biggr)
}
is the exponential representation of the Rogers function $f(r; \xi)$, and define three auxiliary Rogers functions:
\formula*[eq:r:quot:bd:aux]{
 h_0(\xi) & = \exp\biggl(\frac{1}{\pi} \biggl(\int_{-\infty}^{-r} + \int_r^\infty\biggr) \biggl(\frac{\xi}{\xi + i s} - \frac{1}{1 + |s|}\biggr) \frac{\pi}{|s|} \, ds\biggr) , \\
 h_1(\xi) & = \sqrt{c_r} \exp\biggl(\frac{1}{\pi} \int_{-r}^r \biggl(\frac{\xi}{\xi + i s} - \frac{1}{1 + |s|}\biggr) \frac{\ph_r(s)}{|s|} \, ds\biggr) , \\
 h_2(\xi) & = \sqrt{c_r} \exp\biggl(\frac{1}{\pi} \biggl(\int_{-\infty}^{-r} + \int_r^\infty\biggr) \biggl(\frac{\xi}{\xi + i s} - \frac{1}{1 + |s|}\biggr) \frac{\pi - \ph_r(s)}{|s|} \, ds\biggr) .
}
By~\eqref{eq:r:wh}, with an appropriate choice of multiplicative constants, we have
\formula{
 h_0^+(-i \xi) & = \exp\biggl(\frac{1}{\pi} \int_r^\infty \biggl(\frac{\xi}{\xi + i s} - \frac{1}{1 + |s|}\biggr) \frac{\pi}{|s|} \, ds\biggr) , \\
 h_1^+(-i \xi) & = c_r^{1/4} \exp\biggl(\frac{1}{\pi} \int_0^r \biggl(\frac{\xi}{\xi + i s} - \frac{1}{1 + |s|}\biggr) \frac{\ph_r(s)}{|s|} \, ds\biggr) , \\
 h_2^+(-i \xi) & = c_r^{1/4} \exp\biggl(\frac{1}{\pi} \int_r^\infty \biggl(\frac{\xi}{\xi + i s} - \frac{1}{1 + |s|}\biggr) \frac{\pi - \ph_r(s)}{|s|} \, ds\biggr) ,
}
and
\formula{
 h_0^-(i \xi) & = \exp\biggl(\frac{1}{\pi} \int_{-\infty}^{-r} \biggl(\frac{\xi}{\xi + i s} - \frac{1}{1 + |s|}\biggr) \frac{\pi}{|s|} \, ds\biggr) , \\
 h_1^-(i \xi) & = c_r^{1/4} \exp\biggl(\frac{1}{\pi} \int_{-r}^0 \biggl(\frac{\xi}{\xi + i s} - \frac{1}{1 + |s|}\biggr) \frac{\ph_r(s)}{|s|} \, ds\biggr) , \\
 h_2^-(i \xi) & = c_r^{1/4} \exp\biggl(\frac{1}{\pi} \int_{-\infty}^{-r} \biggl(\frac{\xi}{\xi + i s} - \frac{1}{1 + |s|}\biggr) \frac{\pi - \ph_r(s)}{|s|} \, ds\biggr) .
}
Clearly,
\formula[eq:r:quot:decomp]{
 f(r; \xi) & = \frac{h_0(\xi) h_1(\xi)}{h_2(\xi)} \, , & f^+(r; \xi) & = \frac{h_0^+(\xi) h_1^+(\xi)}{h_2^+(\xi)} \, , & f^-(r; \xi) & = \frac{h_0^-(\xi) h_1^-(\xi)}{h_2^-(\xi)} \, .
}
In particular, $\log f(r; \xi) + \log h_2(\xi) = \log h_0(\xi) + \log h_1(\xi) + 2 n \pi i$ for some integer $n$. By the definition of a (non-constant) Rogers function, the complex arguments of $f(r; \xi)$ and $h_j(\xi)$, $j = 1, 2, 3$, belong to $(\Arg \xi - \tfrac{\pi}{2}, \Arg \xi + \tfrac{\pi}{2})$ when $\re \xi > 0$, so that necessarily $n = 0$. Thus,
\formula[eq:r:quot:decomp:log]{
 \log f(r; \xi) & = \log h_0(\xi) + \log h_1(\xi) - \log h_2(\xi)
}
when $\re \xi > 0$. Finally, by a short calculation (we omit the details),
\formula[eq:r:quot:bd:h0]{
 h_0^+(-i \xi) & = \frac{r - i \xi}{1 + r} , & h_0^-(i \xi) & = \frac{r + i \xi}{1 + r} , & h_0(\xi) & = \frac{\xi^2 + r^2}{(1 + r)^2} .
}

\emph{Step 2.}
We claim that Proposition~\ref{prop:r:bd2} applies both to $h_1(\xi)$ and to $h_2(\xi)$. Thus, we need to show that $\lvert\Arg h_j(\xi)\rvert \le C(f, r) \re \xi / |\xi|$ for $\xi \in \Gamma_f$ and $j = 1, 2$. If $\lvert\Arg \xi\rvert \le \lvert\Arg \zeta\rvert$, then $\re \xi / |\xi| \ge \re \zeta / |\zeta|$, and so $\lvert\Arg h_j(\xi)\rvert \le \pi \le C(\zeta) \re \xi / |\xi|$. Therefore, it suffices to consider $\xi$ such that $\lvert\Arg \xi\rvert > \lvert\Arg \zeta\rvert$.

Observe that if $\xi \in \Gamma_f$ and $|\xi| < r$, then $f(\xi) < f(\zeta)$, and hence
\formula[eq:r:quot:bd:f1]{
 \Arg f(r; \xi) & = \Arg \frac{(\xi - \zeta) (\xi + \overline{\zeta})}{f(\xi) - f(\zeta)} = \Arg((\zeta - \xi)(\overline{\zeta} + \xi)) .
}
Similarly, if $\xi \in \Gamma_f$, $|\xi| > r$ and $\lvert\Arg \xi\rvert > \lvert\Arg \zeta\rvert$, then $f(\xi) > f(\zeta)$, and
\formula[eq:r:quot:bd:f2]{
 \Arg(-f(r; \xi)) & = \Arg((\zeta - \xi)(\overline{\zeta} + \xi)) .
}
Finally, it is relatively simple to see that
\formula[eq:r:quot:bd:f3]{
 \lvert\Arg((\zeta - \xi)(\overline{\zeta} + \xi))\rvert & \le C(\zeta) \, \frac{\re \xi}{1 + |\xi|^2} \le C(\zeta) \, \frac{\re \xi}{|\xi|}
}
whenever $\re \xi > 0$ and $\lvert\Arg \xi\rvert > \lvert\Arg \zeta\rvert$. Indeed: if
\formula{
 \thet & = \lvert\Arg((\zeta - \xi)(\overline{\zeta} + \xi))\rvert = \biggl|\Arg \frac{\zeta - \xi}{\zeta + \overline{\xi}}\biggr| ,
}
then, by~\eqref{eq:triangle},
\formula{
 \sin \frac{\thet}{2} & \le \frac{|(\zeta - \xi) - (\zeta + \overline{\xi})|}{|(\zeta - \xi) + (\zeta + \overline{\xi})|} = \frac{\re \xi}{|\zeta - i \im \xi|} \le C(\zeta) \, \frac{\re \xi}{1 + (\im \xi)^2} \, .
}
However, since $\lvert\Arg \xi\rvert > \lvert\Arg \zeta\rvert$, we have $|\xi| \le C(\zeta) \lvert\im \xi\rvert$, and hence
\formula{
 \thet & \le \tfrac{\pi}{2} \sin \tfrac{\thet}{2} \le C(\zeta) \, \frac{\re \xi}{1 + |\xi|^2} \, .
}

We turn to estimates of $h_0(\xi)$, $h_1(\xi)$ and $h_2(\xi)$. Suppose that $\xi \in \Gamma_f$ and $|\xi| < r$. Then it is easy to see that $|\xi + i s|^2 \ge C(f, r) (1 + s^2)$ when $s < -r$ or $s > r$; again we omit the details. Therefore,
\formula{
 \lvert\Arg h_2(\xi)\rvert & = \lvert\im \log h_2(\xi)\rvert = \biggl|\frac{1}{\pi} \biggl(\int_{-\infty}^{-r} + \int_r^\infty\biggr) \frac{s \re \xi}{|\xi + i s|^2} \frac{\pi - \ph_r(s)}{|s|} \, ds\biggr| \\
 & \le C(f, r) \biggl(\int_{-\infty}^{-r} + \int_r^\infty\biggr) \frac{|s| \re \xi}{1 + s^2} \frac{\pi}{|s|} \, ds = C(f, r) \re \xi \le C(f, r) \, \frac{\re \xi}{|\xi|} \, .
}
By the same calculation, using the expression for $h_0$ given in~\eqref{eq:r:quot:bd:aux}, we have
\formula{
 \lvert\Arg h_0(\xi)\rvert & \le C(f, r) \, \frac{\re \xi}{|\xi|} \, .
}
By~\eqref{eq:r:quot:decomp}, \eqref{eq:r:quot:bd:f1} and~\eqref{eq:r:quot:bd:f3}, we also have
\formula{
 \lvert\Arg h_1(\xi)\rvert & = \biggl| \Arg \frac{f(r; \xi) h_2(\xi)}{h_0(\xi)} \biggr| \le C(f, r) \, \frac{\re \xi}{|\xi|} \, .
}
A very similar argument works if $\xi \in \Gamma_f$, $|\xi| > r$ and $\lvert\Arg \xi\rvert > \lvert\Arg \zeta\rvert$. In this case we have $|\xi + i s| \ge C(f, r) (1 + |\xi|^2)$ when $-r < s < r$; once again we omit the details. Thus,
\formula{
 \lvert\Arg h_1(\xi)\rvert & = \lvert\im \log h_2(\xi)\rvert = \biggl|\frac{1}{\pi} \int_{-r}^r \frac{s \re \xi}{|\xi + i s|^2} \frac{\ph_r(s)}{|s|} \, ds\biggr| \\
 & \le C(f, r) \int_{-r}^r \frac{s \re \xi}{1 + |\xi|^2} \frac{\pi}{|s|} \, ds = C(f, r) \, \frac{\re \xi}{1 + |\xi|^2} \le C(f, r) \, \frac{\re \xi}{|\xi|} \, .
}
Furthermore, it is again easy to see that
\formula{
 \lvert\Arg (-h_0(\xi))\rvert & = \lvert\Arg (-\xi^2 - r^2)\rvert \le C(f, r) \frac{\re \xi}{1 + |\xi|^2} \le C(f, r) \frac{\re \xi}{|\xi|} \, ,
}
and again we omit the details. Consequently, by~\eqref{eq:r:quot:decomp}, \eqref{eq:r:quot:bd:f2} and~\eqref{eq:r:quot:bd:f2},
\formula{
 \lvert\Arg h_2(\xi)\rvert & = \biggl| \Arg \frac{h_1(\xi) (-h_0(\xi))}{(-f(r; \xi))} \biggr| \le C(f, r) \, \frac{\re \xi}{|\xi|} \, .
}
We have thus proved that $\lvert\Arg h_j(\xi)\rvert \le C(f, r) \re \xi / |\xi|$ whenever $\xi \in \Gamma_f$, $\lvert\Arg \xi\rvert > \lvert\Arg \zeta\rvert$ and $j = 1, 2$. This implies that Proposition~\ref{prop:r:bd2} indeed applies to $h_1$ and $h_2$: for $j = 1, 2$ and $\xi_1, \xi_2 \in D_f^+ \cup D_f^-$ we have
\formula*[eq:r:quot:bd:h12]{
 \hspace*{7em} & \hspace*{-7em} \frac{1}{2 \pi i} \int_{\Gamma_f^\star} \biggl( \frac{1}{z - \xi_1} - \frac{1}{z - \xi_2} \biggr) \log h_j(z) dz \\
 & = \begin{cases}
  \log h_j^+(-i \xi_1) - \log h_j^+(-i \xi_2) & \text{if $\xi_1, \xi_2 \in D_f^+$,} \\ 
  \log h_j^-(i \xi_2) - \log h_j^-(i \xi_1) & \text{if $\xi_1, \xi_2 \in D_f^-$,} \\ 
  \log h_j^+(-i \xi_1) + \log h_j^-(i \xi_2) & \text{if $\xi_1 \in D_f^+$, $\xi_2 \in D_f^-$.}
 \end{cases}
}

\emph{Step 3.}
As in Lemma~\ref{lem:aux:int}, we define
\formula[eq:r:quot:bd:h34]{
 h_3(\xi) & = \frac{(\xi - \zeta) (\xi + \overline{\zeta})}{(\xi + i r)^2} \, , & h_4(\xi) & = \frac{(\xi - \zeta) (\xi + \overline{\zeta})}{(\xi - i r)^2} \, .
}
By that result, for $\xi_1, \xi_2 \in D_f^+ \cup D_f^-$ we have
\formula*[eq:r:quot:bd:h3]{
 \hspace*{7em} & \hspace*{-7em} \frac{1}{2 \pi i} \int_{\Gamma_f^\star} \biggl( \frac{1}{z - \xi_1} - \frac{1}{z - \xi_2} \biggr) \log h_3(z) dz \\
 & = \ind_{D_f^+}(\xi_1) \log h_3(\xi_1) - \ind_{D_f^+}(\xi_2) \log h_3(\xi_2) ,
}
and
\formula*[eq:r:quot:bd:h4]{
 \hspace*{7em} & \hspace*{-7em} \frac{1}{2 \pi i} \int_{\Gamma_f^\star} \biggl( \frac{1}{z - \xi_1} - \frac{1}{z - \xi_2} \biggr) \log h_4(z) dz \\
 & = \ind_{D_f^-}(\xi_2) \log h_4(\xi_2) - \ind_{D_f^-}(\xi_1) \log h_4(\xi_1) .
}

\emph{Step 4.}
We are ready to combine all results that we have found in the previous steps. By~\eqref{eq:r:quot:decomp:log}, for $\xi \in \Gamma_f$ such that $|\xi| < r$, we have
\formula{
 \log |f(\xi) - f(\zeta)| & = \log (f(\zeta) - f(\xi)) = \log (-(\xi - \zeta) (\xi + \overline{\zeta})) - \log f(r; \zeta) \\
 & = \log (-(\xi - \zeta) (\xi + \overline{\zeta})) - \log h_0(\xi) - \log h_1(\xi) + \log h_2(\xi) \\
 & = \log (-(\xi - \zeta) (\xi + \overline{\zeta})) - \log (\xi^2 + r^2) + \log (1 + r)^2 \\
 & \hspace*{17em} - \log h_1(\xi) + \log h_2(\xi) .
}
Observe that if $f_1(\xi)$, $f_2(\xi)$ are holomorphic functions of $\xi$ in some (connected) region, and $f_1(\xi), f_2(\xi), f_1(\xi) f_2(\xi) \in \C \setminus (-\infty, 0]$, then $\log (f_1(\xi) f_2(\xi)) - \log f_1(\xi) - \log f_2(\xi)$ is constant. Applying this to $f_1(\xi) = -(\xi - \zeta) (\xi + \overline{\zeta})$ and $f_2(\xi) = 1 / (\xi^2 + r^2)$ (we omit the details), we find that
\formula{
 \log |f(\xi) - f(\zeta)| & = \log \frac{-(\xi - \zeta) (\xi + \overline{\zeta})}{\xi^2 + r^2} + \log (1 + r)^2 - \log h_1(\xi) + \log h_2(\xi)
}
when $|\xi| < r$. The same argument applied to $f_1(\xi) = f_2(\xi) = -(\xi - \zeta) (\xi + \overline{\zeta}) / (\xi^2 + r^2)$ (again we omit the details) shows that
\formula{
 \log |f(\xi) - f(\zeta)| & = \frac{1}{2} \log \frac{(\xi - \zeta)^2 (\xi + \overline{\zeta})^2}{(\xi^2 + r^2)^2} + \log (1 + r)^2 - \log h_1(\xi) + \log h_2(\xi) \\
 & = \frac{1}{2} \log h_3(\xi) + \frac{1}{2} \log h_4(\xi) + \log (1 + r)^2 - \log h_1(\xi) + \log h_2(\xi)
}
when $|\xi| < r$. A similar calculation can be carried out when $\xi \in \Gamma_f$ and $|\xi| > r$: in this case we have
\formula{
 \log |f(\xi) - f(\zeta)| & = \log (f(\xi) - f(\zeta)) = \log ((\xi - \zeta) (\xi + \overline{\zeta})) - \log f(r; \zeta) \\
 & = \log ((\xi - \zeta) (\xi + \overline{\zeta})) - \log h_0(\xi) - \log h_1(\xi) + \log h_2(\xi) \\
 & = \log \frac{(\xi - \zeta) (\xi + \overline{\zeta})}{\xi^2 + r^2} + \log (1 + r)^2 - \log h_1(\xi) + \log h_2(\xi) \\
 & = \frac{1}{2} \log h_3(\xi) + \frac{1}{2} \log h_4(\xi) + \log (1 + r)^2 - \log h_1(\xi) + \log h_2(\xi) ;
}
once again we omit the details. In either case, we obtain the same expression:
\formula[eq:r:quot:bd:decomposition]{
 \log |f(\xi) - f(\zeta)| & = \frac{1}{2} \log h_3(\xi) + \frac{1}{2} \log h_4(\xi) + \log (1 + r)^2 - \log h_1(\xi) + \log h_2(\xi) ,
}
and all that remains is to combine this with the results of the previous steps. If $\xi_1 \in D_f^+$ and $\xi_2 \in D_f^-$, then, by~\eqref{eq:r:quot:bd:h12}, \eqref{eq:r:quot:bd:h3} and~\eqref{eq:r:quot:bd:h4},
\formula{
 \hspace*{7em} & \hspace*{-7em} \exp\biggl(\frac{1}{2 \pi i} \int_{\Gamma_f^\star} \biggl( \frac{1}{z - \xi_1} - \frac{1}{z - \xi_2} \biggr) \log |f(z) - f(\zeta)| dz\biggr) \\
 & = \exp\biggl(\frac{1}{2} \log h_3(\xi_1) + \frac{1}{2} \log h_4(\xi_2) + \log (1 + r)^2 \\
 & \hspace*{3em} - \bigl(\log h_1^+(-i \xi_1) + \log h_1^-(i \xi_2)\bigr) + \bigl(\log h_2^+(-i \xi_1) + \log h_2^-(i \xi_2)\bigr)\biggr) \\
 & = (1 + r)^2 \sqrt{h_3(\xi_1)} \sqrt{h_4(\xi_2)} \, \frac{h_2^+(-i \xi_1)}{h_1^+(-i \xi_1)} \, \frac{h_2^-(i \xi_2)}{h_1^-(i \xi_2)} \, .
}
Using~\eqref{eq:r:quot:decomp}, and then the explicit expressions~\eqref{eq:r:quot:bd:h0} for $h_0(\xi)$ and~\eqref{eq:r:quot:bd:h34} for $h_3(\xi)$ and $h_4(\xi)$, we arrive at
\formula{
 \hspace*{7em} & \hspace*{-7em} \exp\biggl(\frac{1}{2 \pi i} \int_{\Gamma_f^\star} \biggl( \frac{1}{z - \xi_1} - \frac{1}{z - \xi_2} \biggr) \log |f(z) - f(\zeta)| dz\biggr) \\
 & = (1 + r)^2 \sqrt{h_3(\xi_1)} \sqrt{h_4(\xi_2)} \, \frac{h_0^+(-i \xi_1)}{f^+(r; -i \xi_1)} \, \frac{h_0^-(i \xi_2)}{f^-(r; i \xi_2)} \\
 & = \sqrt{\frac{(\xi_1 - \zeta) (\xi_1 + \overline{\zeta})}{(\xi_1 + i r)^2}} \sqrt{\frac{(\xi_2 - \zeta) (\xi_2 + \overline{\zeta})}{(\xi_2 - i r)^2}} \, \frac{r - i \xi_1}{f^+(r; -i \xi_1)} \, \frac{r + i \xi_2}{f^-(r; i \xi_2)} \, .
}
Therefore, slightly abusing the notation,
\formula{
 \hspace*{7em} & \hspace*{-7em} \exp\biggl(\frac{1}{2 \pi i} \int_{\Gamma_f^\star} \biggl( \frac{1}{z - \xi_1} - \frac{1}{z - \xi_2} \biggr) \log |f(z) - f(\zeta)| dz\biggr) \\
 & = \frac{\sqrt{-(\xi_1 - \zeta) (\xi_1 + \overline{\zeta})} \sqrt{-(\xi_2 - \zeta) (\xi_2 + \overline{\zeta})}}{f^+(r; -i \xi_1) f^-(r; i \xi_2)} \, ,
}
as desired; strictly speaking, instead of $\sqrt{-(\xi_j - \zeta) (\xi_j + \overline{\zeta})}$ we should have written $g(\xi_j)$ on the right-hand side.

The argument is very similar if $\xi_1, \xi_2 \in D_f^+$, and we only sketch the calculation: using first~\eqref{eq:r:quot:bd:decomposition} combined with~\eqref{eq:r:quot:bd:h12}, \eqref{eq:r:quot:bd:h3} and~\eqref{eq:r:quot:bd:h4}; then~\eqref{eq:r:quot:decomp}; and finally~\eqref{eq:r:quot:bd:h0} and~\eqref{eq:r:quot:bd:h34}, we obtain
\formula{
 \hspace*{7em} & \hspace*{-7em} \exp\biggl(\frac{1}{2 \pi i} \int_{\Gamma_f^\star} \biggl( \frac{1}{z - \xi_1} - \frac{1}{z - \xi_2} \biggr) \log |f(z) - f(\zeta)| dz\biggr) \\
 & = \frac{\sqrt{h_3(\xi_1)}}{\sqrt{h_3(\xi_2)}} \, \frac{h_2^+(-i \xi_1)}{h_1^+(-i \xi_1)} \, \frac{h_1^+(-i \xi_2)}{h_2^+(-i \xi_2)} \displaybreak[0] \\
 & = \frac{\sqrt{h_3(\xi_1)}}{\sqrt{h_3(\xi_2)}} \, \frac{h_0^+(-i \xi_1)}{f^+(r; -i \xi_1)} \, \frac{f^+(r; -i \xi_2)}{h_0^+(-i \xi_2)} \displaybreak[0] \\
 & = \sqrt{\frac{(\xi_1 - \zeta) (\xi_1 + \overline{\zeta})}{(\xi_1 + i r)^2}} \sqrt{\frac{(\xi_2 + i r)^2}{(\xi_2 - \zeta) (\xi_2 + \overline{\zeta})}} \, \frac{r - i \xi_1}{f^+(r; -i \xi_1)} \, \frac{f^+(r; -i \xi_2)}{r - i \xi_2} \\
 & = \frac{\sqrt{-(\xi_1 - \zeta) (\xi_1 + \overline{\zeta})} \, f^+(r; -i \xi_2)}{\sqrt{-(\xi_2 - \zeta) (\xi_2 + \overline{\zeta})} \, f^+(r; -i \xi_1)} \, .
}
We omit essentially the same calculations for the remaining case $\xi_1, \xi_2 \in D_f^-$.
\end{proof}

\subsection{Balance conditions}

In most results, we require that for some $\eps > 0$, $\zeta_f(r)$ lies in the sector $\{\xi \in \C : \lvert\Arg \xi\rvert \le \tfrac{\pi}{2} - \eps\}$ for $r$ large enough. This \emph{balance condition} is needed in order to deform the contour of integration from $\R$ to $\Gamma_f^\star$ when the integrand is singular near the imaginary axis. However, it also implies some balance between the Wiener--Hopf factors, as stated in the following two results. Throughout this section, by $C$ we denote a generic positive constant that may depend on parameters listed in brackets, and the exact value of $C$ can be different every time it appears, even in a single expression.

\begin{lemma}
\label{lem:r:balance}
If $f(\xi)$ is a non-constant Rogers function, $0 < \eta < \xi$, $M > 0$ and
\formula{
 \Arg \frac{\zeta_f(r) - i \eta}{\zeta_f(r) - i \xi} & \le M \Arg \, \frac{\zeta_f(r) + i \xi}{\zeta_f(r) + i \eta}
}
for every $r > 0$, then
\formula{
 \frac{f^+(\xi)}{f^+(\eta)} & \le \biggl(\frac{f^-(\xi)}{f^-(\eta)}\biggr)^{\! M} .
}
Similarly, if $0 < \eta < \xi$, $M > 0$ and
\formula{
 \Arg \, \frac{\zeta_f(r) + i \xi}{\zeta_f(r) + i \eta} & \le M \Arg \frac{\zeta_f(r) - i \eta}{\zeta_f(r) - i \xi}
}
for every $r > 0$, then
\formula{
 \frac{f^-(\xi)}{f^-(\eta)} & \le \biggl(\frac{f^+(\xi)}{f^+(\eta)}\biggr)^{\! M} .
}
In particular, for every $\thet \in (0, \tfrac{\pi}{2})$ there is a constant $C(\thet) > 0$ with the following properties. If $f(\xi)$ is a non-constant Rogers function, $0 \le \ro_1 \le \ro_2 \le \infty$ and $\Arg \zeta_f(r) \le \thet$ when $\ro_1 < r < \ro_2$, then 
\formula{
 \frac{f^+(\xi)}{f^+(\eta)} \le \biggl(\frac{f^-(\xi)}{f^-(\eta)}\biggr)^{\! C(\thet)} .
}
whenever $2 \ro_1 \le \eta < \xi \le \tfrac{1}{2} \ro_2$. Similarly, if $f(\xi)$ is a non-constant Rogers function, $0 \le \ro_1 \le \ro_2 \le \infty$ and $\Arg \zeta_f(r) \ge -\thet$ when $\ro_1 < r < \ro_2$, then 
\formula{
 \frac{f^-(\xi)}{f^-(\eta)} \le \biggl(\frac{f^+(\xi)}{f^+(\eta)}\biggr)^{\! C(\thet)} .
}
whenever $2 \ro_1 \le \eta < \xi \le \tfrac{1}{2} \ro_2$.
\end{lemma}

\begin{proof}
By Theorem~5.7 in~\cite{kwasnicki:rogers} (see Proposition~\ref{prop:r:wh:bd} above), for $\xi, \eta > 0$ we have
\formula{
 \log \frac{f^+(\xi)}{f^+(\eta)} & = \frac{1}{\pi} \int_0^\infty \Arg \frac{\zeta_f(r) - i \eta}{\zeta_f(r) - i \xi} \, \frac{d\lambda_f(r)}{\lambda_f(r)} \, , \\
 \log \frac{f^-(\xi)}{f^-(\eta)} & = \frac{1}{\pi} \int_0^\infty \Arg \frac{\zeta_f(r) + i \xi}{\zeta_f(r) + i \eta} \, \frac{d\lambda_f(r)}{\lambda_f(r)} \, ,
}
and the first assertion of the lemma follows.

For the proof of the other one, we observe that if $0 < \eta < \xi$ and $r > 0$, then
\formula{
 \Arg \frac{\zeta_f(r) - i \eta}{\zeta_f(r) - i \xi} & = \im \log \frac{\zeta_f(r) - i \eta}{\zeta_f(r) - i \xi} = \int_\eta^\xi \im \frac{i}{\zeta_f(r) - i s} \, ds = \int_\eta^\xi \frac{\re \zeta_f(r)}{|\zeta_f(r) - i s|^2} \, ds ,
}
and, in a similar way,
\formula{
 \Arg \frac{\zeta_f(r) + i \xi}{\zeta_f(r) + i \eta} & = \int_\eta^\xi \frac{\re \zeta_f(r)}{|\zeta_f(r) + i s|^2} \, ds .
}
Suppose that $\Arg \zeta_f(r) \le \thet$ when $\ro_1 < r < \ro_2$. Note that if $2 \ro_1 < s < \tfrac{1}{2} \ro_2$ and either $r \ge \ro_2$ or $r \le \ro_1$, we have
\formula{
 \frac{|\zeta_f(r) + i s|}{|\zeta_f(r) - i s|} & \le \frac{s + r}{|r - s|} \le 3 .
}
On the other hand, if $2 \ro_1 < s < \tfrac{1}{2} \ro_2$ and $\ro_1 < r < \ro_2$, then, by~\eqref{eq:triangle},
\formula{
 \frac{|\zeta_f(r) + i s|}{|\zeta_f(r) - i s|} & \le \frac{s + r}{(s + r) \sin(\tfrac{1}{2} (\tfrac{\pi}{2} - \thet))} = C(\thet) .
}
It follows that if $2 \ro_1 < s < \tfrac{1}{2} \ro_2$, then for all $r > 0$ we have
\formula{
 \frac{1}{|\zeta_f(r) - i s|} & \le C(\thet) \, \frac{1}{|\zeta_f(r) + i s|} \, .
}
Integrating both sides of the above inequality, we find that if $2 \ro_1 \le \eta < \xi \le \tfrac{1}{2} \ro_2$, then for all $r > 0$ we have
\formula{
 \Arg \frac{\zeta_f(r) - i \xi}{\zeta_f(r) - i \eta} & \le C(\thet) \Arg \frac{\zeta_f(r) + i \eta}{\zeta_f(r) + i \xi} \, .
}
It remains to apply the first part of the lemma with $M = C(\thet)$ to get the desired inequality. A very similar argument applies if we assume that $\Arg \zeta_f(r) \ge - \thet$ when $\ro_1 < r < \ro_2$.
\end{proof}

\begin{corollary}
\label{cor:r:balance}
If $0 \le \ro_1 \le \ro_2 \le \infty$, $\thet \in (0, \tfrac{\pi}{2})$, and $f(\xi)$ is a non-constant Rogers function such that $\Arg \zeta_f(r) \le \thet$ when $\ro_1 < r < \ro_2$, then there are constants $C_1, C_2(\thet) > 0$ such that
\formula{
 \frac{f^-(\xi)}{f^-(\eta)} & \ge C_1 \biggl( \frac{|f(\xi)|}{|f(\eta)|} \biggr)^{\! C_2(\thet)}
}
whenever $2 \ro_1 \le \eta < \xi \le \tfrac{1}{2} \ro_2$. Similarly, if $\Arg \zeta_f(r) \ge -\thet$ when $\ro_1 < r < \ro_2$, then
\formula{
 \frac{f^+(\xi)}{f^+(\eta)} & \ge C_1 \biggl( \frac{|f(\xi)|}{|f(\eta)|} \biggr)^{\! C_2(\thet)}
}
whenever $2 \ro_1 \le \eta < \xi \le \tfrac{1}{2} \ro_2$. 
\end{corollary}

\begin{proof}
By Corollary~\ref{cor:cbf:bound} applied to the complete Bernstein functions $f^+(\xi)$ and $f^-(\eta)$, we have
\formula{
 \frac{|f(\xi)|}{|f(\eta)|} & = \frac{|f^+(-i \xi) f^-(i \xi)|}{|f^+(-i \eta) f^-(i \eta)|} \le C \, \frac{f^+(\xi)}{f^+(\eta)} \, \frac{f^-(\xi)}{f^-(\eta)} \, .
}
Using Lemma~\ref{lem:r:balance}, we find that if $\Arg \zeta_f(r) \le \thet$ for $r \ge \ro$, then
\formula{
 \frac{|f(\xi)|}{|f(\eta)|} & \le C \biggl(\frac{f^-(\xi)}{f^-(\eta)}\biggr)^{\! C(\thet)} \frac{f^-(\xi)}{f^-(\eta)}
}
when $2 \ro \le \eta < \xi$. This implies the first part of the corollary. The other one is proved in a very similar way.
\end{proof}

The next lemma provides a sufficient condition for a power-type behaviour of a Rogers function near $0$.

\begin{lemma}
\label{lem:r:growth}
Suppose that $f(\xi)$ is a non-constant Rogers function, $0 < \ro \le \infty$, $\eps > 0$ and $\delta \in [0, \tfrac{\pi}{2})$, and either
\formula{
 \Arg f(i e^{-i \delta} r) & \ge \eps
}
whenever $0 < r < \ro$, or
\formula{
 \Arg f(-i e^{i \delta} r) & \le -\eps
}
whenever $0 < r < \ro$ (if $\delta = 0$, we understand that $\Arg f(i e^{-i \delta} r) = \Arg f(i r) = \ph(-r)$ and $\Arg f(-i e^{i \delta} r) = \Arg f(-i r) = \ph(r)$, where $\ph(s)$ is the function in the exponential representation~\eqref{eq:r:exp} of the Rogers function $f(\xi)$). Then there is a constant $C(\eps, \delta) > 0$ such that
\formula[eq:r:growth]{
 \frac{|f(\xi)|}{|f(\eta)|} & \ge C(\eps, \delta) \biggl(\frac{\xi}{\eta}\biggr)^{\! \eps / (\pi - \delta)}
}
whenever $0 < \eta \le \xi < \ro$.
\end{lemma}

\begin{proof}
We only consider the case $\Arg f(i e^{-i \delta} r) \ge \eps$ when $0 < r < \ro$, with $\ro$ finite. The case $\ro = \infty$ is then an immediate extension, and the case $\Arg f(-i e^{i \delta} r) \le -\eps$ when $0 < r < \ro$ is completely analogous.

We first consider the case $\delta = 0$. With the notation of the exponential representation~\eqref{eq:r:exp}, we thus have $\ph(r) \ge \eps$ when $-\ro < r < 0$. Define the auxiliary Rogers functions
\formula{
 g(\xi) & = c \exp\biggl(\frac{1}{\pi} \int_{-\infty}^\infty \biggl(\frac{\xi}{\xi + i s} - \frac{1}{1 + |s|}\biggr) \frac{\ph(s) - \eps \ind_{(-\ro, 0)}(s)}{|s|} \, ds\biggr) , \\
 h(\xi) & = \exp\biggl(\frac{1}{\pi} \int_{-\ro}^0 \biggl(\frac{\xi}{\xi + i s} - \frac{1}{1 + |s|}\biggr) \frac{\eps}{|s|} \, ds\biggr) ,
}
so that $f(\xi) = g(\xi) h(\xi)$. By a simple calculation,
\formula{
 h(\xi) & = \biggl(\frac{(1 + \ro) i \xi}{i \xi + \ro}\biggr)^{\!\eps / \pi} ,
}
and $|g(\xi) / g(\eta)| \ge 1$ when $0 < \eta \le \xi$ by Proposition~\ref{prop:r:inc}. It follows that
\formula{
 \frac{|f(\xi)|}{|f(\eta)|} & = \frac{|g(\xi)|}{|g(\eta)|} \, \frac{|h(\xi)|}{|h(\eta)|} \ge \biggl(\frac{\xi |i \eta + \ro|}{\eta |i \xi + \ro|}\biggr)^{\! \eps / \pi}
}
whenever $0 < \eta \le \xi$. It remains to note that $|i \eta + \ro|^2 / |i \xi + \ro|^2 = (\eta^2 + \ro^2) / (\xi^2 + \ro^2) \ge \tfrac{1}{2}$ when additionally $\xi < \ro$.

Let us record the following observation. If $\re \xi > 0$ and $|\xi| < \ro$, then, by a short calculation,
\formula{
 \im \frac{i \xi}{(i \xi + \ro)^2} & = \frac{(\ro^2 - |\xi|^2) \re \xi}{|i \xi + \ro|^2} \ge 0 ,
}
and so $\Arg(i \xi) \ge 2 \Arg(i \xi + \ro)$. From the definition of a Rogers function, it follows that
\formula{
 \Arg f(\xi) & = \Arg g(\xi) + \Arg h(\xi) \\
 & = \Arg \frac{g(\xi)}{\xi} + \Arg \xi + \frac{\eps}{\pi} \, (\Arg(i \xi) - \Arg(i \xi + \ro)) \\
 & \ge -\frac{\pi}{2} + \Arg \xi + \frac{\eps}{2 \pi} \Arg (i \xi) \\
 & = \biggl(1 + \frac{\eps}{2 \pi}\biggr) \Arg \xi - \biggl(1 - \frac{\eps}{2 \pi}\biggr) \frac{\pi}{2} \, .
}
In particular,
\formula[eq:r:growth:aux]{
 \Arg f(\xi) & \ge \biggl(1 - \frac{\eps}{2 \pi}\biggr) \frac{\eps}{4}
}
when $\Arg \xi = (1 - \tfrac{\eps}{2 \pi}) \tfrac{\pi}{2}$ and $|\xi| < \ro$. In other words, if the assumption of the lemma is satisfied with $\delta = 0$, then it is also satisfied for some $\delta > 0$ (with $\eps$ replaced by the right-hand side of~\eqref{eq:r:growth:aux}).

Let us now consider the case $0 < \delta < \tfrac{\pi}{2}$. We define an auxiliary function
\formula{
 h(\xi) & = e^{i \delta/2} \xi^{\delta / \pi} f(e^{-i \delta/2} \xi^{1 - \delta / \pi}) .
}
Note that $g(\xi) = e^{i \delta/2} \xi^{\delta / \pi}$ is a Rogers function, and hence, by Proposition~\ref{prop:r:prop}\ref{it:r:prop:d}, $h(\xi)$ is a Rogers function. Furthermore, if $0 < r < \ro^{\pi / (\pi - \delta)}$, we have
\formula{
 \Arg h(-i r) & = \Arg\bigl(e^{i \delta} r^{\delta / \pi} f(i e^{-i \delta} r^{1 - \delta / \pi})\bigr) \ge \delta + \eps .
}
Therefore, by the first part of the proof,
\formula{
 \frac{|h(\xi)|}{|h(\eta)|} & \ge C(\eps, \delta) \biggl(\frac{\xi}{\eta}\biggr)^{\! (\delta + \eps) / \pi}
}
whenever $0 < \eta \le \xi < \ro^{\pi / (\pi - \delta)}$. Using the definition of $h$, we find that
\formula{
 \frac{|f(e^{-i \delta/2} \xi^{1 - \delta / \pi})|}{|f(e^{-i \delta/2} \eta^{1 - \delta / \pi})|} & \ge C(\eps, \delta) \biggl(\frac{\xi}{\eta}\biggr)^{\! \eps / \pi} .
}
Finally, combining the above estimate with Proposition~\ref{prop:r:bound} (applied with $r$ and $\xi$ replaced by $\xi^{1 - \delta / \pi}$ and $e^{-i \delta/2} \xi^{1 - \delta / \pi}$ for the numerator, and by $\eta^{1 - \delta / \pi}$ and $e^{-i \delta/2} \eta^{1 - \delta / \pi}$ for the denominator), we conclude that
\formula{
 \frac{|f(\xi^{1 - \delta / \pi})|}{|f(\eta^{1 - \delta / \pi})|} & \ge C(\eps, \delta) \biggl(\frac{\xi}{\eta}\biggr)^{\! \eps / \pi}
}
when $0 < \eta \le \xi < \ro^{\pi / (\pi - \delta)}$, which is equivalent to~\eqref{eq:r:growth}.
\end{proof}

\begin{lemma}
\label{lem:r:balance:quot}
Suppose that $0 < \ro \le \infty$, $\eps > 0$ and $\delta \in [0, \tfrac{\pi}{2})$. Suppose that $f(\xi)$ is a non-constant Rogers function, and we have
\formula{
 \Arg f(i r e^{-i \delta}) & \ge \eps ,
}
whenever $0 < r < \ro$ (if $\delta = 0$, we understand that $\Arg f(-i e^{i \delta} r) = \Arg f(-i r) = \ph(r)$ and $\Arg f(i e^{-i \delta} r) = \Arg f(i r) = \ph(-r)$, where $\ph(s)$ is the function in the exponential representation~\eqref{eq:r:exp} of the Rogers function $f(\xi)$). Suppose also that
\formula{
 \inf \{ \Arg \zeta_f(r) : r \in (0, \ro) \} & > -\frac{\pi}{2} \, .
}
If $r \in Z_f$, then, for some constants $C_1(\eps, \delta), C_2(\eps, \delta), C_3(\eps, \delta, f) > 0$, the Wiener--Hopf factor $f^-(r; \xi)$ of the Rogers function $f(r; \xi)$ satisfies
\formula{
 \frac{f^-(r; \xi)}{f^-(r; \eta)} & \le C_1(\eps, \delta) \biggl( \frac{\xi}{\eta} \biggr)^{\! 1 - C_2(\eps, \delta)}
}
whenever $0 < C_3(\eps, \delta, f) r \le \eta < \xi < \tfrac{1}{2} \ro$.

Similarly, if 
\formula{
 \Arg f(-i r e^{i \delta}) & \le -\eps
}
whenever $0 < r < \ro$, and
\formula{
 \sup \{ \Arg \zeta_f(r) : r \in (0, \ro) \} & < \frac{\pi}{2} \, ,
}
then
\formula{
 \frac{f^+(r; \xi)}{f^+(r; \eta)} & \le C_1 \biggl( \frac{\xi}{\eta} \biggr)^{\! 1 - C_2(\eps, \delta)}
}
whenever $0 < C_3(\eps, \delta, f) r \le \eta < \xi \le \tfrac{1}{2} \ro$.
\end{lemma}

\begin{figure}
\centering
\begin{tikzpicture}
\footnotesize
\coordinate (X) at (3,0);
\coordinate (Y) at (0,3);
\coordinate (Xn) at (-1,0);
\coordinate (Yn) at (0,-3);
\coordinate (O) at (0,0);
\coordinate (R) at (0.4159,2.4954);
\coordinate (S) at (0.4159,-2.4954);
\coordinate (T) at (0.8,2.4);
\coordinate (Ry) at (0,2.52982);
\coordinate (Sy) at (0,-2.52982);
\fill[fill=olive!50!white, draw=olive] (O) -- (S) arc (-80.54:-90:2.52982) -- (O);
\node[below right, olive] at (Sy) {$\Arg f < 0$};
\fill[fill=teal!50!white, draw=teal] (O) -- (R) arc (80.54:90:2.52982) -- (O);
\node[above right, teal] at (Ry) {$\Arg f > 0$};
\draw[very thick, violet] (O) -- (T) node[midway, right] {$\inf \Arg f \ge \eps$};
\pic["$\delta$",draw,angle eccentricity=1.25,angle radius=0.75cm] {angle=T--O--Y};
\pic["$\thet$",draw,angle eccentricity=1.45,angle radius=0.4cm] {angle=X--O--R};
\pic["$\thet$",draw,angle eccentricity=1.35,angle radius=0.5cm] {angle=S--O--X};
\draw[->] (Xn) -- (X) node[above] {$\re$};
\draw[->] (Yn) -- (Y) node[left] {$\im$};
\end{tikzpicture}
\caption{Setting for the proof of Lemma~\ref{lem:r:balance:quot}.}
\label{fig:balance}
\end{figure}
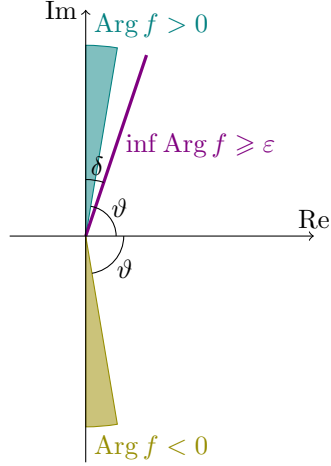

\begin{proof}
As both parts of the lemma are completely analogous, we only prove the first statement. As it was noted in the proof of Lemma~\ref{lem:r:growth} (see~\eqref{eq:r:growth:aux}), with no loss of generality we may assume that $\delta > 0$. We fix $r \in Z_f$, and we consider the Rogers function
\formula[eq:r:balance:quot:aux0]{
 g(\xi) & = \frac{\xi^2}{f(r; \xi)} = \frac{\xi^2}{(\xi - \zeta_f(r)) (\xi + \overline{\zeta_f(r)})} \, (f(\xi) - \lambda_f(r)) ,
}
introduced in Proposition~\ref{prop:r:quot}. We will show that Corollary~\ref{cor:r:balance} applies to $g(\xi)$, and then use this fact to prove the desired estimate for the Wiener--Hopf factors of $f(r; \xi)$. We will frequently use the estimate
\formula{
 \lvert\Arg (1 - \xi)\rvert & \le \lvert\log (1 - \xi)\rvert \le 2 |\xi| ,
}
valid when $|\xi| \le \tfrac{1}{2}$. We let
\formula{
 \thet & = \sup \{ \lvert\Arg \zeta_f(r)\rvert : r \in (0, \ro) \} .
}
Note that $\thet < \tfrac{\pi}{2}$.

\smallskip

\emph{Step 1.}
We first estimate $\Arg (f(\xi) - \lambda_f(r))$ when $\Arg \xi = \tfrac{\pi}{2} - \delta$ and $|\xi| / r$ is sufficiently large. Observe that if $|f(\xi)| \ge 2 \lambda_f(r)$, then 
\formula{
 \lvert\Arg (f(\xi) - \lambda_f(r)) - \Arg f(\xi)\rvert & = \biggl| \Arg \biggl( 1 - \frac{\lambda_f(r)}{f(\xi)} \biggr) \biggr| \le 2 \, \frac{|\lambda_f(r)|}{|f(\xi)|} \, .
}
In particular, if additionally $|f(\xi)| \ge 4 \eps^{-1} \lambda_f(r)$, then
\formula[eq:r:balance:quot:aux1]{
 \lvert\Arg (f(\xi) - \lambda_f(r)) - \Arg f(\xi)\rvert & \le \frac{\eps}{2} \, .
}
We now find the lower bound on $|f(\xi)|$ using Proposition~\ref{prop:r:bound} and Lemma~\ref{lem:r:growth}. Recall that $\Arg \xi = \tfrac{\pi}{2} - \delta$, $\lambda_f(r) = f(\zeta_f(r))$ and $\lvert\Arg \zeta_f(r)\rvert \le \thet$. Thus, if $0 < r \le |\xi| < \ro$, we have
\formula{
 |f(\xi)| & \ge C(\delta) f(|\xi|) \ge C(\eps, \delta) \biggl(\frac{|\xi|}{r}\biggr)^{\! \eps / (\pi - \delta)} |f(r)| \ge C(\eps, \delta, \thet) \biggl(\frac{|\xi|}{r}\biggr)^{\! \eps / (\pi - \delta)} \lambda_f(r)
}
(we applied Proposition~\ref{prop:r:bound} with $r$ and $\xi$ replaced by $|\xi|$ and $\xi$ in the former inequality, and by $\zeta_f(r)$ and $r$ in the latter one). In particular, if $\Arg \xi = \tfrac{\pi}{2} - \delta$ and $0 < C(\eps, \delta, \thet) r \le |\xi| < \ro$, then $|f(\xi)| > 4 \eps^{-1} \lambda_f(r)$, and consequently~\eqref{eq:r:balance:quot:aux1} holds.

\smallskip

\emph{Step 2.}
We now turn to the estimate of the argument of the other factor on the right-hand side of~\eqref{eq:r:balance:quot:aux0}. If $|\xi| \ge 2 r$, we have
\formula{
 \biggl| \Arg \frac{\xi^2}{(\xi - \zeta_f(r)) (\xi + \overline{\zeta_f(r)})} \biggr| & = \biggl| \Arg \biggl( 1 - \frac{\zeta_f(r)}{\xi} \biggr) + \Arg \biggl( 1 + \frac{\overline{\zeta_f(r)}}{\xi} \biggr) \biggr| \\
 & \le 2 \, \frac{|\zeta_f(r)|}{|\xi|} + 2 \, \frac{|\overline{\zeta_f(r)}|}{|\xi|} = 4 \, \frac{r}{|\xi|} \, .
}
In particular, if $|\xi| > 8 r \eps^{-1}$, then
\formula{
 \biggl| \Arg \frac{\xi^2}{(\xi - \zeta_f(r)) (\xi + \overline{\zeta_f(r)})} \biggr| & < \frac{\eps}{2} \, .
}

\smallskip

\emph{Step 3.}
By combining the estimates from the previous two steps, we find that if $\Arg \xi = \tfrac{\pi}{2} - \delta$ and $0 < C(\eps, \delta, \thet) r < |\xi| < \ro$, then
\formula{
 \Arg g(\xi) & = \Arg \frac{\xi^2}{(\xi - \zeta_f(r)) (\xi + \overline{\zeta_f(r)})} + \Arg (f(\xi) - f(\zeta_f(r))) \\
 & \ge -\frac{\eps}{2} + \biggl(\Arg f(\xi) - \frac{\eps}{2}\biggr) \ge 0 .
}
Consequently, if $s = |\xi|$, then $\Arg \xi > \Arg \zeta_g(s)$ by Proposition~\ref{prop:r:real}\ref{it:r:real:b}. Thus, $\Arg \zeta_g(s) \le \tfrac{\pi}{2} - \delta$ whenever $0 < C(\eps, \delta, \thet) r < s < \ro$. By Corollary~\ref{cor:r:balance} we find that
\formula{
 \frac{g^-(\xi)}{g^-(\eta)} & \ge C \biggl( \frac{|g(\xi)|}{|g(\eta)|} \biggr)^{\! C(\delta)}
}
whenever $0 < C(\eps, \delta, \thet) r < \eta < \xi \le \tfrac{1}{2} \ro$.

\smallskip

\emph{Step 4.}
Clearly, $f^+(r; \xi) = \xi / g^+(\xi)$ and $f^-(r; \xi) = \xi / g^-(\xi)$ are the Wiener--Hopf factors of the Rogers function $f(r; \xi) = \xi^2 / g(\xi) = f^+(r; -i \xi) f^-(r; i \xi)$. Therefore, 
\formula{
 \frac{f^-(r; \xi)}{f^-(r; \eta)} & = \frac{\xi g^-(\eta)}{\eta g^-(\xi)} \le C \, \frac{\xi}{\eta} \biggl( \frac{|g(\eta)|}{|g(\xi)|} \biggr)^{\! C(\delta)}
}
whenever $0 < C(\eps, \delta, \thet) r \le \eta < \xi \le \tfrac{1}{2} \ro$. Finally, when $0 < C(\eps, \delta, \thet) r \le \xi < \ro$, then, as in Step~1,
\formula{
 |f(\xi)| & \ge C(\eps, \delta) \biggl(\frac{\xi}{r}\biggr)^{\! \eps / (\pi - \delta)} |f(r)| \ge C(\eps, \delta, \thet) \biggl(\frac{\xi}{r}\biggr)^{\! \eps / (\pi - \delta)} \lambda_f(r) \ge 2 \lambda_f(r) , 
}
and hence
\formula{
 |g(\xi)| & = \frac{\xi^2}{|\xi - \zeta_f(r)| |\xi + \overline{\zeta_f(r)}|} \, |f(\xi) - \lambda_f(r)| \ge \frac{\xi^2}{(\xi + r)^2} \, (|f(\xi)| - \lambda_f(r)) \ge \frac{|f(\xi)|}{8} \, ,
}
and
\formula{
 |g(\xi)| & = \frac{\xi^2}{|\xi - \zeta_f(r)| |\xi + \overline{\zeta_f(r)}|} \, |f(\xi) - \lambda_f(r)| \le \frac{\xi^2}{(\xi - r)^2} \, (|f(\xi)| + \lambda_f(r)) \le 8 |f(\xi)| .
}
It follows that if $0 < C(\eps, \delta, \thet) r \le \eta < \xi \le \tfrac{1}{2} \ro$, then
\formula{
 \frac{f^-(r; \xi)}{f^-(r; \eta)} & \le C \, \frac{\xi}{\eta} \biggl( \frac{|f(\eta)|}{|f(\xi)|} \biggr)^{\! C(\delta)} .
}
It remains to again apply Lemma~\ref{lem:r:growth} to bound $|f(\eta)| / |f(\xi)|$ by $C(\eps, \delta) (\eta / \xi)^{\eps / (\pi - \delta)}$.
\end{proof}

\subsection{Notation}
\label{sec:notation}

For reader's convenience, we gather the notation used in the remaining part of the article. We always assume that $f(\xi)$ is a non-constant and non-degenerate Rogers function, and we use freely the following symbols:
\begin{itemize}
\item $D_f \sub \C$ is the domain of $f(\xi)$, equal to $\C \setminus i \R$ possibly augmented by an appropriate part of $i \R$;
\item $\Gamma_f$ is the spine of $f(\xi)$, the set of $\zeta \in \C$ such that $\re \zeta > 0$ and $f(\zeta) \in (0, \infty)$;
\item $\zeta_f(r)$ is the parameterisation of the spine $\Gamma_f$, augmented by appropriate segments of $i \R$: for every $r > 0$, $\re \zeta_f(r) = r e^{i \thet}$, where $\thet \in [-\tfrac{\pi}{2}, \tfrac{\pi}{2}]$ and $\sign \im f(e^{i \alpha} r) = \sign (\alpha - \thet)$ for every $\alpha \in (-\tfrac{\pi}{2}, \tfrac{\pi}{2})$ (see Proposition~\ref{prop:r:real}\ref{it:r:real:b});
\item $Z_f$ is a subset of $(0, \infty)$, consisting of those $r > 0$ for which $\re \zeta_f(r) > 0$; equivalently, $Z_f$ is the set of moduli of points on $\Gamma_f$;
\item $\lambda_f(r)$ describes the values of $f(\xi)$ along $\Gamma_f$: $\lambda_f(r) = f(\zeta_f(r))$ for $r \in Z_f$ (or, more generally, whenever $\zeta_f(r) \in D_f$), and $\lambda_f(r)$ is an increasing continuous function of $r > 0$;
\item $\Gamma_f^\star$ is the symmetrised spine of $f(\xi)$: the union of $\Gamma_f$, its mirror image $-\overline{\Gamma_f}$, and the endpoints of $\Gamma_f$;
\item for $r \in Z_f$, $f(r; \xi)$ is the (inverse) difference quotient of $f(\xi)$, defined by~\eqref{eq:r:fr}, and $f^+(r; \xi)$ and $f^-(r; \xi)$ are the Wiener--Hopf factors of the Rogers function $f(r; \xi)$.
\end{itemize}
For simplicity, in all proofs we drop subscript $f$ from the notation whenever this causes no confusion. As we have already done a few times above, in the remaining part of the article $C$ denotes a generic positive constant that may depend on parameters listed in brackets, and the value of $C$ may change even within a single expression.

%
%

\section{Inversion of temporal Laplace transform}
\label{sec:heat}

This section contains the first part of the proof of our main results, Theorem~\ref{thm:heat} and~Theorem~\ref{thm:extrema}. The argument is completed in the next section. We follow, with major modifications, the approach of~\cite{kk:stable}, where strictly stable Lévy processes were studied, and we heavily use the results of~\cite{kwasnicki:rogers} and Section~\ref{sec:rogers}. In particular, we use the notation summarised in Section~\ref{sec:notation}.

Except in section~\ref{sec:pr}, throughout this part we assume that $X_t$ is a (non-killed and non-constant) Lévy process with completely monotone jumps, and $f(\xi)$ is the corresponding Rogers function (the holomorphic extension of the characteristic exponent of $X_t)$. Note that $f(\xi)$ is non-zero and we have $f(0^+) = 0$, so that $\lambda_f(0^+) = 0$. Later on, we will impose additional assumptions on $f(\xi)$. In particular, we will have $\lambda_f(\infty) = \infty$, which implies that the transition kernel $p_t(x, dy)$ of $X_t$ has a density function $p_t(y - x)$, and consequently also $p_t^+(x, dy)$ has a density function $p_t^+(x, y)$.

For notational convenience, except in statements of results, we commonly drop the subscript $f$ from the notation. For example, we write $\zeta(r)$, $\lambda(r)$ and $\Gamma$ rather than $\zeta_f(r)$, $\lambda_f(r)$ and $\Gamma_f$.

\subsection{Pecherskii--Rogozin-type expression}
\label{sec:pr}

Suppose that $\sigma > 0$, $\re \xi > 0$ and $\re \eta > 0$. Our first goal is to express the tri-variate Laplace transform
\formula{
 & \int_0^\infty \int_0^\infty \int_{(0, \infty)} e^{-\sigma t - \eta x - \xi y} p_t^+(x, dy) dx dt 
}
in terms of the Wiener--Hopf factors $\kappa^+(\sigma, \xi)$ and $\kappa^-(\sigma, \eta)$. Our method is quite standard, see, for example, Section~2.1 in~\cite{kk:stable}, and it applies to an arbitrary Lévy process $X_t$. It is based on the following fundamental result in fluctuation theory of Lévy processes.

\begin{proposition}[see Theorem~6.15(i) and identity~(6.28) in~\cite{kyprianou}]
Suppose that $X_t$ is a Lévy process, $\sigma > 0$, $\exprv$ is an exponentially distributed random time with mean $1 / \sigma$, and $\exprv$ is independent from the process $X_t$. Then the random variables $\ul{X}_{\exprv} - X_0$ and $X_{\exprv} - \ul{X}_{\exprv}$ are independent, and the random variables $\ol{X}_{\exprv} - X_0$ and $X_{\exprv} - \ul{X}_{\exprv}$ have equal distribution. 
\end{proposition}

For $\exprv$ as in the above proposition, we have 
\formula{
 \int_0^\infty \int_{(0, \infty)} e^{-\sigma t - \xi y} p_t^+(x, dy) dt & = \int_0^\infty e^{-\sigma t} \ex^x \bigl(\exp(-\xi X_t) \ind_{(0, \infty)}(\ul{X}_t)\bigr) dt \\
 & = \frac{1}{\sigma} \, \ex^x \bigl(\exp(-\xi X_{\exprv}) \ind_{(0, \infty)}(\ul{X}_{\exprv})\bigr) .
}
By the above proposition,
\formula{
 \ex^x\bigl(\exp(-\xi X_{\exprv}) \ind_{(0, \infty)}(\ul{X}_{\exprv})\bigr) & = \ex^x\bigl(\exp(-\xi (X_{\exprv} - \ul{X}_{\exprv})) \exp(-\xi \ul{X}_{\exprv}) \ind_{(0, \infty)}(\ul{X}_{\exprv})\bigr) \\
 & = \ex^x\bigl(\exp(-\xi (X_{\exprv} - \ul{X}_{\exprv}))\bigr) \ex^x\bigl(\exp(-\xi \ul{X}_{\exprv}) \ind_{(0, \infty)}(\ul{X}_{\exprv})\bigr) \\
 & = \ex^x\bigl(\exp(-\xi (\ol{X}_{\exprv} - x))\bigr) \ex^x\bigl(\exp(-\xi \ul{X}_{\exprv}) \ind_{(0, \infty)}(\ul{X}_{\exprv})\bigr) .
}
Translation invariance implies that
\formula{
 \ex^x\bigl(\exp(-\xi (\ol{X}_{\exprv} - x))\bigr) & = \ex^0\bigl(\exp(-\xi \ol{X}_{\exprv})\bigr)
}
and
\formula{
 \ex^x\bigl(\exp(-\xi \ul{X}_{\exprv}) \ind_{(0, \infty)}(\ul{X}_{\exprv})\bigr) & = \ex^0\bigl(\exp(-\xi \ul{X}_{\exprv}) e^{-\xi x} \ind_{(-x, \infty)}(\ul{X}_{\exprv})\bigr) .
}
It follows that
\formula{
 \int_0^\infty \int_{(0, \infty)} e^{-\sigma t - \xi y} p_t^+(x, dy) dt & = \frac{1}{\sigma} \, \ex^0\bigl(\exp(-\xi \ol{X}_{\exprv})\bigr) \ex^0\bigl(\exp(-\xi \ul{X}_{\exprv}) e^{-\xi x} \ind_{(-x, \infty)}(\ul{X}_{\exprv})\bigr)
}
Since for $y < 0$ we have
\formula{
 \int_0^\infty e^{-\eta x} e^{-\xi x} \ind_{(-x, \infty)}(y) dx & = \int_{-y}^\infty e^{-(\xi + \eta) x} dx = \frac{e^{(\xi + \eta) y}}{\xi + \eta} \, ,
}
we eventually find that
\formula{
 \int_0^\infty \int_0^\infty \int_{(0, \infty)} e^{-\sigma t - \eta x - \xi y} p_t^+(x, dy) dx dt & = \frac{1}{\sigma (\xi + \eta)} \, \ex^0\bigl(\exp(-\xi \ol{X}_{\exprv})\bigr) \ex^0\bigl(\exp(\eta \ul{X}_{\exprv})\bigr) .
}
By the Pecherskii--Rogozin identities~\eqref{eq:pr},
\formula{
 \ex^0 \exp(-\xi \ol{X}_{\exprv}) & = \frac{\kappa^+(\sigma, 0)}{\kappa^+(\sigma, \xi)} \, , & \ex^0 \exp(\eta \ul{X}_{\exprv}) & = \frac{\kappa^-(\sigma, 0)}{\kappa^-(\sigma, \eta)} \, .
}
Furthermore, by the factorisation identity~\eqref{eq:pr:fact}
\formula{
 \kappa^+(\sigma, 0) \kappa^-(\sigma, 0) & = \frac{\sigma}{\kappa^\circ(\sigma)} \, .
}
We have thus proved the following result, which seems to be rather standard, but difficult to trace in the literature in this particular form; see Theorem~18 in~\cite{doney-2} or Theorem~7.7 in~\cite{kyprianou} for related developments.

\begin{proposition}
\label{prop:pt:pr}
If $X_t$ is a Lévy process, $p_t^+(x, dy) = \pr^x(X_t \in dy, \ul{X}_t > 0)$ is its transition kernel in the half-line $(0, \infty)$, and $\kappa^+(\sigma, \xi)$, $\kappa^-(\sigma, \eta)$ and $\kappa^\circ(\sigma)$ are the corresponding Wiener--Hopf factors, then
\formula[eq:pt:pr]{
 \int_0^\infty \int_0^\infty \int_{(0, \infty)} e^{-\sigma t - \eta x - \xi y} p_t^+(x, dy) dx dt & = \frac{1}{\xi + \eta} \, \frac{1}{\kappa^\circ(\sigma) \kappa^+(\sigma, \xi) \kappa^-(\sigma, \eta)}
}
whenever $\sigma > 0$, $\re \xi > 0$ and $\re \eta > 0$.
\end{proposition}

\subsection{Baxter--Donsker-type expression and inversion of the Laplace transform}

By Theorem~1.1 in~\cite{kwasnicki:rogers}, the right-hand side of~\eqref{eq:pt:pr} is a Stieltjes function of~$\sigma$. This allows us to invert the Laplace transform with respect to the temporal variable ($t \mapsto \sigma$) in~\eqref{eq:pt:pr}. To this end, we use integral expressions for the Wiener--Hopf factors $\kappa^+(\sigma, \xi)$ and $\kappa^-(\sigma, \eta)$ in terms of the characteristic exponent $f(\xi)$, developed in~\cite{kwasnicki:rogers}.

Let $f_\sigma(\xi) = \sigma + f(\xi)$ for $\sigma > 0$. By formula~(6.2) in~\cite{kwasnicki:rogers}, we have
\formula{
 \kappa^\circ(\sigma) \kappa^+(\sigma, \xi) \kappa^-(\sigma, \eta) & = f_\sigma^+(\xi) f_\sigma^-(\eta)
}
when $\sigma > 0$, $\re \xi > 0$ and $\re \eta > 0$; here $f_\sigma^+(\xi)$ and $f_\sigma^-(\eta)$ are the Wiener--Hopf factors of the Rogers function $f_\sigma(\xi)$, see Proposition~\ref{prop:r:wh}. Using Proposition~\ref{prop:r:wh:bd} we find that for $\sigma, \xi, \eta > 0$,
\formula{
 \kappa^\circ(\sigma) \kappa^+(\sigma, \xi) \kappa^-(\sigma, \eta) & = \sigma \exp\biggl(\frac{1}{\pi} \int_0^\infty \psi(r) \, \frac{d\lambda(r)}{\sigma + \lambda(r)} \biggr) \\
 & = \sigma \exp\biggl(\frac{1}{\pi} \int_0^{\lambda(\infty)} \psi(\lambda^{-1}(s)) \, \frac{ds}{\sigma + s} \biggr) ,
}
where
\formula{
 \psi(r) & = \Arg(\zeta(r) + i \eta) - \Arg(\zeta(r) - i \xi) .
}
Note that $\psi(r)$ takes values in $[0, \pi]$, and $\psi(0^+) = \pi$, $\psi(\infty) = 0$; see the proof of Lemma~6.1 in~\cite{kwasnicki:rogers} for further details. It follows that
\formula*[eq:pt:bd]{
 & \int_0^\infty \int_0^\infty \int_{(0, \infty)} e^{-\sigma t - \eta x - \xi y} p_t^+(x, dy) dx dt \\
 & \qquad = \frac{1}{\sigma (\xi + \eta)} \, \exp\biggl(-\frac{1}{\pi} \int_0^{\lambda(\infty)} \psi(\lambda^{-1}(s)) \, \frac{ds}{\sigma + s} \biggr) .
}
Clearly, the right-hand side of~\eqref{eq:pt:bd} extends to a holomorphic function of $\sigma \in \C \setminus [-\lambda(\infty), 0]$. We denote this function by $\Psi(\sigma)$ (for fixed $\xi, \eta > 0$). By Proposition~\ref{prop:r:wh:bd}, $1 / \Psi(\sigma)$ is a complete Bernstein function, and hence $\Psi(\sigma)$ is a Stieltjes function. We have thus essentially proved the following result.

\begin{proposition}
\label{prop:pt:laplace}
If $X_t$ is a Lévy process with completely monotone jumps, $p_t^+(x, dy) = \pr^x(X_t \in dy, \ul{X}_t > 0)$ is its transition kernel in the half-line $(0, \infty)$, and $f(\xi)$ is the corresponding Rogers function, then
\formula[eq:pt:laplace]{
 \int_0^\infty \int_{(0, \infty)} e^{-\eta x - \xi y} p_t^+(x, dy) dx & = \frac{1}{\pi} \int_{(0, \lambda_f(\infty))} e^{-s t} \mu_f(\xi, \eta, ds)
}
for $t \ge 0$ and $\xi, \eta > 0$. Here $\mu_f(\xi, \eta, ds)$ is a nonnegative measure on $(0, \lambda_f(\infty))$, given by
\formula[eq:pt:mu]{
 \mu_f(\xi, \eta, ds) & = -\lim_{t \to 0^+} \im \Psi_f(\xi, \eta, -s + i t) ds
}
in the sense of vague convergence of measures on $[0, \lambda_f(\infty)]$, with
\formula[eq:pt:phi]{
 \Psi_f(\xi, \eta, \sigma) & = \frac{1}{\sigma (\xi + \eta)} \, \exp\biggl(-\frac{1}{\pi} \int_0^{\lambda_f(\infty)} \psi_f(\xi, \eta, \lambda_f^{-1}(s)) \, \frac{ds}{\sigma + s} \biggr)
}
for $\sigma \in \C \setminus (-\infty, 0]$, and
\formula[eq:pt:theta]{
 \psi_f(\xi, \eta, r) & = \Arg(\zeta_f(r) + i \eta) - \Arg(\zeta_f(r) - i \xi)
}
for $r > 0$.
\end{proposition}

\begin{proof}
We continue to use the simplified notation $\Psi(\sigma) = \Psi_f(\xi, \eta, \sigma)$ and $\psi(r) = \psi_f(\xi, \eta, r)$. We have already observed that $\Psi(\sigma)$ is a Stieltjes function and $\Psi(\sigma)$ is holomorphic on $\C \setminus [-\lambda(\infty), 0]$, so that $\Psi(\sigma)$ has the following Stieltjes representation~\eqref{eq:s:int}: when $\sigma \in \C \setminus [-\lambda(\infty), 0]$, we have
\formula[eq:pt:phi:int]{
 \Psi(\sigma) & = \frac{b}{\sigma} + c + \frac{1}{\pi} \int_{(0, \lambda(\infty)]} \frac{1}{\sigma + s} \, \mu(ds)
}
for some $b, c \ge 0$ and an appropriate measure $\mu(ds) = \mu_f(\xi, \eta, ds)$. In order to complete the proof, we only need to prove that $b = c = 0$ and $\mu(\{\lambda(\infty)\}) = 0$. Indeed: the Stieltjes measure $\mu(ds)$ is given by~\eqref{eq:pt:mu}, and formula~\eqref{eq:pt:laplace} follows from~\eqref{eq:pt:bd} and continuity of the left-hand side.

By~\eqref{eq:pt:phi} and the monotone convergence theorem, we have
\formula{
 b & = \lim_{\sigma \to 0^+} \sigma \Psi(\sigma) = \frac{1}{\xi + \eta} \, \exp\biggl(-\frac{1}{\pi} \int_0^{\lambda(\infty)} \psi(\lambda^{-1}(s)) \, \frac{ds}{s} \biggr) .
}
Since the measure $ds / s$ is infinite in every right neighbourhood of $0$, $\lambda^{-1}(0^+) = 0$ and $\psi(0^+) = \pi$, the above integral diverges, and, consequently, $b = 0$. Furthermore, again by~\eqref{eq:pt:phi},
\formula{
 c & = \lim_{\sigma \to \infty} \Psi(\sigma) \le \lim_{\sigma \to \infty} \frac{1}{\sigma (\xi + \eta)} = 0 ,
}
so that also $c = 0$. Finally, let $\sigma_0 = \lambda(\infty)$ and suppose that $\sigma_0 < \infty$. Then $\mu(ds)$ is concentrated on $(0, \sigma_0]$, $\Psi(-\sigma)$ is given by~\eqref{eq:pt:phi:int} for $\sigma > \sigma_0$, and we have, by the dominated convergence theorem,
\formula{
 \mu(\{\sigma_0\}) & = \lim_{\sigma \to \sigma_0^+} \int_{(0, \sigma_0]} \frac{\sigma - \sigma_0}{\sigma - s} \, \mu(ds) = \lim_{\sigma \to \sigma_0^+} (-\pi (\sigma - \sigma_0) \Psi(-\sigma)) .
}
In order to prove that $\mu(\{\sigma_0\}) = 0$, we observe that, by~\eqref{eq:pt:phi},
\formula{
 \mu(\{\sigma_0\}) & = \frac{\pi}{\xi + \eta} \, \lim_{\sigma \to \sigma_0^+} \frac{\sigma - \sigma_0}{\sigma_0} \, \exp\biggl(\frac{1}{\pi} \int_0^{\sigma_0} \psi(\lambda^{-1}(s)) \, \frac{ds}{\sigma - s} \biggr) \\
 & = \frac{\pi}{\xi + \eta} \, \lim_{\sigma \to \sigma_0^+} \frac{\sigma}{\sigma_0} \, \exp\biggl(-\frac{1}{\pi} \int_0^{\sigma_0} (\pi - \psi(\lambda^{-1}(s))) \, \frac{ds}{\sigma - s} \biggr) .
}
By the monotone convergence theorem, we find that
\formula{
 \mu(\{\sigma_0\}) & = \frac{\pi}{\xi + \eta} \, \exp\biggl(-\frac{1}{\pi} \int_0^{\sigma_0} (\pi - \psi(\lambda^{-1}(s))) \, \frac{ds}{\sigma_0 - s} \biggr) .
}
Since $\lambda^{-1}(\sigma_0^-) = \infty$ and $\psi(\infty) = 0$, the above integral diverges, and consequently $\mu(\{\sigma_0\}) = 0$.
\end{proof}

\subsection{Reformulation in terms of difference quotients}

Our goal in this section is to transform the expression for $\mu_f(\xi, \eta, ds)$ in Proposition~\ref{prop:pt:laplace} to a more manageable form. Recall that in~\eqref{eq:r:fr} we defined
\formula{
 f(r; \xi) & = \frac{(\xi - \zeta_f(r)) (\xi + \overline{\zeta_f(r)})}{f(\xi) - \lambda_f(r)}
}
whenever $r \in Z_f$, and that $f(r; \xi)$ is a Rogers function of $\xi$. We denote by $f^+(r; \xi)$ and $f^-(r; \xi)$ the Wiener--Hopf factors of this Rogers function.

\begin{proposition}
\label{prop:pt:quot}
Suppose that $X_t$ is a Lévy process with completely monotone jumps, and $f(\xi)$ is the corresponding Rogers function. Let $t \ge 0$, $\xi, \eta > 0$, and suppose that $i \xi \in D_f^+$ and $-i \eta \in D_f^-$. Then
\formula{
 & \int_0^\infty \int_{(0, \infty)} e^{-\eta x - \xi y} p_t^+(x, dy) dx \\
 & \qquad = \frac{1}{\pi} \int_{Z_f} e^{-t \lambda_f(r)} \, \frac{f^+(r; \xi)}{|\zeta_f(r) - i \xi|^2} \, \frac{f^-(r; \eta)}{|\zeta_f(r) + i \eta|^2} \, \re \zeta_f(r) d\lambda_f(r) .
}
Similarly, if $t \ge 0$, $\xi, \eta > 0$ and $i \xi, -i \eta \in D_f^+$, then
\formula{
 & \int_0^\infty \int_{(0, \infty)} e^{-\eta x - \xi y} p_t^+(x, dy) dx \\
 & \qquad = \frac{1}{\pi} \int_{Z_f} e^{-t \lambda_f(r)} \, \frac{f^+(r; \xi)}{|\zeta_f(r) - i \xi|^2} \, \frac{-1}{f^+(r; -\eta)} \, \re \zeta_f(r) d\lambda_f(r) ,
}
while if $t \ge 0$, $\xi, \eta > 0$ and $i \xi, -i \eta \in D_f^-$, then
\formula{
 & \int_0^\infty \int_{(0, \infty)} e^{-\eta x - \xi y} p_t^+(x, dy) dx \\
 & \qquad = \frac{1}{\pi} \int_{Z_f} e^{-t \lambda_f(r)} \, \frac{-1}{f^-(r; -\xi)} \, \frac{f^-(r; \eta)}{|\zeta_f(r) + i \eta|^2} \, \re \zeta_f(r) d\lambda_f(r) .
}
\end{proposition}

\begin{proof}
We continue to use the notation $\Psi(\sigma) = \Psi_f(\xi, \eta, \sigma)$, $\mu(ds) = \mu_f(\xi, \eta, ds)$ and $\psi(r) = \psi_f(\xi, \eta, r)$ introduced in the proof of Proposition~\ref{prop:pt:laplace}. 

Suppose that $i \xi \in D_f^+$ and $-i \eta \in D_f^-$. Observe that $\tilde{\psi}(z) = \Arg(z + i \eta) - \Arg(z - i \xi)$ is a differentiable function in the region $\re z \ge 0$, $z \ne i \xi$, $z \ne -i \eta$. By Propositions~\ref{prop:zeta:holder} and~\ref{prop:lambda:holder}, $\zeta(r)$ and $\lambda^{-1}(s)$ are locally Hölder continuous on $(0, \infty)$ and $(0, \lambda(\infty))$, respectively, and hence also $\psi(\lambda^{-1}(s)) = \tilde{\psi}(\zeta(\lambda^{-1}(s)))$ is Hölder continuous. Therefore, the Hilbert transform of $\psi(\lambda^{-1}(s))$ is well-defined and continuous on $(0, \lambda(\infty))$, and by the Sokhotski--Plemelj's formula we have (see, e.g., Equation~(4.8) in~\cite{gakhov})
\formula{
 \lim_{\tau \to 0^+} \frac{1}{\pi} \int_0^{\lambda(\infty)} \psi(\lambda^{-1}(s)) \, \frac{ds}{-\sigma + i \tau + s} & = \frac{1}{\pi} \pvint_0^{\lambda(\infty)} \psi(\lambda^{-1}(s)) \, \frac{ds}{s - \sigma} - i \psi(\lambda^{-1}(\sigma))
}
locally uniformly with respect to $\sigma \in (0, \lambda(\infty))$. It follows that
\formula{
 \lim_{\tau \to 0^+} \Psi(-\sigma + i \tau) & = -\frac{1}{\sigma (\xi + \eta)} \, \exp\biggl(i \psi(\lambda^{-1}(\sigma)) - \frac{1}{\pi} \pvint_0^{\lambda(\infty)} \psi(\lambda^{-1}(s)) \, \frac{ds}{s - \sigma} \biggr)
}
locally uniformly with respect to $\sigma \in (0, \lambda(\infty))$; here the principal value integral $\pvint$ is defined as the limit of integrals over the complement of $(\sigma - \eps, \sigma + \eps)$ as $\eps \to 0^+$. Consequently, $\mu(ds)$ is absolutely continuous on $(0, \lambda(\infty))$, and by Proposition~\ref{prop:pt:laplace},
\formula{
\begin{aligned}
 & \int_0^\infty \int_{(0, \infty)} e^{-\eta x - \xi y} p_t^+(x, dy) dx \\
 & \qquad = \frac{1}{\pi (\xi + \eta)} \int_0^{\lambda(\infty)} \frac{e^{-\sigma t}}{\sigma} \, \exp\biggl(-\frac{1}{\pi} \pvint_0^{\lambda(\infty)} \psi(\lambda^{-1}(s)) \, \frac{ds}{s - \sigma} \biggr) \sin(\psi(\lambda^{-1}(\sigma))) d\sigma \\
 & \qquad = \frac{1}{\pi (\xi + \eta)} \int_0^\infty \frac{e^{-t \lambda(r)}}{\lambda(r)} \, \exp\biggl(-\frac{1}{\pi} \pvint_0^\infty \psi(q) \, \frac{d\lambda(q)}{\lambda(q) - \lambda(r)} \biggr) \sin(\psi(r)) d\lambda(r) ;
\end{aligned}
}
here, on the right-hand side, for simplicity, by $\pvint$ we denote the limit of integrals over the complement of $(\lambda(\lambda^{-1}(r) - \eps), \lambda(\lambda^{-1}(r) + \eps))$ as $\eps \to 0^+$. Observe that if $r \notin Z_f$, then $\zeta(r) \in i \R$, so that $\psi(r) \in \{0, \pi\}$ and $\sin(\psi(r)) = 0$. Thus, we conclude that
\formula[eq:pt:pv]{
\begin{aligned}
 & \int_0^\infty \int_{(0, \infty)} e^{-\eta x - \xi y} p_t^+(x, dy) dx \\
 & \quad = \frac{1}{\pi (\xi + \eta)} \int_{Z_f} \frac{e^{-t \lambda(r)}}{\lambda(r)} \, \exp\biggl(-\frac{1}{\pi} \pvint_0^\infty \psi(q) \, \frac{d\lambda(q)}{\lambda(q) - \lambda(r)} \biggr) \sin(\psi(r)) d\lambda(r)
\end{aligned}
}
for every $t > 0$.

Below we simplify the above expression. Recall that
\formula{
 \psi(r) & = \Arg(\zeta(r) + i \eta) - \Arg(\zeta(r) - i \xi) .
}
Thus,
\formula*[eq:pt:sin]{
 \sin(\psi(r)) & = \frac{\eta + \im \zeta(r)}{|\zeta(r) + i \eta|} \, \frac{\re \zeta(r)}{|\zeta(r) - i \xi|} - \frac{\re \zeta(r)}{|\zeta(r) + i \eta|} \, \frac{-\xi + \im \zeta(r)}{|\zeta(r) - i \xi|} \\
 & = \frac{(\xi + \eta) \re \zeta(r)}{|\zeta(r) + i \eta| |\zeta(r) - i \xi|} \, .
}
We now transform the inner integral in~\eqref{eq:pt:pv}. Fix $r \in Z_f$, and denote
\formula{
 I & = \exp \biggl( -\frac{1}{\pi} \pvint_0^\infty \psi(q) \, \frac{d\lambda(q)}{\lambda(q) - \lambda(r)} \biggr) .
}
As in the proof of Theorem~5.7 in~\cite{kwasnicki:rogers}, we find that $\psi(q) \le C(\xi, \eta) / q$ (see~(5.14) in~\cite{kwasnicki:rogers} and the comments at the end of the proof of Theorem~5.7 therein), and that $\log |\lambda(q) - \lambda(r)| \le C(f, r) q^{1/2}$ for $q > 2 r$ (see~(5.16) in~\cite{kwasnicki:rogers}). It follows that $\psi(q) \log |\lambda(q) - \lambda(r)|$ converges to zero as $q \to \infty$. Furthermore, $\psi(q) \log |\lambda(q) - \lambda(r)|$ converges to $\pi \log \lambda(r)$ as $q \to 0^+$. Therefore, integrating by parts, we find that
\formula{
 \pvint_0^\infty \psi(q) \, \frac{d\lambda(q)}{\lambda(q) - \lambda(r)} & = -\pi \log \lambda(r) - \int_0^\infty \psi'(q) \log |\lambda(q) - \lambda(r)| dq .
}
Since
\formula{
 \psi'(q) & = \im \biggl(\frac{\zeta'(q)}{\zeta(q) + i \eta} - \frac{\zeta'(q)}{\zeta(q) - i \xi}\biggr)
}
(see~(5.15) in~\cite{kwasnicki:rogers}), we obtain
\formula{
 I & = \lambda(r) \exp \biggl( \frac{1}{\pi} \int_0^\infty \im \biggl(\frac{\zeta'(q)}{\zeta(q) + i \eta} - \frac{\zeta'(q)}{\zeta(q) - i \xi}\biggr) \log |\lambda(q) - \lambda(r)| dq \biggr) .
}
As in the proof of Theorem~5.7 in~\cite{kwasnicki:rogers}, for almost all $q \in (0, \infty) \setminus Z_f$ we have $\zeta(q), \zeta'(q) \in i \R$, and hence the integrand on the right-hand side of the above equation is equal to zero. Thus,
\formula{
 I & = \lambda(r) \exp \biggl( \frac{1}{\pi} \int_{Z_f} \im \biggl(\frac{\zeta'(q)}{\zeta(q) + i \eta} - \frac{\zeta'(q)}{\zeta(q) - i \xi}\biggr) \log |\lambda(q) - \lambda(r)| dq \biggr) .
}
Recall that if we define $\zeta(0) = 0$ and $\zeta(-r) = -\overline{\zeta(r)}$, then $\zeta(r)$ for $r \in (-Z_f) \cup Z_f$ is a parameterisation of $\Gamma^\star$. Furthermore, as in the proof of Corollary~5.6 in~\cite{kwasnicki:rogers}, we have
\formula{
 \hspace*{4em} & \hspace*{-4em} \im \biggl(\frac{\zeta'(q)}{\zeta(q) + i \eta} - \frac{\zeta'(q)}{\zeta(q) - i \xi}\biggr) \\
 & = \frac{1}{2 i} \biggl(\frac{\zeta'(q)}{\zeta(q) + i \eta} - \frac{\zeta'(q)}{\zeta(q) - i \xi}\biggr) - \frac{1}{2 i} \biggl(\frac{\overline{\zeta'(q)}}{\overline{\zeta(q)} - i \eta} - \frac{\overline{\zeta'(q)}}{\overline{\zeta(q)} + i \xi}\biggr) \\
 & = \frac{1}{2 i} \biggl(\frac{\zeta'(q)}{\zeta(q) + i \eta} - \frac{\zeta'(q)}{\zeta(q) - i \xi}\biggr) + \frac{1}{2 i} \biggl(\frac{\zeta'(-q)}{\zeta(-q) + i \eta} - \frac{\zeta'(-q)}{\zeta(q) - i \xi}\biggr) .
}
Therefore, we find that
\formula{
 I & = \lambda(r) \exp \biggl( \frac{1}{2 \pi i} \int_{\Gamma_f^\star} \biggl(\frac{1}{z + i \eta} - \frac{1}{z - i \xi}\biggr) \log |f(z) - \lambda(r)| dz \biggr) \\
 & = \lambda(r) \exp \biggl( -\frac{1}{2 \pi i} \int_{\Gamma_f^\star} \biggl(\frac{1}{z - i \xi} - \frac{1}{z + i \eta}\biggr) \log |f(z) - \lambda(r)| dz \biggr) .
}
Recall that $i \xi \in D_f^+$ and $-i \eta \in D_f^-$. From Proposition~\ref{prop:r:quot:bd} with $\xi_1 = i \xi$ and $\xi_2 = -i \eta$, it follows that
\formula[eq:pt:exp]{
 I & = \frac{\lambda(r) f^+(r; \xi) f^-(r; \eta)}{|i \xi - \zeta(r)| |{-i \eta} - \zeta(r)|} \, .
}
By combining~\eqref{eq:pt:pv}, \eqref{eq:pt:sin} and~\eqref{eq:pt:exp}, we arrive at
\formula{
 & \int_0^\infty \int_{(0, \infty)} e^{-\eta x - \xi y} p_t^+(x, dy) dx \\
 & \qquad = \frac{1}{\pi (\xi + \eta)} \int_{Z_f} \frac{e^{-t \lambda(r)}}{\lambda(r)} \, \frac{\lambda(r) f^+(r; \xi) f^-(r; \eta)}{|\zeta(r) - i \xi| |\zeta(r) + i \eta|} \, \frac{(\xi + \eta) \re \zeta(r)}{|\zeta(r) + i \eta| |\zeta(r) - i \xi|} \, d\lambda(r) \\
 & \qquad = \frac{1}{\pi} \int_{Z_f} e^{-t \lambda(r)} \, \frac{f^+(r; \xi) f^-(r; \eta) \re \zeta(r)}{|\zeta(r) - i \xi|^2 |\zeta(r) + i \eta|^2} \, d\lambda(r) ,
}
as desired.

The proof in the remaining two cases: when $i \xi \in D_f^+$ and $-i \eta \in D_f^+$, or when $i \xi \in D_f^-$ and $-i \eta \in D_f^-$, is very similar, with two modifications. First, the function $\psi(r)$ has a jump at $r = \eta$ or at $r = \xi$. This makes the treatment of the Hilbert transform near $s = \lambda(\eta)$ or $s = \lambda(\xi)$ somewhat more complicated, but otherwise this does not affect the argument. Second, the application of Proposition~\ref{prop:r:quot:bd} for~\eqref{eq:pt:exp} is slightly different, and thus the final formula takes a different form. We omit the details.
\end{proof}

\subsection{Generalised eigenfunctions}

Recall that by Proposition~\ref{prop:pt:quot}, under suitable assumptions on $\xi, \eta > 0$,
\formula*[eq:pt:quot]{
 & \int_0^\infty \int_{(0, \infty)} e^{-\eta x - \xi y} p_t^+(x, dy) dx \\
 & \qquad = \frac{1}{\pi} \int_{Z_f} e^{-t \lambda(r)} \, \frac{f^+(r; \xi)}{|\zeta(r) - i \xi|^2} \, \frac{f^-(r; \eta)}{|\zeta(r) + i \eta|^2} \, \re \zeta(r) d\lambda(r) ,
}
In order to invert the double Laplace transform in the above identity, we first identify the factors
\formula{
 & \frac{f^+(r; \xi)}{|\zeta(r) - i \xi|^2} && \text{and} && \frac{f^-(r; \eta)}{|\zeta(r) + i \eta|^2} \, ,
}
up to appropriate normalisation factors, with Laplace transform of what we call \emph{generalised eigenfunctions}. As before, we assume that $f(z)$ is a non-constant, non-degenerate Rogers function associated to a Lévy process $X_t$ with completely monotone jumps. Recall that for $r \in Z_f$, $f(r; z)$ is the Rogers function defined as in Proposition~\ref{prop:r:quot}, and $f^+(r; z)$ and $f^-(r; z)$ are the corresponding Wiener--Hopf factors.

Since $f^+(r; \xi)$ is a complete Bernstein function of $\xi$, we have, by Proposition~\ref{prop:cbf:quot} applied to $h(\xi) = f^+(r; \xi)$ and $\zeta = \zeta(r)$,
\formula*[eq:pt:lfp:aux]{
 \frac{f^+(r; \xi)}{|\zeta(r) - i \xi|^2} & = \frac{f^+(r; \xi)}{(\xi + i \zeta(r))(\xi - i \overline{\zeta(r)})} \\
 & = \frac{1}{2 i \re \zeta(r)} \, \biggl( \frac{f^+(r; i \overline{\zeta(r)})}{\xi - i \overline{\zeta(r)}} - \frac{f^+(r; -i \zeta(r))}{\xi + i \zeta(r)} \biggr) - \tilde{g}(\xi)
}
for $\xi \in \C \setminus (-\infty, 0]$, where $\tilde{g}(\xi)$ is an appropriate Stieltjes function (depending on $f$ and $r$). Observe that $f^+(r; i \overline{\zeta(r)}) = \overline{f^+(r; -i \zeta(r))}$ and $\xi - i \overline{\zeta(r)} = \overline{\xi + i \zeta(r)}$ when $\xi > 0$. Therefore, the expression in brackets on the right-hand side of~\eqref{eq:pt:lfp:aux} is equal to $-2 i \im(f^+(r; -i \zeta(r)) / (\xi + i \zeta(r)))$. It follows that for $\xi > 0$ we have
\formula[eq:pt:lfp]{
 \frac{\re \zeta(r)}{|f^+(r; -i \zeta(r))|} \, \frac{f^+(r; \xi)}{|\zeta(r) - i \xi|^2} & = -\frac{1}{|f^+(r; -i \zeta(r))|} \, \im \frac{f^+(r; -i \zeta(r))}{\xi + i \zeta(r)} - g_+(r; \xi) ,
}
where $g_+(r; \xi) = \tilde{g}(\xi) \re \zeta(r) / |f^+(r; -i \zeta(r))|$ is a Stieltjes function of $\xi$. Note that for a given $\zeta \in \C$, the function $1 / (\xi + i \zeta)$ (defined in the region $\re \xi > \im \zeta$) is the Laplace transform of $e^{i \zeta x}$. Therefore, if $\xi > \max\{0, \im \zeta(r)\}$, then the expression given in~\eqref{eq:pt:lfp} is the Laplace transform of the function
\formula{
 F_+(r; y) & = -\frac{1}{|f^+(r; -i \zeta(r))|} \, \im \bigl( f^+(r; -i \zeta(r)) e^{-i \zeta(r) y} \bigr) - G_+(r; y) \\
 & = e^{b y} \sin(a y + c_+) - G_+(r; y) ,
}
where $G_+(r; y)$ is an appropriate completely monotone function, $\zeta(r) = a + b i$, $c_+ = -\Arg f^+(r; -i \zeta(r))$, and $a > 0$, $b \in \R$ and $c_+ \in (-\pi, \pi)$. Similarly, for $\eta > 0$ we have
\formula{
 \frac{\re \zeta(r)}{|f^-(r; i \zeta(r))|} \, \frac{f^-(r; \eta)}{|\zeta(r) + i \eta|^2} & = \frac{\re \zeta(r)}{|f^-(r; i \zeta(r))|} \, \frac{f^-(r; \eta)}{(\eta - i \zeta(r))(\eta + i \overline{\zeta(r)})} \\
 & = \frac{1}{|f^-(r; i \zeta(r))|} \, \im \frac{f^-(r; i \zeta(r))}{\eta - i \zeta(r)} - g_-(r; \eta) ,
}
and, given that $\eta > \max\{0, -\im \zeta(r)\}$, the above expression is the Laplace transform of the function
\formula{
 F_-(r; x) & = \frac{1}{|f^-(r; i \zeta(r))|} \, \im \bigl( f^-(r; i \zeta(r)) e^{i \zeta(r) x} \bigr) - G_-(r; x) \\
 & = e^{-b x} \sin(a x + c_-) - G_-(r; x) ;
}
here $g_-(r; \eta)$ is an appropriate Stieltjes function, $G_-(r; x)$ is an appropriate completely monotone function, $\zeta(r) = a + b i$ as above, and $c_- = \Arg f^-(r; i \zeta(r))$.

We abuse the notation and write $\laplace F_+(r; \xi)$ for the holomorphic extension of the Laplace transform of $F_+(r; y)$ (as a function of $y$) from $\{\xi \in \C : \re \xi > \max\{0, \im \zeta(r)\}\}$ to $\C \setminus ((-\infty, 0] \cup \{-i \zeta(r), i \overline{\zeta(r)}\})$. Likewise, $\laplace F_-(r; \eta)$ denotes a similar extension of the Laplace transform of $F_-(r; x)$. Recall that
\formula{
 f^+(r; -i \zeta(r)) f^-(r; i \zeta(r)) & = f(r; \zeta(r)) = \frac{2 \re \zeta(r)}{f'(\zeta(r))} \, ,
}
and observe that on $Z_f$ we have
\formula{
 d \lambda(r) & = f'(\zeta(r)) \zeta'(r) dr = |f'(\zeta(r))| \, \lvert\zeta'(r)\rvert dr .
}
Thus, we conclude that~\eqref{eq:pt:quot} is equivalent to
\formula{
 & \int_0^\infty \int_{(0, \infty)} e^{-\eta x - \xi y} p_t^+(x, dy) dx \\
 & \qquad = \frac{1}{\pi} \int_{Z_f}  e^{-t \lambda(r)} \biggl(\frac{\re \zeta(r)}{|f^+(r; -i \zeta(r))|} \, \frac{f^+(r; \xi)}{|\zeta(r) - i \xi|^2}\biggr) \biggl(\frac{\re \zeta(r)}{|f^-(r; i \zeta(r))|} \, \frac{f^-(r; \eta)}{|\zeta(r) + i \eta|^2} \biggr) \times \\
 & \qquad \hspace*{10em} \times \frac{|f^+(r; -i \zeta(r)) f^-(r; i \zeta(r))|}{\re \zeta(r)} \, d\lambda(r) \\
 & \qquad = \frac{2}{\pi} \int_{Z_f} e^{-t \lambda(r)} \laplace F_+(r; \xi) \laplace F_-(r; \eta) \lvert\zeta'(r)\rvert dr .
}
We have thus essentially proved the main result of this section, Proposition~\ref{prop:pt:laplace}. Before we state it, let us introduce the following formal definition of the generalised eigenfunctions $F_+(r; y)$ and $F_-(r; x)$.

\begin{definition}
\label{def:eig}
Suppose that $f(\xi)$ is a Rogers function, and let $Z_f$, $\zeta_f(r)$ and $\lambda_f(r)$ be defined as in Section~\ref{sec:rogers}. Let $r \in Z_f$, let $f(r; z) = (z - \zeta_f(r)) (z + \overline{\zeta_f(r)}) / (f(z) - \lambda_f(r))$ be the Rogers function defined in Proposition~\ref{prop:r:quot}, and let $f^+(r; \xi)$ and $f^-(r; \eta)$ denote the corresponding Wiener--Hopf factors. Define
\formula{
 a_f(r) & = \re \zeta_f(r) , & c_{f+}(r) & = -\Arg f^+(r; -i \zeta_f(r)) , \\
 b_f(r) & = \im \zeta_f(r) , & c_{f-}(r) & = \Arg f^-(r; i \zeta_f(r)) ,
}
and let
\formula{
 F_{f+}(r; y) & = e^{b_f(r) y} \sin(a_f(r) y + c_{f+}(r)) - G_{f+}(r; y) , \\
 F_{f-}(r; x) & = e^{-b_f(r) x} \sin(a_f(r) x + c_{f-}(r)) - G_{f-}(r; x) ,
}
where $G_{f+}(r; y)$ and $G_{f-}(r; y)$ are completely monotone functions with Laplace transforms
\formula{
 \laplace G_{f+}(r; \xi) & = -\im \frac{e^{-i c_{f+}(r)}}{\xi + i \zeta_f(r)} - \frac{\re \zeta_f(r)}{|f^+(r; -i \zeta_f(r))|} \, \frac{f^+(r; \xi)}{|\zeta_f(r) - i \xi|^2} \, , \\
 \laplace G_{f-}(r; \eta) & = \im \frac{e^{i c_{f-}(r)}}{\eta - i \zeta_f(r)} - \frac{\re \zeta_f(r)}{|f^-(r; i \zeta_f(r))|} \, \frac{f^-(r; \eta)}{|\zeta_f(r) + i \eta|^2}
}
for $\xi, \eta > 0$.
\end{definition}

The notation introduced above is kept until the end of the article. We stress that
\formula{
 \laplace F_{f+}(r; \xi) & = \frac{\re \zeta_f(r)}{|f^+(r; -i \zeta_f(r))|} \, \frac{f^+(r; \xi)}{(\xi + i \zeta_f(r)) (\xi - i \overline{\zeta_f(r)})} \, , \\
 \laplace F_{f-}(r; \eta) & = \frac{\re \zeta_f(r)}{|f^-(r; i \zeta_f(r))|} \, \frac{f^-(r; \eta)}{(\eta - i \zeta_f(r)) (\eta + i \overline{\zeta_f(r)})}
}
whenever $\re \xi > \max\{0, \im \zeta_f(r)\}$ and $\re \eta > \max\{0, -\im \zeta_f(r)\}$, and the above expressions extend to meromorphic functions of $\xi, \eta \in \C \setminus (-\infty, 0]$, with two simple poles at $\xi = -i \zeta_f(r)$ and $\xi = i \overline{\zeta_f(r)}$, and at $\eta = i \zeta_f(r)$ and $\eta = -i \overline{\zeta_f(r)}$, respectively. Since $f^+(r; \xi)$ and $f^-(r; \eta)$ are complete Bernstein functions, by~\eqref{eq:cbf:arg} we have
\formula{
 0 \le c_{f+}(r) & \le \tfrac{\pi}{2} - \Arg \zeta_f(r) , & 0 \le c_{f-}(r) & \le \tfrac{\pi}{2} + \Arg \zeta_f(r) .
}
Furthermore, we observe that
\formula{
 c_{f+}(r) - c_{f-}(r) & = -\Arg f^+(r; -i \zeta_f(r)) - \Arg f^-(r; i \zeta_f(r)) \\
 & = -\Arg \bigl( f^+(r; -i \zeta_f(r)) f^-(r; i \zeta_f(r)) \bigr) = -\Arg f(r; \zeta_f(r)) .
}
Since $f(r; \zeta_f(r)) = 2 \re \zeta_f(r) / f'(\zeta_f(r))$, $\re \zeta_f(r) > 0$, $f'(\zeta_f(r)) = \lambda_f'(r) / \zeta_f'(r)$ and $\lambda_f'(r) > 0$, we find that
\formula{
 c_{f+}(r) - c_{f-}(r) & = \Arg f'(\zeta_f(r)) = -\Arg \zeta_f'(r) .
}
The following result is a restatement of Theorem~\ref{thm:heat:laplace}.

\begin{proposition}
\label{prop:pt:lap}
Suppose that $X_t$ is a Lévy process with completely monotone jumps, and $f(\xi)$ is the corresponding Rogers function. Let $t \ge 0$, $\re \xi, \re \eta > 0$, and suppose that $i \xi \in D_f^+$ and $-i \eta \in D_f^-$. Then
\formula*[eq:pt:lap]{
 & \int_0^\infty \int_{(0, \infty)} e^{-\eta x - \xi y} p_t^+(x, dy) dx \\
 & \qquad = \frac{2}{\pi} \int_{Z_f} e^{-t \lambda_f(r)} \laplace F_{f+}(r; \xi) \laplace F_{f-}(r; \eta) \lvert\zeta_f'(r)\rvert dr .
}
\end{proposition}

\begin{proof}
We have already proved that formula~\eqref{eq:pt:lap} holds for $\xi, \eta > 0$ such that $i \xi \in D_f^+$ and $-i \eta \in D_f^-$. The bivariate Laplace transform on the left-hand side of~\eqref{eq:pt:lap} clearly defines a holomorphic function of $\xi$ and $\eta$ in the region $\re \xi, \re \eta > 0$. It is therefore sufficient to prove that the right-hand side of~\eqref{eq:pt:lap} defines a holomorphic function of $\xi$ and $\eta$ in the region $\re \xi, \re \eta > 0$, $i \xi \in D_f^+$, $-i \eta \in D_f^-$. This follows from Morera's theorem and Fubini's theorem. Indeed: for every $\eta$ such that $\re \eta > 0$ and $-i \eta \in D_f^-$, and for every closed contour $\tilde{\Gamma}$ contained in the region $\re \xi > 0$, $i \xi \in D_f^+$, we will momentarily show that we may apply Fubini's theorem to find that
\formula*[eq:pt:lap:aux]{
 & \int_{\tilde{\Gamma}} \biggl(\frac{2}{\pi} \int_0^\infty e^{-t \lambda(r)} \laplace F_{f+}(r; \xi) \laplace F_-(r; \eta) \lvert\zeta'(r)\rvert dr \biggr) d\xi \\
 & \qquad = \frac{2}{\pi} \int_0^\infty e^{-t \lambda(r)} \biggl(\int_{\tilde{\Gamma}} \laplace F_{f+}(r; \xi) d\xi\biggr) \laplace F_-(r; \eta) \lvert\zeta'(r)\rvert dr = 0 .
}
By Morera's theorem, this implies that the function defined by the right-hand side of~\eqref{eq:pt:lap} is holomorphic with respect to $\xi$. A similar argument proves that it is holomorphic with respect to $\eta$, and the desired result follows by Hartog's theorem.

Therefore, it remains to find a sufficiently good estimate of the integrand on the right-hand side of~\eqref{eq:pt:lap}, which will allow us to apply Fubini's theorem in~\eqref{eq:pt:lap:aux}. Fix any $\xi_0 > 0$ such that $i \xi_0 \in D_f^+$. Since $i \tilde{\Gamma}$ is a compact subset of $D_f^+$, the distance $\delta$ between $i \tilde{\Gamma}$ and $D_f^+$ is positive, and $\tilde{\Gamma}$ lies in a disk $\{\xi \in \C : |\xi| < R\}$ for some $R > 0$. Hence, for all $\xi \in \tilde{\Gamma}$ and $r \in Z_f$ we have
\formula{
 |\xi + i \zeta(r)| |\xi - i \overline{\zeta(r)}| & \ge (\max\{\delta, r - R\})^2 \\
 & \ge C(\delta, R, \xi_0) |\xi_0 + i \zeta(r)| |\xi_0 - i \overline{\zeta(r)}| \\
 & = C(\delta, R, \xi_0) (\xi_0 + i \zeta(r)) (\xi_0 - i \overline{\zeta(r)}) .
}
On the other hand, by Corollary~\ref{cor:cbf:bound} applied to the complete Bernstein function $f^+(r; \xi)$, for all $\xi \in \tilde{\Gamma}$ we have
\formula{
 |f^+(r; \xi)| & \le C(\tilde{\Gamma}, \xi_0) f^+(r; \xi_0) .
}
Therefore,
\formula{
 |\laplace F_+(r; \xi)| & = \frac{\re \zeta(r)}{|f^+(r; -i \zeta(r))|} \, \frac{|f^+(r; \xi)|}{|\xi + i \zeta(r)| |\xi - i \overline{\zeta(r)}|} \\
 & \le C(f, \tilde{\Gamma}, \xi_0) \, \frac{\re \zeta(r)}{|f^+(r; -i \zeta(r))|} \, \frac{f^+(r; \xi_0)}{(\xi_0 + i \zeta(r)) (\xi_0 - i \overline{\zeta(r)})} \\
 & = C(f, \tilde{\Gamma}, \xi_0) \laplace F_+(r; \xi_0) .
}
Similarly, with any fixed $\eta_0 > 0$ such that $-i \eta_0 \in D_f^-$,
\formula{
 |\laplace F_-(r; \eta)| & \le C(f, \eta_0, \eta) \laplace F_-(r; \eta_0) .
}
It follows that
\formula{
 & \int_{\tilde{\Gamma}} \int_0^\infty e^{-t \lambda(r)} |\laplace F_{f+}(r; \xi)| |\laplace F_-(r; \eta)| \lvert\zeta'(r)\rvert dr |d\xi| \\
 & \qquad \le C(f, \tilde{\Gamma}, \xi_0, \eta_0, \eta) \int_0^\infty e^{-t \lambda(r)} \laplace F_{f+}(r; \xi_0) \laplace F_-(r; \eta_0) \lvert\zeta'(r)\rvert dr < \infty ,
}
as desired.
\end{proof}

\subsection{Properties of eigenfunctions}

Proposition~\ref{prop:pt:lap}, which is the main result of this section, requires no assumptions on the (non-constant) Rogers function $f(\xi)$. More specialised results are given in the next section. Before we move on, however, we prove two estimates of the generalised eigenfunctions $F_{f+}(r; y)$ and $F_{f-}(r; x)$, defined in Definition~\ref{def:eig}. Recall that $\Gamma_f$, $Z_f$, $\zeta_f(r)$ and $\lambda_f(r)$ were introduced in Section~\ref{sec:rogers} (see, in particular, Section~\ref{sec:notation}).

\begin{lemma}
\label{lem:eig:lbound}
Suppose that $X_t$ is a Lévy process with completely monotone jumps, and $f(\xi)$ is the corresponding Rogers function. Let $\dist(z, \Gamma_f)$ denote the distance between $z$ and $\Gamma_f$, let $r \in Z_f$, let $\thet_f(r) = \Arg \zeta_f(r)$, and let $r_0 = \inf Z_f$. Then,
\formula{
 \lvert\laplace F_{f+}(r; \xi)\rvert & \le \frac{36 (r_0 + |\xi|)^2}{(\dist(i \xi, \Gamma_f))^2} \, \frac{1}{r + |\xi|} \, \frac{\cos \thet_f(r)}{\cos(\tfrac{1}{2} \thet_f(r) - \tfrac{\pi}{4}) \cos(\tfrac{1}{2} \Arg \xi)} \, , \\
 \lvert\laplace F_{f-}(r; \eta)\rvert & \le \frac{36 (r_0 + |\eta|)^2}{(\dist(-i \eta, \Gamma_f))^2} \, \frac{1}{r + |\eta|} \, \frac{\cos \thet_f(r)}{\cos(\tfrac{1}{2} \thet_f(r) + \tfrac{\pi}{4}) \cos(\tfrac{1}{2} \Arg \eta)} \, ,
}
for $\xi, \eta \in \C \setminus (-\infty, 0]$. In particular, if $\xi, \eta > 0$, then
\formula{
 0 \le \laplace F_{f+}(r; \xi) & \le \frac{72 (\xi + r_0)^2}{(\dist(i \xi, \Gamma_f))^2} \, \frac{1}{\xi + r} \, , \\
 0 \le \laplace F_{f-}(r; \eta) & \le \frac{72 (\eta + r_0)^2}{(\dist(-i \eta, \Gamma_f))^2} \, \frac{1}{\eta + r} \, .
}
\end{lemma}

\begin{proof}
Denote $d = \dist(i \xi, \Gamma)$, $\zeta = \zeta(r)$ and $\thet = \thet(r) = \Arg \zeta$. Clearly, $\Arg \sqrt{-i \zeta} = \tfrac{\thet}{2} - \tfrac{\pi}{4}$. Recall that, by definition,
\formula{
 \laplace F_{f+}(r; \xi) & = \frac{\re \zeta}{|\xi + i \zeta|^2} \, \frac{f^+(r; \xi)}{|f^+(r; -i \zeta)|} \, .
}
Clearly, $\laplace F_{f+}(r; \xi) \ge 0$ when $\xi > 0$. For a general $\xi \in \C \setminus (-\infty, 0]$, by Corollary~\ref{cor:cbf:bound} applied to the complete Bernstein function $f^+(r; \xi)$, we have
\formula{
 \frac{|f^+(r; -i \zeta)|}{f^+(r; r)} & \ge \frac{1}{2 \sqrt{2}} \, \cos(\tfrac{\thet}{2} - \tfrac{\pi}{4})
}
and
\formula{
 \frac{|f^+(r; \xi)|}{f^+(r; r)} & \le \sqrt{2} \, \frac{r + |\xi|}{r} \, \frac{1}{\cos(\tfrac{1}{2} \Arg \xi)} \, .
}
Thus,
\formula{
 \lvert\laplace F_{f+}(r; \xi)\rvert & \le 4 \, \frac{r \cos \thet}{|\xi + i \zeta|^2} \, \frac{r + |\xi|}{r} \, \frac{1}{\cos(\tfrac{1}{2} \Arg \xi) \cos(\tfrac{\thet}{2} - \tfrac{\pi}{4})}
}
By definition, $|\xi + i \zeta| \ge d$, and additionally $|\xi + i \zeta| \ge r - |\xi| \ge \tfrac{1}{3} (r + |\xi|)$ when $r > 2 |\xi|$. Since $d \le r_0 + |\xi|$, we obtain $|\xi + i \zeta| \ge \tfrac{1}{3} (d / (r_0 + |\xi|)) (r + |\xi|)$, and hence
\formula{
 \lvert\laplace F_{f+}(r; \xi)\rvert & \le 36 \, \frac{r (r_0 + |\xi|)^2 \cos \thet}{d^2 (r + |\xi|)^2} \, \frac{r + |\xi|}{r} \, \frac{1}{\cos(\tfrac{1}{2} \Arg \xi) \cos(\tfrac{\thet}{2} - \tfrac{\pi}{4})} \, .
}
This is the desired bound for $\lvert\laplace F_{f+}(r; \xi)\rvert$. When $\xi > 0$, then $\cos(\tfrac{1}{2} \Arg \xi) = 1$, and the desired estimate follows by the trigonometric identity~\eqref{eq:coscos}.

The results for $\laplace F_{f-}(r; \eta)$ are proved in a similar way.
\end{proof}

\begin{lemma}
\label{lem:eig:est}
Suppose that $X_t$ is a Lévy process with completely monotone jumps, and $f(\xi)$ is the corresponding Rogers function. For every $r \in Z_f$, we have
\formula{
 |F_{f+}(r; y)| & \le 22 r y \biggl(1 + \frac{\re \zeta_f(r)}{r} \, \frac{f^+(r; 1 / y)}{|f^+(r; -i \zeta_f(r))|} \biggr) , \\
 |F_{f-}(r; x)| & \le 22 r x \biggl(1 + \frac{\re \zeta_f(r)}{r} \, \frac{f^-(r; 1 / x)}{|f^-(r; i \zeta_f(r))|} \biggr) .
}
when $0 < x, y < 1 / (2 r)$, and
\formula{
 |F_{f+}(r; y)| & \le 2 e^{y \max\{\im \zeta_f(r), 0\}} , & |F_{f-}(r; x)| & \le 2 e^{x \max\{-\im \zeta_f(r), 0\}}
}
for general $x, y > 0$.
\end{lemma}

\begin{proof}
Again, we consider only the estimate of $F_+(r; y)$. With $\zeta(r) = a(r) + i b(r)$, we have
\formula*[eq:eig:est:1]{
 F_+(r; y) & = e^{b(r) y} \sin(a(r) y + c_+(r)) - G_+(r; y) \\
 & = \im \bigl( e^{-i c_+(r)} (1 - e^{-i \zeta(r) y}) \bigr) + (\sin c_+(r) - G_+(r; y)) .
}
Using $|\zeta(r)| = r$ and the mean value property, we find that
\formula[eq:eig:est:2]{
 |1 - e^{-i \zeta(r) y}| & \le \lvert-i \zeta(r) y\rvert \max\{1, e^{b(r) y}\} \le r y (1 + e^{b(r) y}) \le r y (1 + e^{r y}) .
}
Furthermore, for $\xi > 0$ we have
\formula{
 & \int_0^\infty e^{-\xi y} (\sin c_+(r) - G_+(r; y)) dy = \frac{\sin c_+(r)}{\xi} - \laplace G_+(r; \xi) \\
 & \hspace*{5em} = \frac{\sin c_+(r)}{\xi} + \im \frac{e^{-i c_+(r)}}{\xi + i \zeta(r)} + \frac{\re \zeta(r)}{|f^+(r; -i \zeta(r))|} \, \frac{f^+(r; \xi)}{|\zeta(r) - i \xi|^2} \, .
}
Since $G_+(r; y)$ is completely monotone, it is a decreasing function of $y > 0$, and we have
\formula{
 G_+(r; 0^+) & = \lim_{\xi \to \infty} \int_0^\infty \xi e^{-\xi y} G_+(r; y) dy \\
 & = \lim_{\xi \to \infty} \biggl( -\xi \im \frac{e^{-i c_+(r)}}{\xi + i \zeta(r)} - \frac{\re \zeta(r)}{|f^+(r; -i \zeta(r))|} \, \frac{\xi f^+(r; \xi)}{|\zeta(r) - i \xi|^2} \biggr) \\
 & = \sin c_+(r) - \frac{\re \zeta(r)}{|f^+(r; -i \zeta(r))|} \, \lim_{\xi \to \infty} \frac{f^+(r; \xi)}{\xi} \le \sin c_+(r) ;
}
the last inequality follows from $f^+(r; \xi) \ge 0$. Thus, $\sin c_+(r) - G_+(r; y) \ge 0$. On the other hand, since $\sin c_+(r) - G_+(r; y)$ is an increasing function of $y > 0$, we have
\formula{
 \int_0^\infty e^{-\xi y} (\sin c_+(r) - G_+(r; y)) dy & \ge \int_p^\infty e^{-\xi y} (\sin c_+(r) - G_+(r; p)) dy \\
 & = \frac{e^{-\xi p} (\sin c_+(r) - G_+(r; p))}{\xi}
}
when $p, \xi > 0$. It follows that
\formula{
 \sin c_+(r) - G_+(r; p) & \le \xi e^{\xi p} \int_0^\infty e^{-\xi y} (\sin c_+(r) - G_+(r; y)) dy \\
 & = \xi e^{\xi p} \biggl(\frac{\sin c_+(r)}{\xi} + \im \frac{e^{-i c_+(r)}}{\xi + i \zeta(r)} + \frac{\re \zeta(r)}{|f^+(r; -i \zeta(r))|} \, \frac{f^+(r; \xi)}{|\zeta(r) - i \xi|^2} \biggr) \\
 & = e^{\xi p} \biggl(\im \frac{-i \zeta(r) e^{-i c_+(r)}}{\xi + i \zeta(r)} + \frac{\re \zeta(r)}{|f^+(r; -i \zeta(r))|} \, \frac{\xi f^+(r; \xi)}{|\zeta(r) - i \xi|^2} \biggr) .
}
When $\xi > r$, then we find that
\formula{
 \sin c_+(r) - G_+(r; p) & \le e^{\xi p} \biggl(\frac{r}{\xi - r} + \frac{\re \zeta(r)}{|f^+(r; -i \zeta(r))|} \, \frac{\xi f^+(r; \xi)}{(\xi - r)^2} \biggr) .
}
Let $y \in (0, 1 / (2 r))$, and set $p = y$ and $\xi = 1 / y$. Then the above estimate reads
\formula*[eq:eig:est:3]{
 \sin c_+(r) - G_+(r; y) & \le e \biggl(\frac{r y}{1 - r y} + \frac{\re \zeta(r)}{|f^+(r; -i \zeta(r))|} \, \frac{y f^+(r; 1 / y)}{(1 - r y)^2} \biggr) \\
 & \le 2 e \biggl( r y + \frac{2 \re \zeta(r)}{|f^+(r; -i \zeta(r))|} \, y f^+(r; 1 / y) \biggr) .
}
By~\eqref{eq:eig:est:1}, \eqref{eq:eig:est:2} and~\eqref{eq:eig:est:3}, whenever $r \in Z_f$, $y > 0$ and $r y < \tfrac{1}{2}$, we have
\formula{
 |F_+(r; y)| & \le |1 - e^{-i \zeta(r) y}| + (\sin c_+(r) - G_+(r; y)) \\
 & \le (1 + e^{r y}) r y + 2 e \biggl( r y + \frac{2 \re \zeta(r)}{|f^+(r; -i \zeta(r))|} \, y f^+(r; 1 / y) \biggr) \\
 & \le 22 \biggl( r y + \frac{\re \zeta(r)}{|f^+(r; -i \zeta(r))|} \, y f^+(r; 1 / y) \biggr) ,
}
as desired. On the other hand, if $r y \ge \tfrac{1}{2}$, we have
\formula{
 |F_{f+}(r; y)| & \le e^{b(r) y} + 1 \le 2 e^{\max\{b(r), 0\} y} ,
}
and the proof is complete.
\end{proof}

%
%

\section{Inversion of spatial Laplace transforms}
\label{sec:inf}

In this section we prove the main results of the paper. We continue to assume that $X_t$ is a (non-killed) Lévy process with completely monotone jumps, and $f(\xi)$ is the corresponding Rogers function --- the characteristic exponent of $X_t$.

\subsection{Transition densities}
\label{sec:inf:pt}

We have the following direct corollary of Proposition~\ref{prop:pt:lap}. Since we give a more general result below in Corollary~\ref{cor:pt:log}, we only sketch the proof.

\begin{proposition}
\label{prop:pt:bounded}
Suppose that $X_t$ is a Lévy process with completely monotone jumps, and $f(\xi)$ is the corresponding Rogers function. Suppose, furthermore, that $\im \zeta(r)$ is bounded on $(0, \infty)$, $t > 0$, and that $\exp(-t \lambda_f(r))$ is integrable over $r \in Z_f$. Then $p_t(x, dy)$ is absolutely continuous with respect to $y \in (0, \infty)$ for every $x > 0$, and the continuous version of the density function is given by
\formula{
 p_t^+(x, y) & = \frac{2}{\pi} \int_{Z_f} e^{-t \lambda_f(r)} F_{f+}(r; y) F_{f-}(r; x) \lvert\zeta_f'(r)\rvert dr .
}
\end{proposition}

\begin{proof}
By Proposition~\ref{prop:pt:lap} and Fubini's theorem (we omit the proof of absolute integrability of the corresponding double integral), when $\xi, \eta > \sup \{ \lvert\im \zeta(r)\rvert : r > 0 \}$, we have
\formula{
 & \int_0^\infty \int_{(0, \infty)} e^{-\eta x - \xi y} p_t^+(x, dy) dx \\
 & \qquad = \frac{2}{\pi} \int_0^\infty \int_0^\infty e^{-\eta x - \xi y} \biggl( \int_{Z_f} e^{-t \lambda(r)} F_+(r; y) F_-(r; x) \lvert\zeta'(r)\rvert dr \biggr) dy dx .
}
The desired result follows by uniqueness of the Laplace transform.
\end{proof}

Below we extend the above identity to a more general class of Lévy processes with completely monotone jumps. To this end, we apply the key idea of this section: we replace the kernel of the the Laplace transform ($e^{-\xi y}$ and $e^{-\eta x}$) by functions $u(y)$ and $v(x)$ admitting a suitable holomorphic extension; a typical example is
\formula{
 u(y) & = \exp(-p y \log (1 + y)), & v(x) & = \exp(-q x \log(1 + x))
}
for any $p, q > 0$. However, we still need to impose a technical assumption about the spine of the Rogers function $f(\xi)$. Before we state the main result, we first recall a simple property of the class of admissible functions $u(y)$ and $v(x)$, which is a minor modification of the result from~\cite{kk:stable}.

\begin{lemma}[Lemma~2.14 in~\cite{kk:stable}]
\label{lem:uv:est}
Let $\eps \in (0, \tfrac{\pi}{2})$, and let $u(y)$ and $v(x)$ be holomorphic functions in the region $\{z \in \C : \lvert\Arg z\rvert < \tfrac{\pi}{2} - \eps\}$, which are real-valued on $(0, \infty)$, and which satisfy
\formula{
 |u(y)| & \le C(u, \eps) e^{-C(u, \eps) |y| \log(1 + |y|)} , & |v(x)| & \le C(v, \eps) e^{-C(v, \eps) |x| \log(1 + |x|)}
}
for $x$ and $y$ in the region $\{z \in \C : \lvert\Arg z\rvert < \tfrac{\pi}{2} - \eps\}$. Then the Laplace transforms of $u(y)$ and $v(x)$ are entire functions, and for every $\delta > 0$ and $p \ge 0$ we have
\formula{
 |\laplace u(\xi)| & \le C(u, \eps, \delta, p) \, \frac{1}{1 + |\xi|} \, , & |\laplace v(\eta)| & \le C(v, \eps, \delta, p) \, \frac{1}{1 + |\eta|}
}
when $\lvert\Arg (p + \xi)\rvert \le \pi - \eps - \delta$ and $\lvert\Arg (p + \eta)\rvert \le \pi - \eps - \delta$.
\end{lemma}

For reader's convenience, we sketch the proof.

\begin{proof}
By considering $e^{p y} u(y)$ and $e^{p x} v(x)$ rather than $u(y)$ and $v(x)$, it is sufficient to consider the case $p = 0$. If $\xi, y \in \C$ and $\lvert\Arg y\rvert < \tfrac{\pi}{2} - \eps$, then
\formula{
 |u(y) e^{-\xi y}| & \le C(u, \eps, \xi) \exp(-(1 + |\xi|) |y| + |\xi| |y|) \le C(u, \eps, \xi) \exp(-|y|) .
}
Therefore, by a standard contour deformation argument, whenever $|\alpha| < \tfrac{\pi}{2} - \eps$, we have
\formula{
 \laplace u(\xi) & = \int_0^\infty u(y) e^{-\xi y} dy = \int_{[0, e^{i \alpha} \infty)} u(y) e^{-\xi y} dy \\
 & = e^{i \alpha} \int_0^\infty u(e^{i \alpha} r) \exp(-e^{i \alpha} \xi r) dr .
}
Suppose that $0 \le \Arg \xi \le \pi - \eps - \delta$. If we choose $\alpha = -\tfrac{\pi}{2} + \eps + \tfrac{1}{2} \delta$, then we find that
\formula{
 |\exp(-e^{i \alpha} \xi r)| & = \exp(-|\xi| r \cos \tfrac{\pi - \delta}{2})) \le \exp(-C(\delta) |\xi| r) ,
}
and therefore
\formula{
 |\laplace u(\xi)| & \le \int_0^\infty |u(e^{i \alpha} r)| \exp(-C(\delta) |\xi| r) dr \\
 & \le C(u, \eps) \int_0^\infty \exp(-C(u, \eps) r \log (1 + r) - C(\delta) |\xi| r) dr \le C(u, \eps, \delta) \, \frac{1}{1 + |\xi|} \, ,
}
as desired. A similar argument leads to the upper bound for $|\laplace v(\eta)|$.
\end{proof}

Our next result is a restatement of Theorem~\ref{thm:heat:spectral}.

\begin{proposition}
\label{prop:pt:uv}
Suppose that $X_t$ is a Lévy process with completely monotone jumps, and $f(\xi)$ is the corresponding Rogers function. Assume that $\eps > 0$, and
\formula{
 \limsup_{r \to \infty} \lvert\Arg \zeta_f(r)\rvert & < \frac{\pi}{2} - \eps .
}
Let $u(y)$ and $v(x)$ satisfy the assumptions of Lemma~\ref{lem:uv:est}: $u(y)$ and $v(x)$ are holomorphic functions in the region $\{z \in \C : \lvert\Arg z\rvert < \tfrac{\pi}{2} - \eps\}$, which are real-valued on $(0, \infty)$, and
\formula{
 |u(y)| & \le C(u, \eps) e^{-C(u, \eps) |y| \log(1 + |y|)} , & |v(x)| & \le C(v, \eps) e^{-C(v, \eps) |x| \log(1 + |x|)}
}
for $x$ and $y$ in the region $\{z \in \C : \lvert\Arg z\rvert < \tfrac{\pi}{2} - \eps\}$. Then
\formula*[eq:pt:uv]{
 & \int_0^\infty \int_{(0, \infty)} u(y) v(x) p_t^+(x, dy) dx \\
 & \qquad = \frac{2}{\pi} \int_{Z_f} e^{-t \lambda_f(r)} \biggl(\int_0^\infty F_{f+}(r; y) u(y) dy\biggr) \biggl(\int_0^\infty F_{f-}(r; x) v(x) dx\biggr) \lvert\zeta_f'(r)\rvert dr
}
for $t \ge 0$.
\end{proposition}

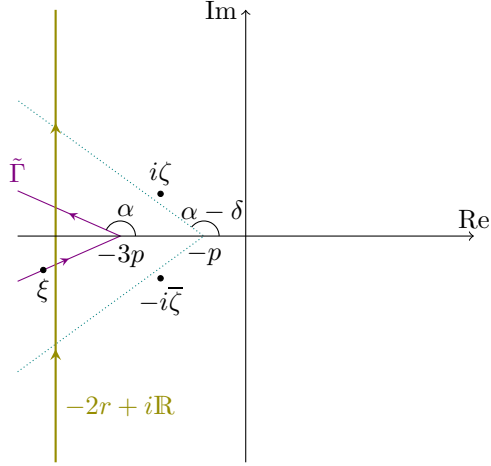
\begin{figure}
\centering
\begin{tikzpicture}[decoration={markings, mark=at position 0.25 with {\arrow{stealth}}, mark=at position 0.75 with {\arrow{stealth}}}]
\footnotesize
\coordinate (X) at (3,0);
\coordinate (Y) at (0,3);
\coordinate (Xn) at (-3,0);
\coordinate (Yn) at (0,-3);
\coordinate (O) at (0,0);
\coordinate (Pn) at (-0.55,0);
\coordinate (PPPn) at (-1.65,0);
\coordinate (RRn) at (-2.5,0);
\coordinate (Z) at (-3,1.8);
\coordinate (Zn) at (-3,-1.8);
\coordinate (ZZ) at (-3,0.6);
\coordinate (ZZn) at (-3,-0.6);
\coordinate (xi) at (-2.6625,-0.45);
\coordinate (zeta) at (-1.12,0.56);
\coordinate (zetabar) at (-1.12,-0.56);
\draw[violet, postaction=decorate] (ZZn) -- (PPPn) -- (ZZ) node[above] {$\tilde{\Gamma}$};
\draw[densely dotted, teal] (Zn) -- (Pn) -- (Z);
\draw[thick, olive, postaction=decorate] (-2.5,-3) -- (RRn) node[near start, right] {$-2 r + i \R$} -- (-2.5,3);
\node[below] at (Pn) {$-p$};
\node[below] at (PPPn) {$-3p$};
\filldraw (xi) circle[radius=1pt] node[below] {$\xi$};
\filldraw (zeta) circle[radius=1pt] node[above] {$i \zeta$};
\filldraw (zetabar) circle[radius=1pt] node[below] {$-i \overline{\zeta}$};
\pic["$\alpha$",draw,angle eccentricity=1.75,angle radius=0.2cm] {angle=X--PPPn--ZZ};
\pic["$\alpha-\delta$",draw,angle eccentricity=1.75,angle radius=0.2cm] {angle=X--Pn--Z};
\draw[->] (Xn) -- (X) node[above] {$\re$};
\draw[->] (Yn) -- (Y) node[left] {$\im$};
\end{tikzpicture}
\caption{Setting for the proof of Proposition~\ref{prop:pt:uv}: $\xi$ lies on the purple contour $\tilde{\Gamma}$, while $i \zeta$ and $-i \overline{\zeta}$ lie to the right of the teal dotted line.}
\label{fig:pt:uv}
\end{figure}

\begin{proof}
By the assumption, there are $\delta > 0$ and $p > 0$ such that
\formula{
 \Arg (\zeta(r) + i p) & \ge -\tfrac{\pi}{2} + \eps + 2 \delta , & \Arg (\zeta(r) - i p) & \le \tfrac{\pi}{2} - \eps - 2 \delta
}
for every $r \in Z_f$. We let $\alpha = \pi - \eps - \delta$. The remaining part of the proof is divided into four steps.

\smallskip

\emph{Step 1.}
Observe that $e^{-2 r y} F_+(r; y)$ and $e^{2 r y} u(y)$ are square integrable. Hence, by Plancherel's theorem,
\formula[]{
 \notag
 \int_0^\infty F_+(r; y) u(y) dy & = \int_0^\infty (e^{-2 r y} F_+(r; y)) (e^{2 r y} u(y)) dy \\
 \label{eq:fpu0}
 & = \frac{1}{2 \pi} \int_{-\infty}^\infty \laplace F_+(r; 2 r - i s) \laplace u(-2 r + i s) ds \\
 \notag
 & = \frac{1}{2 \pi i} \int_{-2 r + i \R} \laplace F_+(r; -\xi) \laplace u(\xi) d\xi .
}
Our goal now is to deform the contour of integration from $-2 r + i \R$ to
\formula{
 \tilde{\Gamma} & = (-3 p + e^{-i \alpha} \infty, -3 p] \cup [-3 p, -3 p + e^{i \alpha} \infty) ,
}
which no longer depends on $r$ (see Figure~\ref{fig:pt:uv}). Note that $\tilde{\Gamma}$ is the boundary of the region $\{\xi \in \C : \lvert\Arg (3 p + \xi)\rvert \le \alpha\}$.

Recall that
\formula{
 \laplace F_+(r; -\xi) & = \frac{\re \zeta(r)}{|f^+(r; -i \zeta(r))|} \, \frac{f^+(r; -\xi)}{(\xi - i \zeta(r)) (\xi + i \overline{\zeta(r)})}
}
is a holomorphic function in $\C \setminus ([0, \infty) \cup \{i \zeta(r), -i \overline{\zeta(r)}\})$. By Lemmas~\ref{lem:eig:lbound} and~\ref{lem:uv:est}, we have
\formula{
 |\laplace F_+(r; -\xi) \laplace u(\xi)| & \le \frac{C(f, p, r)}{1 + |\xi|} \, \frac{C(u, \eps, \delta, p)}{1 + |\xi|} = \frac{C(f, u, \eps, \delta, p, r)}{(1 + |\xi|)^2}
}
in the regions contained between the contours $-2 r + i \R$ and $\tilde{\Gamma}$, namely the triangle
\formula{
 \{\xi \in \C : \lvert\Arg (3 p + \xi)\rvert \ge \alpha \text{ and } \re \xi \ge -2 r\}
}
(if $-2 r < -3 p$), and the unbounded region
\formula{
 \{\xi \in \C : \lvert\Arg (3 p + \xi)\rvert \le \alpha \text{ and } \re \xi \le -2 r\}
}
(which has one component if $-2 r > -3 p$ and two components otherwise). We stress that the poles at $\xi = i \zeta(r)$ and $\xi = -i \overline{\zeta(r)}$ lie outside of this region due to inequality $\Arg(\zeta(r) - i p) \le \tfrac{\pi}{2} - \eps - 2 \delta$. By~\eqref{eq:fpu0} and a standard contour deformation argument,
\formula[eq:pi:1a]{
 \int_0^\infty F_+(r; y) u(y) dy & = \frac{1}{2 \pi i} \int_{\tilde{\Gamma}} \laplace F_+(r; -\xi) \laplace u(\xi) d\xi .
}
Similarly, using the other inequality $\Arg(\zeta(r) + i p) \le \tfrac{\pi}{2} - \eps - 2 \delta$, we obtain
\formula[eq:pi:1b]{
 \int_0^\infty F_-(r; x) v(x) dx & = \frac{1}{2 \pi i} \int_{\tilde{\Gamma}} \laplace F_-(r; -\eta) \laplace v(\eta) d\eta .
}
It follows that if we denote the right-hand side of~\eqref{eq:pt:uv} by $I$, then
\formula[eq:pt:aux1]{
 I & = \frac{2}{\pi} \int_{Z_f} e^{-t \lambda(r)} \biggl( \frac{1}{2 \pi i} \int_{\tilde{\Gamma}} \laplace F_+(r; -\xi) \laplace u(\xi) d\xi \biggr) \biggl( \frac{1}{2 \pi i} \int_{\tilde{\Gamma}} \laplace F_-(r; -\eta) \laplace v(\eta) d\eta \biggr) \lvert\zeta'(r)\rvert dr ,
}
and in particular $I$ is well-defined if the outer integral on the right-hand side of~\eqref{eq:pt:aux1} is well-defined. We claim that in fact the triple integral on the right-hand side is absolutely convergent, and hence not only indeed both sides of~\eqref{eq:pt:aux1} are well-defined, but also we may apply Fubini's theorem to change the order of integration on the right-hand side.

\smallskip

\emph{Step 2.}
Let $r \in Z_f$ and $\xi \in \tilde{\Gamma}$. Observe that (see Figure~\ref{fig:pt:uv})
\formula{
 \lvert\Arg(2 p + \xi)\rvert & \ge \lvert\Arg(3 p + \xi)\rvert = \alpha
}
and
\formula{
 \lvert\Arg(2 p + i \zeta(r))\rvert & \ge \lvert\Arg(p + i \zeta(r))\rvert \ge \alpha - \delta ,
}
so that, by~\eqref{eq:triangle}, we have $|\xi - i \zeta(r)| \ge (r + |\xi|) \sin \tfrac{\delta}{2}$. Also, $|\xi - i \zeta(r)| \ge 2 p \sin \delta$. It follows that
\formula{
 \dist(-i \xi, \Gamma) & \ge C(\delta, p) (1 + |\xi|) .
}
Also, $\lvert\Arg(-\xi)\rvert < \tfrac{\pi}{2}$. Thus, by Lemma~\ref{lem:eig:lbound}, if $\thet(r) = \Arg \zeta(r)$, we have
\formula{
 |\laplace F_+(r; -\xi)| & \le C(f, \delta, p) \, \frac{1}{r + |\xi|} \, \frac{\cos \thet(r)}{\cos(\tfrac{1}{2} \thet(r) - \tfrac{\pi}{4})}
}
On the other hand, by Lemma~\ref{lem:uv:est},
\formula{
 |\laplace u(\xi)| & \le C(u, \eps, \delta, p) \, \frac{1}{1 + |\xi|} \, .
}
Combining the above two estimates, we obtain
\formula{
 \int_{\tilde{\Gamma}} |\laplace F_+(r; -\xi) \laplace u(\xi)| |d\xi| & \le C(f, u, \eps, \delta, p) \, \frac{\cos \thet(r)}{\cos(\tfrac{1}{2} \thet(r) - \tfrac{\pi}{4})} \int_{\tilde{\Gamma}} \frac{1}{(1 + |\xi|) (r + |\xi|)} |d\xi|
}
The integral on the right-hand side is bounded by
\formula{
 C(\eps, \delta) \int_p^\infty \frac{1}{(1 + s) (r + s)} \, ds & \le C(\eps, \delta) \int_0^\infty \frac{1}{(1 + s) (r + s)} \, ds = C(\eps, \delta) \, \frac{\log r}{r - 1} \, .
}
Thus,
\formula[eq:pi:2]{
 \int_{\tilde{\Gamma}} |\laplace F_+(r; -\xi) \laplace u(\xi)| |d\xi| & \le C(f, u, \eps, \delta, p) \, \frac{\cos \thet(r)}{\cos(\tfrac{1}{2} \thet(r) - \tfrac{\pi}{4})} \, \frac{\log r}{r - 1} \, .
}
An analogous estimate holds for the integral of $|\laplace F_-(r; -\eta) \laplace v(\eta)|$:
\formula[eq:pi:3]{
 \int_{\tilde{\Gamma}} |\laplace F_-(r; -\eta) \laplace v(\eta)| |d\eta| & \le C(f, v, \eps, \delta, p) \, \frac{\cos \thet(r)}{\cos(\tfrac{1}{2} \thet(r) + \tfrac{\pi}{4})} \, \frac{\log r}{r - 1} \, .
}
Additionally, $\cos(\tfrac{\thet}{2} - \tfrac{\pi}{4}) \cos(\tfrac{\thet}{2} + \tfrac{\pi}{4}) = \tfrac{1}{2} \cos \thet$ (see~\eqref{eq:coscos}), and so we find that
\formula{
 & \biggl( \int_{\tilde{\Gamma}} |\laplace F_+(r; -\xi) \laplace u(\xi)| |d\xi| \biggr) \biggl( \int_{\tilde{\Gamma}} |\laplace F_-(r; -\eta) \laplace v(\eta)| |d\eta| \biggr) \\
 & \qquad \le C(f, u, \eps, \delta, p) \biggl(\frac{\log r}{r - 1}\biggr)^2 \cos \thet(r) .
}
It follows that
\formula{
 & \int_{Z_f} e^{-t \lambda(r)} \biggl( \int_{\tilde{\Gamma}} |\laplace F_+(r; -\xi) \laplace u(\xi)| |d\xi| \biggr) \biggl( \int_{\tilde{\Gamma}} |\laplace F_-(r; -\eta) \laplace v(\eta)| |d\eta| \biggr) \lvert\zeta'(r)\rvert dr \\
 & \qquad \le C(u, v, \eps, \delta, p) \int_0^\infty \biggl(\frac{\log r}{r - 1} \biggr)^2 \lvert\zeta'(r)\rvert dr ,
}
and by Proposition~\ref{prop:r:real}\ref{it:r:real:d} and a simple calculation, the right-hand side is finite. Indeed: since $\log r / (r - 1)$ is decreasing on $(0, \infty)$, we have
\formula{
 \int_0^\infty \biggl(\frac{\log r}{r - 1} \biggr)^2 \lvert\zeta'(r)\rvert dr & = \sum_{n = -\infty}^\infty \int_{2^n}^{2^{n + 1}} \biggl(\frac{\log r}{r - 1} \biggr)^2 \lvert\zeta'(r)\rvert dr \\
 & \le \sum_{n = -\infty}^\infty \biggl(\frac{n \log 2}{2^n - 1}\biggr)^2 \int_{2^n}^{2^{n + 1}} \lvert\zeta'(r)\rvert dr \\
 & \le 300 (\log 2)^2 \sum_{n = -\infty}^\infty \frac{2^n n^2}{(2^n - 1)^2} < \infty ,
}
where, abusing the notation, we agree that $n \log 2 / (2^n - 1) = 1$ when $n = 0$. Our claim from Step~1 is thus proved.

\smallskip

\emph{Step 3.}
We are now almost ready to complete the proof. By~\eqref{eq:pt:aux1} and Fubini's theorem,
\formula{
 I & = \frac{2}{\pi} \, \frac{1}{2 \pi i} \, \frac{1}{2 \pi i} \int_{\tilde{\Gamma}} \int_{\tilde{\Gamma}} \biggl( \int_{Z_f} e^{-t \lambda(r)} \laplace F_+(r; -\xi) \laplace F_-(r; -\eta) \lvert\zeta'(r)\rvert dr \biggr) \laplace u(\xi) \laplace v(\eta) d\xi d\eta .
}
Using Proposition~\ref{prop:pt:lap}, we find that
\formula[eq:pt:aux3]{
 I & = \frac{2}{\pi} \, \frac{1}{2 \pi i} \, \frac{1}{2 \pi i} \int_{\tilde{\Gamma}} \int_{\tilde{\Gamma}} \biggl( \int_0^\infty \int_{(0, \infty)} e^{\eta x + \xi y} p_t^+(x, dy) dx \biggr) \laplace u(\xi) \laplace v(\eta) d\xi d\eta .
}
We will momentarily prove that the quadruple integral on the right-hand side is absolutely convergent. Thus, by Fubini's theorem,
\formula[eq:pt:aux2]{
 I & = \frac{2}{\pi} \int_0^\infty \int_{(0, \infty)} \biggl( \frac{1}{2 \pi i} \int_{\tilde{\Gamma}} e^{\xi y} \laplace u(\xi) d\xi \biggr) \biggl( \frac{1}{2 \pi i} \int_{\tilde{\Gamma}} e^{\eta x} \laplace v(\eta) d\eta \biggr) p_t^+(x, dy) dx .
}
To complete the proof, we apply an appropriate contour deformation and the Fourier inversion formula. Suppose that $y > 0$ and observe that
\formula{
 \frac{1}{2 \pi i} \int_{\tilde{\Gamma}} e^{\xi y} \laplace u(\xi) d\xi & = \frac{1}{2 \pi i} \lim_{R \to \infty} \int_{[-3 p + e^{-i \alpha} R, -3 p] \cup [-3 p, -3 p + e^{i \alpha} R]} e^{\xi y} \laplace u(\xi) d\xi .
}
Furthermore,
\formula{
 \biggl|\int_{[-3 p + e^{i \alpha} R, i R \sin \alpha]} e^{\xi y} \laplace u(\xi) d\xi\biggr| & \le \int_{-3 p + R \cos \alpha}^0 |e^{(s + i R \sin \alpha) y} \laplace u(s + i R \sin \alpha)| ds \\
 & \le \int_{-\infty}^0 e^{s y} |\laplace u(s + i R \sin \alpha)| ds .
}
By Lemma~\ref{lem:uv:est}, we have $|\laplace u(\xi)| \le C(u, \eps, \delta, p) (1 + |\xi|)^{-1}$ in the region contained between $\tilde{\Gamma}$ and $i \R$, that is, in $\{\xi \in \C : \lvert\Arg (3 p + \xi)\rvert \le \alpha , \, \re \xi \le 0\}$, and hence, by the dominated convergence theorem,
\formula{
 \lim_{R \to \infty} \int_{[-3 p + e^{i \alpha} R, i R \sin \alpha]} e^{\xi y} \laplace u(\xi) d\xi & = 0 .
}
Similarly,
\formula{
 \lim_{R \to \infty} \int_{[-3 p + e^{-i \alpha} R, -i R \sin \alpha]} e^{\xi y} \laplace u(\xi) d\xi & = 0 .
}
By Cauchy's theorem,
\formula{
 & \biggl(\int_{[-3 p + e^{-i \alpha} R, -3 p] \cup [-3 p, -3 p + e^{i \alpha} R]} + \int_{[-3 p + e^{i \alpha} R, i R \sin \alpha]} \\
 & \hspace*{10em} - \int_{[-i R \sin \alpha, i R \sin \alpha]} - \int_{[-3 p + e^{-i \alpha} R, -i R \sin \alpha]}\biggr) e^{\xi y} \laplace u(\xi) d\xi = 0 ,
}
and therefore
\formula{
 \frac{1}{2 \pi i} \int_{\tilde{\Gamma}} e^{\xi y} \laplace u(\xi) d\xi & = \frac{1}{2 \pi i} \lim_{R \to \infty} \int_{[-i R \sin \alpha, i R \sin \alpha]} e^{\xi y} \laplace u(\xi) d\xi .
}
Finally, using the Fourier inversion formula (see, for example, Theorem~7.6 in~\cite{folland}), we conclude that
\formula{
 \frac{1}{2 \pi i} \int_{\tilde{\Gamma}} e^{\xi y} \laplace u(\xi) d\xi & = u(y) .
}
A similar argument shows that
\formula[eq:pi:4]{
 \frac{1}{2 \pi i} \int_{\tilde{\Gamma}} e^{\eta x} \laplace v(\eta) d\eta & = v(x)
}
for $x > 0$, and the desired result follows from~\eqref{eq:pt:aux2}.

\emph{Step 4.} It remains to prove that the integral on the right-hand side of~\eqref{eq:pt:aux3} is absolutely convergent. Denote
\formula{
 J & = \int_0^\infty \int_{(0, \infty)} \int_{\tilde{\Gamma}} \int_{\tilde{\Gamma}} |e^{\eta x + \xi y} \laplace u(\xi) \laplace v(\eta)| |d\xi| |d\eta| p_t^+(x, dy) dx .
}
Using Lemma~\ref{lem:uv:est}, we obtain
\formula{
 J & \le C(u, v, \eps, \delta, p) \int_0^\infty \int_{(0, \infty)} \int_{\tilde{\Gamma}} \int_{\tilde{\Gamma}} \frac{e^{x \re \eta + y \re \xi}}{(1 + |\xi|) (1 + |\eta|)} \, |d\xi| |d\eta| p_t^+(x, dy) dx .
}
The estimate
\formula{
 \int_{\tilde{\Gamma}} \frac{e^{y \re \xi}}{1 + |\xi|} \, |d\xi| & \le 2 \int_0^\infty \frac{e^{-3 p y + s y \cos \alpha}}{1 + s} \, ds \le C(\alpha) e^{-3 p y} \, \frac{\log(e + y^{-1})}{1 + y}
}
and a similar estimate for the integral with respect to $\eta$ lead to
\formula{
 J & \le C(u, v, \eps, \delta, p) \int_0^\infty \int_{(0, \infty)} \frac{\log(e + x^{-1}) \log(e + y^{-1})}{(1 + x) (1 + y)} \, p_t^+(x, dy) dx .
}
By the Cauchy--Schwarz inequality,
\formula{
 J & \le C(u, v, \eps, \delta, p) \biggl(\int_0^\infty \int_{(0, \infty)} \biggl(\frac{\log(e + x^{-1})}{1 + x}\biggr)^2 p_t^+(x, dy) dx\biggr)^{1/2} \times \\
 & \hspace*{10em} \times \biggl(\int_0^\infty \int_{(0, \infty)} \biggl(\frac{\log(e + y^{-1})}{1 + y}\biggr)^2 p_t^+(x, dy) dx\biggr)^{1/2} .
}
Finally, $\int_{(0, \infty)} p_t^+(x, dy) \le 1$ for every $x$, and, by Hunt's switching identity (Theorem~II.5 in~\cite{bertoin}), $\int_0^\infty p_t^+(x, A) dx \le |A|$ for every Borel $A \sub \R$. Thus,
\formula{
 J & \le C(u, v, \eps, \delta, p) \biggl(\int_0^\infty \biggl(\frac{\log(e + x^{-1})}{1 + x}\biggr)^2 dx\biggr)^{1/2} \biggl(\int_0^\infty \biggl(\frac{\log(e + y^{-1})}{1 + y}\biggr)^2 dy\biggr)^{1/2} < \infty ,
}
and the proof is complete.
\end{proof}

\begin{remark}
\label{rem:pi}
Under the assumptions of Proposition~\ref{prop:pt:uv}, we have
\formula{
 \biggl|\int_0^\infty F_{f+}(r; y) u(y) dy\biggr| & \le C(f, u) \, \frac{\log r}{r - 1} \, , \\
 \biggl|\int_0^\infty F_{f-}(r; x) v(x) dx\biggr| & \le C(f, v) \, \frac{\log r}{r - 1} \, ;
}
see~\eqref{eq:pi:1a}, \eqref{eq:pi:1b} and~\eqref{eq:pi:2} (and~\eqref{eq:coscos}). Combining these estimates with Proposition~\ref{prop:r:real}\ref{it:r:real:d}, we find that
\formula{
 \int_{Z_f} \biggl| \int_0^\infty F_{f+}(r; y) u(y) dy\biggr|^2 |\zeta_f'(r)| dr & < \infty
}
and
\formula{
 \int_{Z_f} \biggl| \int_0^\infty F_{f-}(r; x) v(x) dx\biggr|^2 |\zeta_f'(r)| dr & < \infty .
}
\end{remark}

The following restatement of Theorem~\ref{thm:heat} seems to be the best possible variant of Proposition~\ref{prop:pt:bounded} that can be proved by using only Proposition~\ref{prop:pt:uv} and Fubini's theorem.

\begin{corollary}
\label{cor:pt:log}
Suppose that $X_t$ is a Lévy process with completely monotone jumps, and $f(\xi)$ is the corresponding Rogers function. Let $\eps > 0$, and assume that
\formula{
 \limsup_{r \to \infty} \lvert\Arg \zeta_f(r)\rvert & < \frac{\pi}{2} - \eps .
}
Assume, furthermore, that for some $\beta \in (1, \tfrac{\pi}{2} / (\tfrac{\pi}{2} - \eps))$ we have
\formula[eq:pt:log]{
 \int_{Z_f} e^{s \lvert\im \zeta_f(r)\rvert - t \lambda_f(r)} \lvert\zeta_f'(r)\rvert dr & \le e^{C(f, t) (1 + s)^\beta}
}
whenever $t, s > 0$. Then $p_t^+(x, dy)$ has a continuous density function, given by
\formula[eq:pt]{
 p_t^+(x, y) & = \frac{2}{\pi} \int_{Z_f} e^{-t \lambda_f(r)} F_{f+}(r; y) F_{f-}(r; x) \lvert\zeta_f'(r)\rvert dr
}
for $t \ge 0$ and $x, y > 0$.
\end{corollary}

\begin{proof}
The argument is quite simple. Observe that since $\beta \in (1, \tfrac{\pi}{2} / (\tfrac{\pi}{2} - \eps))$, the functions $u(y) = \exp(-\xi y^\beta)$ and $v(x) = \exp(-\eta x^\beta)$ satisfy the condition in Proposition~\ref{prop:pt:uv} when $\xi, \eta > 0$. We use this result and Fubini's theorem to find that
\formula*[eq:pt:uv:aux]{
 & \int_0^\infty \int_{(0, \infty)} e^{-\xi y^\beta - \eta x^\beta} p_t^+(x, dy) dx \\
 & \qquad = \int_0^\infty \int_0^\infty \biggl( \frac{2}{\pi} \int_{Z_f} e^{-t \lambda(r)} F_+(r; y) F_-(r; x) \lvert\zeta'(r)\rvert dr \biggr) e^{-\xi y^\beta - \eta x^\beta} dx dy ;
}
we will momentarily justify the application of Fubini's theorem when $\xi$ and $\eta$ are large enough. Now, we exploit uniqueness of the (bivariate) Laplace transform. In variables $\tilde{x} = x^\beta$ and $\tilde{y} = y^\beta$, formula~\eqref{eq:pt:uv:aux} is an equality of Laplace transforms, evaluated at $(\eta, \xi)$, of two functions of $(\tilde{x}, \tilde{y})$. This gives the desired equality~\eqref{eq:pt} almost every pair $(\tilde{x}, \tilde{y})$, and hence for almost every pair $(x, y)$. At the end of the proof we use a continuity argument to conclude that in fact the equality holds everywhere.

We now prove absolute integrability of the integrand on the right-hand side of~\eqref{eq:pt:uv:aux}, thus justifying the use of Fubini's theorem. Let $b(r) = \im \zeta(r)$. Clearly, $|F_+(r; y)| \le e^{b(r) y} + 1 \le 2 e^{|b(r)| y}$ and similarly $|F_-(r; x)| \le e^{-b(r) x} + 1 \le 2 e^{|b(r)| x}$. It follows that
\formula{
 |e^{-t \lambda(r)} F_+(r; y) F_-(r; x) u(y) v(x) \zeta'(r)| & \le 4 e^{|b(r)| (x + y) - t \lambda(r) - \eta x^\beta - \xi y^\beta} \lvert\zeta'(r)\rvert .
}
By assumption, there is a constant $A$ (which depends only on $f$ and $t$) such that
\formula{
 \int_{Z_f} |e^{-t \lambda(r)} F_+(r; y) F_-(r; x) u(y) v(x) \zeta'(r)| dr & \le e^{A (1 + x + y)^\beta - \eta x^\beta - \xi y^\beta} .
}
Observe that $(1 + x + y)^\beta \le 3^\beta (1 + x^\beta + y^\beta)$. Thus, if $\xi, \eta > 1 + 3^\beta A$, we find that
\formula{
 \int_{Z_f} |e^{-t \lambda(r)} F_+(r; y) F_-(r; x) u(y) v(x) \zeta'(r)| dr & \le e^{3^\beta - x^\beta - y^\beta} ,
}
and the right-hand side is clearly integrable with respect to $x, y > 0$, as desired.

We thus know that~\eqref{eq:pt} holds for almost all $x, y > 0$. In order to extend this equality to all $x, y > 0$, we now show that both sides of~\eqref{eq:pt} are continuous functions of $x, y > 0$ with values in $[0, \infty]$. By assumption, $\lvert\Arg \zeta(r)\rvert \le \tfrac{\pi}{2} - \eps$ for $r$ large enough, and so, by Proposition~\ref{prop:r:bound}, we have $|f(r)| \ge C(\eps) \lambda(r)$ for $r$ large enough. It follows that with $R$ sufficiently large, we have
\formula{
 \int_{-\infty}^\infty |e^{-t f(\xi)}| d\xi & \le 2 \int_0^\infty e^{-t |f(r)|} dr \le 2 R + 2 \int_R^\infty e^{-t C(\eps) \lambda(r)} dr \\
 & \le 2 R + 2 \int_R^\infty e^{-t C(\eps) \lambda(r)} \lvert\zeta'(r)\rvert dr < \infty .
}
Therefore, for each $t > 0$ and $x \in \R$, the distribution of the random variable $X_t$ under $\pr^x$ has a density function $p_t(y - x)$ such that $p_t(y - x)$ is a jointly continuous function of $t > 0$ and $x, y \in \R$. By the Dynkin--Hunt formula, we have
\formula{
 p_t^+(x, y) & = p_t(y - x) - \ex^x(p_{t - \tau_{(0, \infty)}}(y - X_{\tau_{(0, \infty)}}) \ind_{\{\tau_{(0, \infty)} < t\}}) ,
}
and a standard argument shows that $p_t^+(x, y)$ is a jointly continuous function of $t, x, y > 0$; we refer to the proof of Theorem~2.4.3 in~\cite{ps:brownian}, which is written for the Brownian motion in $\R^d$, but applies verbatim to the present setting. On the other hand, the right-hand side of~\eqref{eq:pt} is a continuous function of $x, y > 0$ by the dominated convergence theorem, as we have already proved that
\formula{
 |e^{-t \lambda(r)} F_+(r; y) F_-(r; x) \zeta'(r)| & \le 4 e^{|b(r)| (x + y) - t \lambda(r)} \lvert\zeta'(r)\rvert \le 4 e^{2 B |b(r)| - t \lambda(r)} \lvert\zeta'(r)\rvert
}
for $x, y \in [0, B]$, and the right-hand side is integrable with respect to $r \in Z_f$ by assumption. This completes the proof.
\end{proof}

Corollary~\ref{cor:pt:log} is not optimal, in the sense that it does not cover all Lévy processes with completely monotone jumps for which formula~\eqref{eq:pt} holds true. For example, not all stable Lévy processes with index greater than $1$ satisfy~\eqref{eq:pt:log}, but~\eqref{eq:pt} is known to hold in this case; see Section~\ref{sec:ex} for further discussion. Nevertheless, any essential extension of Corollary~\ref{cor:pt:log} would require new methods or ideas. We remark that for the case of stable Lévy processes, the spine $\Gamma_f$ of $f(\xi)$ is a half-line originating at $0$, and $\lambda_f(r)$ is a power function, so that $\lambda_f(r)$ has an appropriate holomorphic extension. This was exploited in~\cite{kk:stable}, and the same concept was used in the related problem of finding the distribution of the hitting time of a point in~\cite{mucha}. Further applications of this technique, however, appear problematic: no general conditions seem to be known which would assert that $\lambda_f(r)$ extends to a holomorphic function in a sufficiently large sector.

\subsection{Infimum functional}
\label{sec:inf:inf}

Before we state our main results in this section, we prove a variant of Proposition~\ref{prop:pt:uv}.

\begin{proposition}
\label{prop:lpt:uv}
Suppose that $X_t$ is a Lévy process with completely monotone jumps, and $f(\xi)$ is the corresponding Rogers function. Assume that $\eps > 0$, and
\formula{
 \liminf_{r \to \infty} \Arg \zeta_f(r) > -\frac{\pi}{2} + \eps .
}
Let $v(x)$ satisfy the assumptions of Lemma~\ref{lem:uv:est}: $v(x)$ is a holomorphic function in the region $\{x \in \C : \lvert\Arg x\rvert < \tfrac{\pi}{2} - \eps\}$, which is real-valued on $(0, \infty)$, and such that
\formula{
 |v(x)| & \le C(v, \eps) e^{-C(v, \eps) |x| \log(1 + |x|)}
}
in the region $\{x \in \C : \lvert\Arg x\rvert < \tfrac{\pi}{2} - \eps\}$. Finally, suppose that $\xi > 0$ and $i \xi \in D_f^+$. Then
\formula*[eq:lpt:uv]{
 & \int_0^\infty \int_{(0, \infty)} v(x) e^{-\xi y} p_t^+(x, dy) dx \\
 & \qquad = \frac{2}{\pi} \int_{Z_f} e^{-t \lambda_f(r)} \laplace F_{f+}(r; \xi) \biggl(\int_0^\infty F_{f-}(r; x) v(x) dx\biggr) \lvert\zeta_f'(r)\rvert dr
}
for $t > 0$.
\end{proposition}

\begin{proof}
We follow closely the proof of Proposition~\ref{prop:pt:uv}, and only indicate necessary changes. Again we choose $\delta > 0$ and $p > 0$ such that
\formula{
 \Arg (\zeta_f(r) + i p) & \ge -\tfrac{\pi}{2} + \eps + 2 \delta
}
for every $r > 0$, and we let $\alpha = \pi - \eps - \delta$.

\smallskip

\emph{Step 1.}
This part is the same as the corresponding part of the proof of Proposition~\ref{prop:pt:uv}: as in~\eqref{eq:pi:1b}, we find that
\formula{
 \int_0^\infty F_-(r; x) v(x) dx & = \frac{1}{2 \pi i} \int_{\tilde{\Gamma}} \laplace F_-(r; -\eta) \laplace v(\eta) d\eta ,
}
where
\formula{
 \tilde{\Gamma} & = (-3 p + e^{-i \alpha} \infty, -3 p] \cup [-3 p, -3 p + e^{i \alpha} \infty) .
}
Note that the proof of the above identity only required the one-sided bound $\Arg (\zeta_f(r) + i p) \ge -\tfrac{\pi}{2} + \eps + 2 \delta$ (for the contour deformation argument; compare with the derivation of~\eqref{eq:pi:1a}), and it did not depend on the inequality $\Arg (\zeta_f(r) - i p) \le \tfrac{\pi}{2} - \eps - 2 \delta$ (assumed in Proposition~\ref{prop:pt:uv}, but not here).

It follows that if we denote the right-hand side of~\eqref{eq:lpt:uv} by $I$, then
\formula[eq:lpt:uv:aux1]{
 I & = \frac{2}{\pi} \int_{Z_f} e^{-t \lambda(r)} \laplace F_+(r; \xi) \biggl( \frac{1}{2 \pi i} \int_{\tilde{\Gamma}} \laplace F_-(r; -\eta) \laplace v(\eta) d\eta \biggr) \lvert\zeta'(r)\rvert dr ,
}
and in particular $I$ is well-defined if the outer integral on the right-hand side of~\eqref{eq:lpt:uv:aux1} is well-defined. As before, we claim that the double integral on the right-hand side is absolutely convergent, so that both sides of~\eqref{eq:lpt:uv:aux1} are well-defined, and additionally Fubini's theorem allows us to change the order of integration on the right-hand side.

\smallskip

\emph{Step 2.}
As in Step~2 of the proof of Proposition~\ref{prop:pt:uv} (see~\eqref{eq:pi:3}), we use Lemmas~\ref{lem:eig:lbound} and~\ref{lem:uv:est} to find that
\formula{
 \int_{\tilde{\Gamma}} |\laplace F_-(r; -\eta) \laplace v(\eta)| |d\eta| & \le C(f, v, \eps, \delta, p) \, \frac{\cos \thet(r)}{\cos(\tfrac{1}{2} \thet(r) + \tfrac{\pi}{4})} \, \frac{\log r}{r - 1} \, .
}
Again we note that for the proof of the above bound we only needed the inequality $\Arg (\zeta_f(r) + i p) \ge -\tfrac{\pi}{2} + \eps + 2 \delta$.

In addition to the above bound, we apply Lemma~\ref{lem:eig:lbound} again to find that
\formula{
 |\laplace F_+(r; \xi)| & \le C(f, \eps, \delta, p, \xi) \, \frac{\cos \thet(r)}{\cos(\tfrac{1}{2} \thet(r) - \tfrac{\pi}{4})} \, \frac{1}{1 + r} \, .
}
It follows that
\formula{
 & |\laplace F_+(r; \xi)| \biggl( \int_{\tilde{\Gamma}} |\laplace F_-(r; -\eta) \laplace v(\eta)| |d\eta| \biggr) \\
 & \qquad \le C(f, v, \eps, \delta, p, \xi) \, \frac{1}{1 + r} \, \frac{\log r}{r - 1} \, \cos \thet(r) .
}
As in the proof of Proposition~\ref{prop:pt:uv}, we conclude that
\formula{
 \int_{Z_f} e^{-t \lambda(r)} |\laplace F_+(r; \xi)| \biggl( \int_{\tilde{\Gamma}} |\laplace F_-(r; -\eta) \laplace v(\eta)| |d\eta| \biggr) \lvert\zeta'(r)\rvert & < \infty .
}

\smallskip

\emph{Step 3.}
The remaining part of the proof is very similar to Step~3 in the proof of Proposition~\ref{prop:pt:uv}. By~\eqref{eq:lpt:uv:aux1} and Fubini's theorem,
\formula{
 I & = \frac{2}{\pi} \frac{1}{2 \pi i} \int_{\tilde{\Gamma}}\biggl( \int_{Z_f} e^{-t \lambda(r)} \laplace F_+(r; \xi) \laplace F_-(r; -\eta) \lvert\zeta'(r)\rvert dr \biggr) \laplace v(\eta) d\eta .
}
Using Proposition~\ref{prop:pt:lap}, we find that
\formula{
 I & = \frac{2}{\pi} \frac{1}{2 \pi i} \int_{\tilde{\Gamma}} \biggl( \int_0^\infty \int_{(0, \infty)} e^{\eta x - \xi y} p_t^+(x, dy) dx \biggr) \laplace v(\eta) d\eta .
}
The triple integral on the right-hand side is absolutely convergent, by an argument very similar to the one used in the proof of Proposition~\ref{prop:pt:uv}:
\formula{
 & \int_0^\infty \int_{(0, \infty)} \int_{\tilde{\Gamma}} |e^{\eta x - \xi y} \laplace v(\eta)| |d\eta| p_t^+(x, dy) dx \\
 & \qquad \le C(v, \eps, \delta, p) \int_0^\infty \int_{(0, \infty)} \int_{\tilde{\Gamma}} \frac{e^{x \re \eta - \xi y}}{1 + |\eta|} \, |d\eta| p_t^+(x, dy) dx \\
 & \qquad \le C(v, \eps, \delta, p) \int_0^\infty \int_{(0, \infty)} \frac{e^{-3 p x} \log(e + x^{-1})}{1 + x} \, e^{-\xi y} p_t^+(x, dy) dx \\
 & \qquad \le C(v, \eps, \delta, p) \int_0^\infty \frac{e^{-3 p x} \log(e + x^{-1})}{1 + x} \, dx < \infty .
}
Thus, by Fubini's theorem,
\formula{
 I & = \frac{2}{\pi} \int_0^\infty \int_{(0, \infty)} \biggl( \frac{1}{2 \pi i} \int_{\tilde{\Gamma}} e^{\eta x} \laplace v(\eta) d\eta \biggr) e^{-\xi y} p_t^+(x, dy) dx .
}
As in the proof of Proposition~\ref{prop:pt:uv} (see~\eqref{eq:pi:4}),
\formula{
 \frac{1}{2 \pi i} \int_{\tilde{\Gamma}} e^{\eta x} \laplace v(\eta) d\eta & = v(x)
}
for $x > 0$, and the proof is complete.
\end{proof}

We provide a corollary analogous to Corollary~\ref{cor:pt:log}.

\begin{corollary}
\label{cor:lpt:log}
Suppose that $X_t$ is a Lévy process with completely monotone jumps, and $f(\xi)$ is the corresponding Rogers function. Assume that $\eps > 0$ and
\formula{
 \liminf_{r \to \infty} \Arg \zeta_f(r) & > -\frac{\pi}{2} + \eps .
}
Assume, furthermore, that for some $\beta \in (1, \tfrac{\pi}{2} / (\tfrac{\pi}{2} - \eps))$ we have
\formula[eq:lpt:log]{
 \int_{Z_f} e^{s \max\{-\im \zeta_f(r), 0\} - t \lambda_f(r)} \, \frac{1}{1 + r} \, \lvert\zeta_f'(r)\rvert dr & \le e^{C(f, t) (1 + s)^\beta}
}
whenever $t, s > 0$. Then
\formula[eq:lpt]{
 \int_{(0, \infty)} e^{-\xi y} p_t^+(x, dy) & = \frac{2}{\pi} \int_{Z_f} e^{-t \lambda_f(r)} \laplace F_{f+}(r; \xi) F_{f-}(r; x) \lvert\zeta_f'(r)\rvert dr
}
for $t > 0$, $x > 0$ and $\xi > 0$ such that $[i \xi, i \infty) \sub D_f^+$.
\end{corollary}

\begin{proof}
The argument is very similar to the proof of Corollary~\ref{cor:pt:log}. We use Proposition~\ref{prop:lpt:uv} for $v(x) = \exp(-\eta x^\beta)$ and Fubini's theorem, to find that
\formula*[eq:lpt:uv:aux]{
 & \int_0^\infty \biggl(\int_{(0, \infty)} e^{-\xi y} p_t^+(x, dy)\biggr) e^{-\eta x^\beta} dx \\
 & \qquad = \int_0^\infty \biggl( \frac{2}{\pi} \int_{Z_f} e^{-t \lambda(r)} \laplace F_+(r; \xi) F_-(r; x) \lvert\zeta'(r)\rvert dr \biggr) e^{-\eta x^\beta} dx .
}
Later, we justify the application of Fubini's theorem when $\eta$ is large enough. If we set $\tilde{x} = x^\beta$, then $v(x) = e^{-\eta \tilde{x}}$. Thus, both sides of the above equality are Laplace transforms (evaluated at $\eta$) of appropriate functions of $\tilde{x}$. By uniqueness of the Laplace transform, we conclude that these functions are equal for almost all $\tilde{x}$, or, equivalently, that
\formula[eq:lpt:ae]{
 \int_{(0, \infty)} e^{-\xi y} p_t^+(x, dy) & = \frac{2}{\pi} \int_{Z_f} e^{-t \lambda(r)} \laplace F_+(r; \xi) F_-(r; x) \lvert\zeta'(r)\rvert dr
}
for almost all $x$. A continuity argument extends this result to all $x > 0$, as indicated at the end of the proof.

We now prove that Fubini's theorem indeed can be applied to the integral on the right-hand side of~\eqref{eq:lpt:uv:aux}. Let $b(r) = \im \zeta(r)$. Clearly,
\formula{
 |F_-(r; x)| & \le e^{-b(r) x} + 1 \le 2 e^{\max\{-b(r), 0\} x} .
}
Furthermore, by Lemma~\ref{lem:eig:lbound},
\formula{
 |\laplace F_+(r; \xi)| & \le C(f, \xi) \, \frac{1}{1 + r} \, .
}
Therefore,
\formula{
 |e^{-t \lambda(r)} \laplace F_+(r; \xi) F_-(r; x) v(x) \zeta'(r)| & \le C(f, \xi) e^{-t \lambda(r) + \max\{-b(r), 0\} x - \eta x^\beta} \, \frac{1}{1 + r} \, \lvert\zeta'(r)\rvert .
}
By assumption, there is a constant $A$ (which depends only on $f$ and $t$) such that
\formula[eq:lpt:bound]{
 \int_0^\infty |e^{-t \lambda(r)} \laplace F_+(r; \xi) F_-(r; x) v(x) \zeta'(r)| dr & \le C(f, \xi) e^{A (1 + x)^\beta - \eta x^\beta} .
}
With $\eta > A$, the right-hand side is clearly integrable with respect to $x > 0$. This completes the justification of the use of Fubini's theorem.

It remains to discuss the continuity argument that extends~\eqref{eq:lpt:ae} to all $x > 0$. The left-hand side of~\eqref{eq:lpt:ae} is continuous in $x > 0$ by the strong Feller property of $p_t^+(x, dy)$; see the Corollary on page~71 in~\cite{chung}. Continuity of the right-hand side follows by the dominated convergence theorem: we have already seen that
\formula{
 |e^{-t \lambda(r)} \laplace F_+(r; \xi) F_-(r; x) \zeta'(r)| & \le C(f, \xi) e^{-t \lambda(r) + \max\{-b(r), 0\} B} \, \frac{1}{1 + r} \, \lvert\zeta'(r)\rvert
}
for $x \in [0, B]$ (see~\eqref{eq:lpt:bound}), and the right-hand side is integrable with respect to $r \in Z_f$ by assumption.
\end{proof}

In order to set $\xi = 0$ and obtain the following restatement of Theorem~\ref{thm:extrema}\ref{thm:extrema:b}, we need stronger assumptions.

\begin{corollary}
\label{cor:inf:log}
Suppose that $X_t$ is a Lévy process with completely monotone jumps, and $f(\xi)$ is the corresponding Rogers function. Let $\eps > 0$, $\delta \in [0, \tfrac{\pi}{2})$, $\ro > 0$, and assume that
\formula[eq:sup]{
 \liminf_{r \to \infty} \Arg \zeta_f(r) & > -\frac{\pi}{2} + \eps
}
and
\formula[eq:inf]{
 \inf \{ \Arg f(i e^{-i \delta} r) : r \in (0, \ro) \} > 0
}
(if $\delta = 0$, we understand that $\Arg f(i e^{-i \delta} r) = \Arg f(i r) = \ph(-r)$, where $\ph(s)$ is the function in the exponential representation~\eqref{eq:r:exp} of the Rogers function $f(\xi)$). Assume, furthermore, that for some $\beta \in (1, \tfrac{\pi}{2} / (\tfrac{\pi}{2} - \eps))$ we have
\formula[eq:inf:log]{
 \int_{Z_f} e^{s \max\{-\im \zeta_f(r), 0\} - t \lambda_f(r)} \, \frac{1}{1 + r} \, \lvert\zeta_f'(r)\rvert dr & \le e^{C(f, t) (1 + s)^\beta}
}
whenever $t, s > 0$. Then
\formula{
 p_t^+(x, (0, \infty)) & = \frac{2}{\pi} \int_{Z_f} e^{-t \lambda_f(r)} \laplace F_{f+}(r; 0^+) F_{f-}(r; x) \lvert\zeta_f'(r)\rvert dr
}
for $t > 0$ and $x > 0$.
\end{corollary}

\begin{proof}
The desired result follows immediately from Corollary~\ref{cor:lpt:log}, once we show that we may apply the dominated convergence theorem to the right-hand side of~\eqref{eq:lpt} as $\xi \to 0^+$. By assumption, we have $\Arg \zeta_f(r) \le \tfrac{\pi}{2} - \delta$ when $0 < r < \ro$, and hence $\delta \le \Arg (i \xi / \zeta_f(r)) \le \pi$ when additonally $\xi > 0$. By~\eqref{eq:triangle}, $|i \xi - \zeta_f(r)| \ge (\xi + r) \sin \tfrac{\delta}{2} \ge C(\delta) \xi$ when $0 < r < \ro$ and $\xi > 0$. If $0 < \xi < \tfrac{1}{2} \ro$ and $r \ge \ro$, then clearly $|i \xi - \zeta_f(r)| \ge r - \xi \ge \xi$, and hence $\dist(i \xi, \Gamma) \ge C(\delta) \xi$ if $0 < \xi < \tfrac{1}{2} \ro$. By Lemma~\ref{lem:eig:lbound},
\formula{
 |\laplace F_+(r; \xi)| & \le C(\delta) \, \frac{1}{\xi + r} \le C(\delta) \, \frac{1}{r}
}
when $r \in Z_f$ and $0 < \xi < \tfrac{1}{2} \ro$. Furthermore, by Lemma~\ref{lem:eig:est} and Corollary~\ref{cor:cbf:bound},
\formula{
 |F_-(r; x)| & \le C r x \biggl(1 + \frac{\re \zeta(r)}{r} \, \frac{f^-(r; 1/x)}{|f^-(r; i \zeta(r))|} \biggr) \le C r x + C(\delta) r x \, \frac{f^-(r; 1/x)}{f^-(r; r)}
}
when $0 < r < 1 / (2 x)$, and thus, by Lemma~\ref{lem:r:balance:quot},
\formula{
 |F_-(r; x)| & \le C r x + C(f, \delta, \eta, \ro) (r x)^{C(f, \eta, \ro)} \le C(f, \delta, \eta, \ro, x) r^{C(f, \eta, \ro)}
}
when $0 < r < C(f, \eta, \ro, x)$. It follows that
\formula{
 |e^{-t \lambda(r)} \laplace F_{f+}(r; \xi) F_{f-}(r; x) \zeta'(r)| & \le C(f, \delta, \eta, \ro, x) r^{C(f, \eta, \ro) - 1} \lvert\zeta'(r)\rvert
}
on some initial interval $(0, R)$, with $R = C(f, \eta, \ro, x)$. Note that by Proposition~\ref{prop:r:real}\ref{it:r:real:d}, the right-hand side is integrable with respect to $r$ over $(0, R)$. On the other hand,
\formula{
 |F_-(r; x)| & \le 2 e^{x \max\{-\im \zeta_f(r), 0\}} ,
}
so that
\formula{
 |e^{-t \lambda(r)} \laplace F_{f+}(r; \xi) F_{f-}(r; x) \zeta'(r)| & \le C(\delta) \, \frac{e^{x \max\{-\im \zeta_f(r), 0\} - t \lambda(r)}}{r} \, \lvert\zeta'(r)\rvert ,
}
and by assumption, the right-hand side is integrable with respect to $r$ over $[R, \infty)$. Consequently, we may apply the dominated convergence theorem to find that
\formula{
 \lim_{\xi \to 0^+} \int_0^\infty e^{-t \lambda(r)} \laplace F_{f+}(r; \xi) F_{f-}(r; x) \zeta'(r) dr & = \int_0^\infty e^{-t \lambda(r)} \laplace F_{f+}(r; 0^+) F_{f-}(r; x) \zeta'(r) dr ,
}
and the desired result follows.
\end{proof}

Corollary~\ref{cor:inf:log} seems to be the best result of its kind which can be obtained by the above method. However, just as it was the case with Corollary~\ref{cor:pt:log}, it is not optimal, in the sense that it does not cover all Lévy processes with completely monotone jumps for which formulae~\eqref{eq:sup} and~\eqref{eq:inf} hold true. One way to extend Corollary~\ref{cor:pt:log} involves contour deformation of the integral with respect to $r$, as in~\cite{kk:stable} and, in a different context, in~\cite{mucha}. Another possible approach is to follow the same argument for $\kappa^-(\sigma, \eta) / \kappa^-(\sigma, \xi) = f_\sigma^-(\eta) / f_\sigma^-(\xi)$ rather than $\kappa^\circ(\sigma) \kappa^+(\sigma, \xi) \kappa^-(\sigma, \eta) = f_\sigma^+(\xi) f_\sigma^-(\eta)$, then set $\xi \to 0^+$, and use the Pecherskii--Rogozin identity~\eqref{eq:pr}. This, however, is beyond the scope of the present paper.

We conclude this section by observing that
\formula{
 p_t^+(x, (0, \infty)) & = \pr^x(\tau_{(0, \infty)} > t) = \pr^x(\underline{X}_t > 0) = \pr^0(\underline{X}_t > -x) .
}
Therefore, Corollary~\ref{cor:inf:log} is indeed a restatement of Theorem~\ref{thm:extrema}\ref{thm:extrema:b}. Clearly, Theorem~\ref{thm:extrema}\ref{thm:extrema:a} follows from Theorem~\ref{thm:extrema}\ref{thm:extrema:b} applied to the dual Lévy process $\hat{X}_t = 2 X_0 - X_t$.

%
%

\section{Examples}
\label{sec:ex}

Below we discuss several classes of Lévy processes with completely monotone jumps, and we check whether the assumptions of our main results are satisfied. Needless to say, closed-form expressions for the generalised eigenfunctions are not available unless the corresponding Rogers function $f(\xi)$ has a particularly simple form --- for example, it is a rational function.

\subsection{Processes with Brownian component}

We claim that if $X_t$ has a non-zero Gaussian coefficient, that is, if $a > 0$ in the Stieltjes representation~\eqref{eq:r:int} of the characteristic exponent $f(\xi)$ of $X_t$, then Theorem~\ref{thm:heat} applies.

In this case it is easy to see that, by Stieltjes representation~\eqref{eq:r:int} and the dominated convergence theorem, $f(\xi) / \xi^2$ converges to $a$ as $|\xi| \to \infty$, uniformly in the sector $\{\xi \in \C : \lvert\Arg \xi\rvert \le \tfrac{\pi}{2} - \eps\}$ for every $\eps > 0$. Therefore, $\lim_{r \to \infty} \Arg f(r e^{i \alpha}) = 2 \alpha$ for every $\alpha \in (-\tfrac{\pi}{2}, \tfrac{\pi}{2})$, and consequently
\formula{
 \lim_{r \to \infty} \lvert\Arg \zeta_f(r)\rvert & = 0 .
}
Informally speaking, this means that the spine $\Gamma_f$ of $f(\xi)$ is asymptotically horizontal. In particular, assumption~\eqref{eq:heat:a1} in Theorem~\ref{thm:heat} holds with an arbitrary $\eps \in (0, \tfrac{\pi}{2})$.

The above observation additionally implies that $\lvert\im \zeta_f(r)\rvert \le C (1 + r)$ for $r \in Z_f$. Furthermore, by Proposition~\ref{prop:r:bound}, we have $\lambda_f(r) \ge C |f(r)| \ge C r^2$ for $r$ large enough, and hence $\lambda_f(r) \ge C r^2 - 1$ for all $r \in Z_f$. These estimates imply that condition~\eqref{eq:heat:a2} holds with $\beta = 2$: we have
\formula{
 \int_{Z_f} e^{s \lvert\im \zeta_f(r)\rvert - t \lambda_f(r)} \lvert\zeta_f'(r)\rvert dr & \le \int_{Z_f} e^{C s (1 + r) - C t r^2 + t} \lvert\zeta_f'(r)\rvert dr \le \frac{C}{\sqrt{t}} \, e^{C s^2 / t + C s + t} .
}
By choosing $\eps > \tfrac{\pi}{4}$ and $\beta = 2$, we see that indeed all assumptions of Theorem~\ref{thm:heat} are satisfied when $X_t$ has a non-zero Gaussian coefficient.

Similarly, conditions~\eqref{eq:extrema:sup:a2} and~\eqref{eq:extrema:sup:a4} in Theorem~\ref{thm:extrema}\ref{thm:extrema:a}, as well as conditions~\eqref{eq:extrema:inf:a2} and~\eqref{eq:extrema:inf:a4} in Theorem~\ref{thm:extrema}\ref{thm:extrema:b}, are satisfied with $\eps > \tfrac{\pi}{4}$ and $\beta = 2$. The remaining assumptions~\eqref{eq:extrema:sup:a3} and~\eqref{eq:extrema:inf:a3} require certain balance between large jumps and thus they are essentially independent of the Gaussian coefficients. Therefore, when applying Theorem~\ref{thm:extrema} for Lévy processes with completely monotone jumps and non-zero Gaussian coefficient, these conditions need to be imposed separately.

\subsection{Strictly stable processes}
\label{sec:ex:strictly}

Suppose that $X_t$ is a strictly stable Lévy process with index $\alpha \in (0, 2]$ and positivity parameter $\ro = \pr(X_t > 0)$. The admissible range for $\ro$ is
\formula{
 \max\{0, 1 - \tfrac{1}{\alpha}\} & \le \ro \le \min\{1, \tfrac{1}{\alpha}\} .
}
If we let
\formula{
 \thet & = (2 \ro - 1) \tfrac{\pi}{2} ,
}
then there is $k > 0$ such that for $\xi > 0$ the characteristic exponent is given by
\formula{
 f(\xi) & = k (e^{-i \thet} \xi)^\alpha .
}
Clearly, the above equality extends to the half-plane $\re \xi > 0$. Note that
\formula[eq:stable:theta]{
 |\thet| & \le \min\biggl\{\frac{\pi}{2}, \frac{2 - \alpha}{\alpha} \, \frac{\pi}{2}\biggr\} .
}
This example is briefly mentioned as Example~5.2(b) in~\cite{kwasnicki:rogers}; for a general discussion of stable processes, we refer to~\cite{bertoin,sato}.

With the above notation, we have $\zeta_f(r) = e^{i \thet} r$ and $\Arg \zeta_f(r) = \thet$: the spine $\Gamma_f$ of $f(\xi)$ is a ray $(0, e^{i \thet} \infty)$. Hence, condition~\eqref{eq:heat:a1} in Theorem~\ref{thm:heat} is satisfied unless $\ro = 0$ or $\ro = 1$ (and necessarily $\alpha \le 1$), while condition~\eqref{eq:heat:a2} takes form
\formula[eq:stable]{
 \int_0^\infty e^{s r \lvert\sin \thet\rvert - k t r^\alpha} dr & \le e^{C(f, t) (1 + s)^\beta} ,
}
with an arbitrary $\beta \in (1, \tfrac{\pi}{2 |\thet|})$. If $\thet = 0$ (that is, $\ro = \tfrac{1}{2}$), then the left-hand side is bounded by $C(f, t)$, and so the estimate is clearly satisfied. When $\thet \ne 0$ and $\alpha \le 1$, then~\eqref{eq:stable} fails to hold with any $\beta$. Finally, if $\thet \ne 0$ and $\alpha > 1$, then a simple calculation shows that
\formula{
 s r \lvert\sin \thet\rvert - \tfrac{1}{2} k t r^\alpha & \le C(\alpha, \thet, k, t) s^{\alpha / (\alpha - 1)} ,
}
so that
\formula{
 \int_0^\infty e^{s r \lvert\sin \thet\rvert - k t r^\alpha} dr & \le \int_0^\infty e^{C(\alpha, \thet, k, t) s^{\alpha / (\alpha - 1)} - \tfrac{1}{2} k t r^\alpha} dr \le C(\alpha, \thet, k, t) e^{C(\alpha, \thet, k, t) s^{\alpha / (\alpha - 1)}} .
}
Thus, condition~\eqref{eq:stable} holds true if (and only if) $\beta \ge \alpha / (\alpha - 1)$.

We conclude that in order to apply Theorem~\ref{thm:heat}, we need to assume that either:
\begin{itemize}
\item $\alpha \in (0, 1]$ and $\thet = 0$, that is, $\ro = \tfrac{1}{2}$;
\item $\alpha \in (1, \tfrac{3}{2}]$ and
\formula{
 |\thet| & < \frac{\alpha - 1}{\alpha} \, \frac{\pi}{2} \quad \text{or, equivalently,} \quad \frac{1}{2 \alpha} < \ro < 1 - \frac{1}{2 \alpha} \, ;
}
\item $\alpha \in (\tfrac{3}{2}, 2]$.
\end{itemize}
Thus, we only partially recover the results of~\cite{kk:stable}, where there are no restrictions on the parameter $\thet$ when $\alpha \in (1, \tfrac{3}{2}]$. As explained at the end of Section~\ref{sec:inf:pt}, the method applied in~\cite{kk:stable} relied on a deformation of the contour in the integral with respect to $r$. This approach can be repeated with the methods developed in the present paper, but it does not seem applicable in the more general setting of Lévy processes with completely monotone jumps considered here, and so we do not pursue this direction.

A very similar analysis shows that Theorem~\ref{thm:extrema}\ref{thm:extrema:a} applies under the following assumptions:
\begin{itemize}
\item $\alpha \in (0, 1]$ and $-\tfrac{\pi}{2} < \thet \le 0$, that is, $0 < \ro \le \tfrac{1}{2}$;
\item $\alpha \in (1, \tfrac{3}{2}]$ and
\formula{
 \thet & < \frac{\alpha - 1}{\alpha} \, \frac{\pi}{2} \quad \text{or, equivalently,} \quad \ro < 1 - \frac{1}{2 \alpha} \, ;
}
\item $\alpha \in (\tfrac{3}{2}, 2]$;
\end{itemize}
we omit the details. Similarly, the conditions required in order to apply Theorem~\ref{thm:extrema}\ref{thm:extrema:b} are obtained from the above ones by replacing $\thet$ by $-\thet$ (or, equivalently, $\ro$ by $1 - \ro$). Again, we only partially recover the results of~\cite{kk:stable}, where the restrictions on $\thet$ are not needed if $\alpha \in (1, \tfrac{3}{2}]$.

\subsection{Stable Lévy processes}
\label{sec:ex:stable}

For Lévy processes which are stable, but not strictly stable, the picture is very similar. If the index of stability $\alpha$ is not equal to $1$, we have
\formula{
 f(\xi) & = k (e^{-i \thet} \xi)^\alpha - i b \xi
}
when $\xi > 0$, and the above formula extends holomorphically to all $\xi$ with $\re \xi > 0$. Here $k > 0$, $\thet$ satisfies~\eqref{eq:stable:theta}, and $b \in \R \setminus \{0\}$. When $\alpha = 1$, then (the holomorphic extension of) the Rogers function $f(\xi)$ is given by
\formula{
 f(\xi) & = k e^{-i \thet} \xi + \frac{2 i \beta k \cos \thet}{\pi} \, \xi \log \xi
}
when $\re \xi > 0$, with $k > 0$, $\thet \in (-\tfrac{\pi}{2}, \tfrac{\pi}{2})$, and $\beta \in [-1, 1] \setminus \{0\}$.

If $\alpha \in (1, 2]$, the spine $\Gamma_f$ of $f(\xi)$ has an asymptote $(0, e^{i \thet} \infty)$ at infinity, and it is asymptotically vertical near $0$. More formally, $\Arg \zeta_f(r)$ converges to $\thet$ as $r \to \infty$, and to $\tfrac{\pi}{2} \sign b$ as $r \to 0^+$. For $\alpha = 1$, the spine $\Gamma_f$ is asymptotically vertical both at $0$ and at infinity. Finally, if $\alpha \in (0, 1)$, the spine $\Gamma_f$ is tangent to $(0, e^{i \thet} \infty)$ at $0$, and it is asymptotically vertical at infinity.

It can be verified that Theorem~\ref{thm:heat} applies only when $\alpha > 1$, under the same conditions as in the strictly stable case, namely if:
\begin{itemize}
\item $\alpha \in (1, \tfrac{3}{2}]$ and $|\thet| < \tfrac{\alpha - 1}{\alpha} \tfrac{\pi}{2}$;
\item $\alpha \in (\tfrac{3}{2}, 2]$.
\end{itemize}
Similarly, Theorem~\ref{thm:extrema}\ref{thm:extrema:a} applies if:
\begin{itemize}
\item $\alpha \in (0, 1)$, $b < 0$ and $\thet > -\tfrac{\pi}{2}$;
\item $\alpha = 1$ and $\beta > 0$;
\item $\alpha \in (1, \tfrac{3}{2}]$, $b > 0$ and $\thet < \tfrac{\alpha - 1}{\alpha} \tfrac{\pi}{2}$;
\item $\alpha \in (\tfrac{3}{2}, 2]$,
\end{itemize}
and conditions for applicability of Theorem~\ref{thm:extrema}\ref{thm:extrema:b} are obtained from the above ones by replacing $\thet$ by $-\thet$, $b$ by $-b$ and $\beta$ by $-\beta$. We omit the details.

\subsection{Stable-like processes: power-type growth}
\label{sec:ex:stable-like}

The analysis of the strictly stable case partially carries over to more general processes. Although the conditions of Theorem~\ref{thm:extrema} are generally more difficult to verify, it is relatively easy to extend the discussion of applicability of Theorem~\ref{thm:heat} in Sections~\ref{sec:ex:strictly} and~\ref{sec:ex:stable}. Indeed, the same argument applies if, for some $\alpha \in (1, 2)$ and a slowly varying function $\ell(r)$, we have
\formula[eq:stable-like:1]{
 \liminf_{r \to \infty} \frac{\lambda_f(r)}{r^\alpha \ell(r)} & > 0 , & \limsup_{r \to \infty} \lvert\Arg \zeta_f(r)\rvert & < \frac{\alpha - 1}{\alpha} \, \frac{\pi}{2} \, ;
}
we omit the details. In this section we provide a sufficient condition for the above two inequalities to hold true.

Note that by Proposition~\ref{prop:r:bound}, the two conditions in~\eqref{eq:stable-like:1} are equivalent to
\formula[eq:stable-like:2]{
 \liminf_{\xi \to \infty} \frac{|f(\xi)|}{\xi^\alpha \ell(\xi)} & > 0 , & \limsup_{r \to \infty} \lvert\Arg \zeta_f(r)\rvert & < \frac{\alpha - 1}{\alpha} \, \frac{\pi}{2} \, .
}
We remark that the former inequality in~\eqref{eq:stable-like:2} is satisfied if the \emph{lower Blumenthal--Getoor index} of $X_t$ is greater than $\alpha$; in particular it holds true if the \emph{lower Matuszewska index} of $|f(\xi)|$ at infinity is greater than $\alpha$. The latter condition in~\eqref{eq:stable-like:2} can be thought of as a quantitative version of the \emph{sector condition}.

Let us write $F \sim G$ if the ratio $F / G$ converges to $1$. We claim that both parts of~\eqref{eq:stable-like:2} are satisfied if $X_t$ has no Gaussian coefficient, and if the density $\nu(z)$ of the Lévy measure of $X_t$ satisfies
\formula{
 \nu(z) & \sim \frac{A_\pm}{|z|^{1 + \alpha} \ell(|z|^{-1})} && \text{as $z \to 0^\pm$,}
}
where $A_+, A_- \ge 0$, $A_+ + A_- > 0$, and, if $1 < \alpha \le \tfrac{3}{2}$, additionally
\formula{
 \biggl| \frac{A_+ - A_-}{A_+ + A_-} \tan \frac{\alpha \pi}{2} \biggr| & < \tan \frac{(\alpha - 1) \pi}{2} \, .
}
Of course, we continue to assume that $X_t$ is a Lévy process with completely monotone jumps, $\alpha \in (1, 2)$ and $\ell$ is a slowly varying function.

The proof of our claim is a standard, although lengthy, application of the theory of regular variation. Recall that for $z > 0$,
\formula{
 \nu(\pm z) & = \int_{(0, \infty)} e^{-z s} \mu_\pm(ds) ,
}
where $\mu_\pm$ are measures concentrated on $(0, \infty)$ such that $\mu(ds) = \mu_+(ds) + \mu_-(-ds)$ is the measure in the Stieltjes representation~\eqref{eq:r:int} of the Rogers function $f(\xi)$ (see Remark~3.4(c) in~\cite{kwasnicki:rogers}). By Karamata's Tauberian theorem (Theorem~1.7.1 in~\cite{bgt}), we have
\formula{
 \mu_{\pm}((0, s)) & \sim \frac{A_\pm s^{1 + \alpha} \ell(s)}{\Gamma(1 + \alpha)} && \text{as $s \to \infty$.}
}
Hence, by Karamata's Abelian theorem for Stieltjes transforms (Theorem~2.2 in~\cite{mst}), we find that for every $\thet \in (-\pi, \pi)$,
\formula{
 \int_1^\infty \frac{1}{(r e^{i \thet} + s)^3} \, \mu_\pm((0, s)) ds & \sim \frac{-\Gamma(1 - \alpha) A_\pm}{2} \, (r e^{i \thet})^{\alpha - 1} \ell(r) && \text{as $r \to \infty$;}
}
note that although Theorem~2.2 in~\cite{mst} is stated for $\thet = 0$, its proof carries over almost verbatim to the general case. Now we use Karamata's Abelian theorem for the integral of a regularly varying function over $(0, r)$ to find that
\formula{
 \int_1^\infty \biggl(\frac{1}{s^2} - \frac{1}{(r e^{i \thet} + s)^2} \biggr) \mu_\pm((0, s)) ds & \sim \Gamma(-\alpha) A_\pm (r e^{i \thet})^\alpha \ell(r) && \text{as $r \to \infty$.}
}
By a simple application of Fubini's theorem (or, informally, integration by parts), we arrive at
\formula{
 \int_{(0, 1]} \frac{e^{i \thet}}{s (e^{i \thet} + s)} \, \mu_\pm(ds) + \int_{(1, \infty)} \frac{r e^{i \thet}}{s (r e^{i \thet} + s)} \, \mu_\pm(ds) & \sim \Gamma(-\alpha) A_\pm (r e^{i \thet})^\alpha \ell(r) && \text{as $r \to \infty$.}
}
Adding a constant to the left-hand side does not affect the above asymptotic equality, and hence we conclude that
\formula[eq:ex:stable-like:aux]{
 \int_{(0, \infty)} \biggl(\frac{r e^{i \thet}}{1 + s} - \frac{r e^{i \thet}}{r e^{i \thet} + s}\biggr) \frac{\mu_\pm(ds)}{s} & \sim \Gamma(-\alpha) A_\pm (r e^{i \thet})^\alpha \ell(r) && \text{as $r \to \infty$.}
}
This is exactly what is needed in order to verify~\eqref{eq:stable-like:2}. Indeed: by the Stieltjes representation~\eqref{eq:r:int} of the Rogers function $f(\xi)$, we have
\formula{
 f(\xi) & = - i b \xi + c + \frac{1}{\pi} \int_{(0, \infty)} \biggl(\frac{\xi}{\xi + i s} + \frac{i \xi}{1 + s}\biggr) \frac{\mu_+(ds)}{s} + \frac{1}{\pi} \int_{(0, \infty)} \biggl(\frac{\xi}{\xi - i s} - \frac{i \xi}{1 + s}\biggr) \frac{\mu_-(ds)}{s} \, .
}
Substituting $\xi = r e^{i \thet}$ with $\thet \in (-\tfrac{\pi}{2}, \tfrac{\pi}{2})$, we obtain
\formula{
 f(r e^{i \thet}) = - i b r e^{i \thet} + c & + \frac{1}{\pi} \int_{(0, \infty)} \biggl(\frac{r e^{i \thet - i \pi/2}}{r e^{i \thet - i \pi/2} + s} - \frac{r e^{i \thet - i \pi/2}}{1 + s}\biggr) \frac{\mu_+(ds)}{s} \\
 & \qquad + \frac{1}{\pi} \int_{(0, \infty)} \biggl(\frac{r e^{i \thet + i \pi/2}}{r e^{i \thet + i \pi/2} + s} - \frac{r e^{i \thet + i \pi/2}}{1 + s}\biggr) \frac{\mu_-(ds)}{s} \, ,
}
and so~\eqref{eq:ex:stable-like:aux} shows that
\formula{
 f(r e^{i \thet}) & \sim \Gamma(-\alpha) (-e^{-i \pi \alpha/2} A_+ - e^{i \pi \alpha/2} A_-) (r e^{i \thet})^\alpha \ell(r) && \text{as $r \to \infty$.}
}
In particular, when $\thet = 0$, we obtain the first part of~\eqref{eq:stable-like:2}. Furthermore, for a general $\thet \in (-\tfrac{\pi}{2}, \tfrac{\pi}{2})$, the above asymptotic equality implies that
\formula{
 \lim_{r \to \infty} \Arg f(r e^{i \thet}) & = \Arg(-e^{-i \pi \alpha/2} A_+ - e^{i \pi \alpha/2} A_-) + \alpha \thet \\
 & = \arctan \biggl( \frac{A_+ - A_-}{A_+ + A_-} \tan \frac{\alpha \pi}{2} \biggr) + \alpha \thet .
}
Let us write $\beta = \arctan((A_+ - A_-) (A_+ + A_-)^{-1} \tan \tfrac{\alpha \pi}{2})$. The right-hand side of the above equality is positive for $\thet > \beta / \alpha$ and negative for $\thet < -\beta / \alpha$. Consequently,
\formula{
 \limsup_{r \to \infty} \lvert\Arg \zeta_f(r)\rvert & \le \frac{\beta}{\alpha} \, .
}
Recall that we have
\formula{
 |\beta| & = \biggl| \arctan \biggl( \frac{A_+ - A_-}{A_+ + A_-} \tan \frac{\alpha \pi}{2} \biggr) \biggr| < \frac{(\alpha - 1) \pi}{2} \, ;
}
indeed: this is always true when $\tfrac{3}{2} < \alpha < 2$, and it is our additional assumption when $1 < \alpha \le \tfrac{3}{2}$. We conclude that
\formula{
 \limsup_{r \to \infty} \lvert\Arg \zeta_f(r)\rvert & \le \frac{\beta}{\alpha} < \frac{\alpha - 1}{\alpha} \, \frac{\pi}{2} \, ,
}
as desired: this is the latter part of~\eqref{eq:stable-like:2}.

\subsection{Brownian motion with drift}

Below we consider the characteristic exponent of one of the simplest Lévy processes with completely monotone jumps: the standard Brownian motion with non-negative drift. It has no jumps, and its characteristic exponent is given by
\formula{
 f(\xi) & = \tfrac{1}{2} \xi^2 - i b \xi = \tfrac{1}{2} (\xi - i b)^2 + \tfrac{1}{2} b^2 ,
}
where $b \ge 0$. In this case $f(\xi)$ is a real number if and only if $\re \xi = 0$ or $\im \xi = b$. Therefore, $Z_f = (b, \infty)$, and for $r \in Z_f$ we have
\formula{
 \zeta_f(r) & = \sqrt{r^2 - b^2} + b i && \text{and} & \lambda_f(r) & = \tfrac{1}{2} r^2 .
}
We find that
\formula{
 f(r; \xi) & = \frac{(\xi - \sqrt{r^2 - b^2} - b i) (\xi + \sqrt{r^2 - b^2} - b i)}{\frac{1}{2} \xi^2 - i b \xi - \frac{1}{2} r^2} = 2 .
}
It follows that $f^+(r; \xi) = f^-(r; \xi) = \sqrt{2}$. In particular, with the notation introduced in Definition~\ref{def:eig},
\formula{
 c_{f+}(r) & = -\Arg f^+(r; -i \zeta_f(r)) = 0 , & c_{f-}(r) & = \Arg f^-(r; i \zeta_f(r)) = 0 .
}
Finally,
\formula{
 \laplace G_{f+}(r; \xi) & = -\im \frac{e^{-i c_{f+}(r)}}{\xi + i \zeta_f(r)} - \frac{\re \zeta_f(r)}{|f^+(r; -i \zeta_f(r))|} \, \frac{f^+(r; \xi)}{|\zeta_f(r) - i \xi|^2} \, , \\
 & = -\im \frac{1}{\xi + i \sqrt{r^2 - b^2} - b} - \frac{\sqrt{r^2 - b^2}}{\sqrt{2}} \, \frac{\sqrt{2}}{\lvert\sqrt{r^2 - b^2} + i b - i \xi\rvert^2} = 0 ,
}
so that $G_{f+}(r; y) = 0$, and similarly $G_{f-}(r; x) = 0$. We conclude that
\formula{
 F_{f+}(r; y) & = e^{b y} \sin(\sqrt{r^2 - b^2} \, y) , & F_{f-}(r; x) & = e^{-b x} \sin(\sqrt{r^2 - b^2} \, x) .
}

The expression in Theorem~\ref{thm:heat} reduces to a well-known formula:
\formula{
 p_t^+(x, y) & = \frac{2}{\pi} \int_b^\infty e^{-t r^2 / 2} \times e^{b y} \sin(\sqrt{r^2 - b^2} \, y) \times e^{-b x} \sin(\sqrt{r^2 - b^2} \, x) \times \frac{r}{\sqrt{r^2 - b^2}} \, dr \\
 & = \frac{2}{\pi} \, e^{b (y - x) - t b^2 / 2} \int_0^\infty e^{-t s^2 / 2} \sin(s y) \sin(s x) ds ,
}
and the right-hand side can be evaluated explicitly to give
\formula{
 p_t^+(x, y) & = e^{b (y - x) - t b^2 / 2} \times \frac{1}{\sqrt{2 \pi t}} (e^{-(x - y)^2 / (2 t)} - e^{-(x + y)^2 / (2 t)}) .
}
The result of Theorem~\ref{thm:extrema}\ref{thm:extrema:a} is not new either: using the fact that $\laplace F_{f-}(r; 0^+) = r^{-2} \sqrt{r^2 - b^2}$, we obtain
\formula{
 \pr^0(\overline{X}_t < y) & = \frac{2}{\pi} \int_b^\infty e^{-t r^2 / 2} \times e^{b y} \sin(\sqrt{r^2 - b^2} \, y) \times \frac{\sqrt{r^2 - b^2}}{r^2} \times \frac{r}{\sqrt{r^2 - b^2}} \, dr \\
 & = \frac{2}{\pi} \, e^{b y - t b^2 / 2} \int_0^\infty e^{-t s^2 / 2} \sin(s y) \, \frac{s}{s^2 + b^2} \, ds ;
}
the last integral can be written in terms of the error function. Theorem~\ref{thm:extrema}\ref{thm:extrema:b} does not apply in the present situation, and it can be verified that the equality in~\eqref{eq:extrema:inf} does not hold.

\subsection{Classical risk process: martingale case}

Our next example is the classical risk process, with characteristic exponent
\formula{
 f(\xi) & = i \xi + \frac{\xi}{\xi + i} = \frac{i \xi^2}{\xi + i} \, .
}
This process has positive jumps following the standard exponential distribution, and a unit negative drift. It is easily verified that
\formula{
 \im f(x + i y) & = \frac{x (x^2 + y^2 + 2 y)}{x^2 + (y + 1)^2} \ ,
}
and therefore $\Gamma_f$ is the semi-circle $|\xi + i| = 1$, $\re \xi > 0$. It follows that $Z_f = (0, 2)$, and
\formula{
 \zeta_f(r) & = -i + i e^{-i \alpha}
}
for an appropriate $\alpha \in (0, \pi)$ (depending on $r$). A simple calculation shows that
\formula{
 r^2 & = 2 - 2 \cos \alpha = 4 \sin^2 \tfrac{\alpha}{2} ,
}
that is, $\alpha = 2 \arcsin(\tfrac{1}{2} r)$. Furthermore,
\formula{
 \zeta_f(r) & = \sin \alpha - (1 - \cos \alpha) i = \tfrac{1}{2} r \sqrt{4 - r^2} - \tfrac{1}{2} r^2 i ,
}
and
\formula{
 \lambda_f(r) & = f(\zeta_f(r)) = \frac{i (-i + i e^{-i \alpha})^2}{i e^{-i \alpha}} = 2 - e^{i \alpha} - e^{-i \alpha} = 2 - 2 \cos \alpha = r^2 .
}
By a direct calculation,
\formula{
 f(r; \xi) & = 1 - i \xi ,
}
and therefore
\formula{
 f^+(r; \xi) & = 1 + \xi , & f^-(r; \xi) & = 1 .
}
It follows that
\formula{
 c_{f+}(r) & = -\Arg f^+(r; -i \zeta_f(r)) = -\Arg (1 - i \zeta_f(r)) = -\Arg e^{-i \alpha} = \alpha , \\
 c_{f-}(r) & = \Arg f^-(r; i \zeta_f(r)) = 0 .
}
Additionally, by a short calculation,
\formula{
 \laplace G_{f+}(r; \xi) & = -\im \frac{e^{-i c_{f+}(r)}}{\xi + i \zeta_f(r)} - \frac{\re \zeta_f(r)}{|f^+(r; -i \zeta_f(r))|} \, \frac{f^+(r; \xi)}{|\zeta_f(r) - i \xi|^2} = 0 \, , \\
 & = -\im \frac{e^{-i \alpha}}{\xi + 1 - e^{-i \alpha}} - \frac{\sin \alpha}{1} \, \frac{1 + \xi}{|\xi + 1 - e^{-i \alpha}|^2} = 0 ,
}
and
\formula{
 \laplace G_{f-}(r; \eta) & = \im \frac{e^{i c_{f-}(r)}}{\eta - i \zeta_f(r)} - \frac{\re \zeta_f(r)}{|f^-(r; i \zeta_f(r))|} \, \frac{f^-(r; \eta)}{|\zeta_f(r) + i \eta|^2} \\
 & = \im \frac{1}{\eta - i \zeta_f(r)} - \frac{\re \zeta_f(r)}{1} \, \frac{1}{|\zeta_f(r) + i \eta|^2} = 0 \, ,
}
so that again $G_{f+}(r; y) = G_{f-}(r; x) = 0$. We conclude that
\formula{
 F_{f+}(r; y) & = e^{-(1 - \cos \alpha) y} \sin(y \sin \alpha + \alpha) , \\
 F_{f-}(r; x) & = e^{(1 - \cos \alpha) x} \sin(x \sin \alpha) ,
}
where $\alpha = 2 \arcsin(\tfrac{1}{2} r)$ and $r \in (0, 2)$. Finally, $\cos \alpha = 1 - \tfrac{1}{2} r^2$ and $\sin \alpha = \tfrac{1}{2} r \sqrt{4 - r^2}$, and so we can rewrite the above expressions without using $\alpha$ as
\formula{
 F_{f+}(r; y) & = e^{-y r^2 / 2} \sin(\tfrac{1}{2} r y \sqrt{4 - r^2} + 2 \arcsin(\tfrac{1}{2} r)) , \\
 F_{f-}(r; x) & = e^{x r^2 / 2} \sin(\tfrac{1}{2} r x \sqrt{4 - r^2}) ,
}
with $r \in (0, 2)$.

Theorems~\ref{thm:extrema}\ref{thm:extrema:b} and~\ref{thm:heat} do not apply in the present case. On the other hand, since $\laplace F_{f-}(r; 0^+) = \sqrt{4 - r^2} / (2 r)$, Theorem~\ref{thm:extrema}\ref{thm:extrema:a} implies that
\formula{
 \pr^0(\overline{X}_t < y) & = \frac{2}{\pi} \int_0^2 e^{-r^2 t} \times e^{-y r^2 / 2} \sin(\tfrac{1}{2} r y \sqrt{4 - r^2} + 2 \arcsin(\tfrac{1}{2} r)) \times \\
 & \hspace*{20em} \times \frac{\sqrt{4 - r^2}}{2 r} \times \frac{2}{\sqrt{4 - r^2}} \, dr \\
 & = \frac{2}{\pi} \int_0^2 e^{-r^2 (t + y / 2)} \sin(\tfrac{1}{2} r y \sqrt{4 - r^2} + 2 \arcsin(\tfrac{1}{2} r)) \, \frac{1}{r} \, dr .
}
In terms of variable $\alpha$, we obtain
\formula{
 \pr^0(\overline{X}_t < y) & = \frac{1}{\pi} \int_0^\pi e^{-(1 - \cos \alpha) (2 t + y)} \sin(y \sin \alpha + \alpha) \, \frac{\sin \alpha}{1 - \cos \alpha} \, d\alpha .
}
This formula is due to Asmussen, see Proposition~IV.1.3 in~\cite{aa}, and it provides an efficient way to evaluate the distribution function of time to ruin in the classical risk model.

\subsection{Classical risk process: small drift case}

The same approach applies to the process with exponentially distributed positive jumps and small negative drift:
\formula{
 f(\xi) & = i \xi + R^2 \, \frac{\xi}{\xi + i} = \frac{i \xi^2 + (R^2 - 1) \xi}{\xi + i} \, ,
}
where $R \ge 1$. We omit all details, and only list the final result and some intermediate expressions. The curve $\Gamma_f$ is the semi-circle $|\xi + i| = R$, $\re \xi > 0$, we have $Z_f = (R - 1, R + 1)$, and
\formula{
 \zeta_f(r) & = -i + i R e^{-i \alpha} , & \lambda_f(r) & = 1 + R^2 - 2 R \cos \alpha = r^2,
}
where $\alpha = \arccos((1 + R^2 - r^2) / (2 R)) \in (0, \pi)$. Furthermore,
\formula{
 f(r; \xi) & = 1 - i \xi ,
}
and hence once again
\formula{
 f^+(r; \xi) & = 1 + \xi , & f^-(r; \xi) & = 1	 .
}
Consequently, as in the previous example,
\formula{
 c_{f+}(r) & = \alpha , & c_{f-}(r) & = 0 ,
}
and
\formula{
 \laplace G_{f+}(r; \xi) & = \laplace G_{f-}(r; \eta) = 0 ,
}
so that $G_{f+}(r; y) = G_{f-}(r; x) = 0$. It follows that
\formula{
 F_{f+}(r; y) & = e^{-(1 - R \cos \alpha) y} \sin(R y \sin \alpha + \alpha) , \\
 F_{f-}(r; x) & = e^{(1 - R \cos \alpha) x} \sin(R x \sin \alpha) .
}
The assumptions of Theorems~\ref{thm:extrema}\ref{thm:extrema:b} and~\ref{thm:heat} are not satisfied, and by Theorem~\ref{thm:extrema}\ref{thm:extrema:a} we have
\formula{
 \pr^0(\overline{X}_t < y) & = \frac{2}{\pi} \int_0^\pi e^{-(1 + R^2 - 2 R \cos \alpha) t - (1 - R \cos \alpha) y} \sin(R y \sin \alpha + \alpha) \, \frac{R^2 \sin \alpha}{1 + R^2 - 2 R \cos \alpha} \, d\alpha .
}
Noteworthy, the above formula, although virtually the same as Asmussen's result given in Proposition~IV.1.3 in~\cite{aa}, covers the complementary case: here the drift is assumed to be smaller than in the martingale case (that is, $R \ge 1$), while Asmussen's formula requires it to be larger ($R \le 1$; our parameter $R$ corresponds to $\sqrt{\beta}$ in the notation of Chapter~IV in~\cite{aa}).

\subsection{Brownian motion with exponentially distributed jumps}

Let us now consider the Lévy process with characteristic exponent
\formula{
 f(\xi) & = \tfrac{1}{2} \xi^2 + i \xi + \frac{\xi}{\xi + i} \, .
}
This process is the sum of the standard Brownian motion and an independent classical risk process. By a direct calculation, $Z_f = (0, \infty)$ and
\formula{
 \zeta_f(r) & = \frac{\sqrt{1 - \beta^3}}{\sqrt{\beta}} - (1 - \beta) i , & \lambda_f(r) & = \frac{(1 - \beta) (1 + 2 \beta)^2}{2 \beta} \, ,
}
with $r^2 = (1 + \beta - 2 \beta^2) / \beta$ and
\formula{
 \beta & = \frac{1 - r^2 + \sqrt{(1 - r^2)^2 + 8}}{4} \in (0, 1) .
}
By another direct calculation,
\formula{
 f(r; \xi) & = \frac{2 (\xi + i)}{\xi + (1 + 2 \beta) i} \, ,
}
and thus
\formula{
 f^+(r; \xi) & = \frac{2 (\xi + 1)}{\xi + (1 + 2 \beta)} \, , & f^-(r; \xi) & = 1 .
}
It can be further calculated that, with the notation $\alpha = \re \zeta_f(r) = \sqrt{1 - \beta^3} / \sqrt{\beta}$,
\formula{
 |f^+(r; -i \zeta_f(r))| & = \biggl| \frac{2 (\beta - i \alpha)}{3 \beta - i \alpha} \biggr| = \frac{2 \sqrt{\beta^2 + \alpha^2}}{\sqrt{9 \beta^2 + \alpha^2}} = \frac{2}{\sqrt{1 + 8 \beta^3}}
}
and
\formula{
 c_{f+}(r) & = -\Arg \frac{2 (\beta - i \alpha)}{3 \beta - i \alpha} = -\Arg \frac{2}{1 + 2 \beta^3 + 2 i \beta^{3/2} \sqrt{1 - \beta^3}} = \arctan \frac{2 \beta^{3/2} \sqrt{1 - \beta^3}}{1 + 2 \beta^3} \, .
}
Since $e^{-i c_{f+}(r)} = f^+(r; -i \zeta_f(r)) / |f^+(r; -i \zeta_f(r))|$, after a short calculation we find that
\formula{
 \laplace G_{f+}(r; \xi) & = -\sqrt{1 + 8 \beta^3} \biggl(\im \frac{\beta - i \alpha}{(3 \beta - i \alpha)(\xi + 1 - \beta + i \alpha)} + \frac{\alpha (\xi + 1)}{|\xi + 1 - \beta + i \alpha|^2 (\xi + 1 + 2 \beta)} \biggr) \\
 & = \frac{2 \beta^{3/2} \sqrt{1 - \beta^3}}{\sqrt{1 + 8 \beta^3}} \, \frac{1}{\xi + 1 + 2 \beta} \, .
}
Furthermore, $c_{f-}(r) = 0$ and $\laplace G_{f-}(r; \eta) = 0$, and therefore
\formula{
 G_{f+}(r; y) & = \frac{2 \beta^{3/2} \sqrt{1 - \beta^3}}{\sqrt{1 + 8 \beta^3}} \, e^{-(1 + 2 \beta) y} , \\
 G_{f-}(r; x) & = 0 .
}
Finally,
\formula{
 F_{f+}(r; y) & = e^{-(1 - \beta) y} \sin\biggl(\frac{\sqrt{1 - \beta^3}}{\sqrt{\beta}} \, y + \arctan \frac{2 \beta^{3/2} \sqrt{1 - \beta^3}}{1 + 2 \beta^3}\biggr) - \frac{2 \beta^{3/2} \sqrt{1 - \beta^3}}{\sqrt{1 + 8 \beta^3}} \, e^{-(1 + 2 \beta) y} , \\
 F_{f-}(r; x) & = e^{(1 - \beta) x} \sin\biggl(\frac{\sqrt{1 - \beta^3}}{\sqrt{\beta}} \, x\biggr) .
}
All our main results apply, and so $p_t^+(x, y)$, $\pr^0(\overline{X}_t < y)$ and $\pr^0(\underline{X}_t > -y)$ can be written as explicit integrals. It seems more convenient to write the integrals in terms of variable $\beta \in (0, 1)$, using
\formula{
 \laplace F_{f+}(r; 0^+) & = \frac{\alpha \sqrt{1 + 8 \beta^3}}{2} \, \frac{2}{(1 + 2 \beta) r^2} = \frac{\sqrt{\beta} \sqrt{1 - \beta^3} \sqrt{1 + 8 \beta^3}}{(1 + 2 \beta) (1 + \beta - 2 \beta^2)} \, , \\
 \laplace F_{f-}(r; 0^+) & = \frac{\alpha}{r^2} = \frac{\sqrt{\beta} \sqrt{1 - \beta^3}}{1 + \beta - 2 \beta^2} \, , \\
 \lvert\zeta_f'(r)\rvert dr & = \sqrt{1 + \biggl(\frac{d}{d \beta} \frac{\sqrt{1 - \beta^3}}{\sqrt{\beta}}\biggr)^2} d\beta = \frac{\sqrt{1 + 8 \beta^3}}{2 \beta^{3/2} \sqrt{1 - \beta^3}} \, d\beta
}
The formulae are nevertheless complicated. For example, by Theorem~\ref{thm:extrema}\ref{thm:extrema:b}, we have
\formula{
 \pr^0(\underline{X}_t > -x) & = \frac{2}{\pi} \int_0^1 e^{-t (1 - \beta) (1 + 2 \beta)^2 / (2 \beta)} \times \frac{\sqrt{\beta} \sqrt{1 - \beta^3} \sqrt{1 + 8 \beta^3}}{(1 + 2 \beta) (1 + \beta - 2 \beta^2)} \\
 & \hspace*{10em} \times e^{(1 - \beta) x} \sin\biggl(\frac{\sqrt{1 - \beta^3}}{\sqrt{\beta}} \, x\biggr) \times \frac{\sqrt{1 + 8 \beta^3}}{2 \beta^{3/2} \sqrt{1 - \beta^3}} \, d\beta \\
 & = \frac{2}{\pi} \int_0^1 e^{-t (1 - \beta) (1 + 2 \beta)^2 / (2 \beta)} e^{(1 - \beta) x} \sin\biggl(\frac{\sqrt{1 - \beta^3}}{\sqrt{\beta}} \, x\biggr) \frac{1 - 2 \beta + 4 \beta^2}{2 \beta (1 + \beta - 2 \beta^2)} \, d\beta .
}

%
%

%
%

\end{document}